\documentclass{amsart}
\usepackage[utf8]{inputenc}

\usepackage{amsfonts, amsthm, amsmath, enumerate, amssymb,mathrsfs, tikz, tikz-cd, ifthen, color, chngcntr, mathtools, multicol,array, bbm, euscript}
\usepackage[margin=1.25in]{geometry}

\usetikzlibrary{arrows}
\usetikzlibrary{decorations.markings}

\newtheorem{theorem}[subsection]{Theorem}

\newtheorem{lemma}[subsection]{Lemma}
\newtheorem{remark}[subsection]{Remark}
\newtheorem{corollary}[subsection]{Corollary}
\theoremstyle{definition}
\newtheorem{definition}[subsection]{Definition}
\newtheorem{example}[subsection]{Example}
\newtheorem{notation}[subsection]{Notation}

\definecolor{light-gray}{gray}{.65}

\newcommand{\M}{\mathbb{M}}
\newcommand{\R}{\mathbb{R}}
\newcommand{\Z}{\mathbb{Z}}

\newcommand{\Hom}{\operatorname{Hom}}
\newcommand{\resettheoremcounters}{\setcounter{class}{0}}

\newcommand{\Hex}{\operatorname{Poly}}
\newcommand{\Sph}{\operatorname{Sph}}
\newcommand{\ebc}{\mathbb{EB}}
\newcommand{\coker}{\operatorname{coker}}
\newcommand{\inv}{^{-1}}

\newcommand{\cone}[3]{
	\draw[thick, #3] (#1+1/2,5) -- (#1+1/2,#2+1/2) -- (10-#2-#2+#1-1/2,5);
	\draw[thick, #3] (#1+1/2,-5) -- (#1 +1/2, #2 -0.5) -- (-8-1/2-#2-#2+#1,-5)
}

\newcommand{\scone}[3]{
	\draw[thick, #3] (#1+1/2,5) -- (#1+1/2,#2+1/2) -- (5,{(5-#1)/2+#2+1/4});
	\draw[thick, #3] (#1+1/2,-5) -- (#1 +1/2, #2 -0.5) -- (-3,{(-3-#1)/2-3/4+#2})
}

\newcommand{\eb}[3]{
	\draw[thick, #3] (#1+1/2,5) -- (#1+1/2,#2+1/2) -- (#1+1+1/2,#2+1/2) -- (5,{(5-#1)/2+#2-1/4});
	\draw[thick, #3] (#1+1/2,-5) -- (#1+1/2,#2-1.5) -- (#1-0.5,#2-1.5) -- (-3,{(-3-#1)/2+#2-5/4});
}

\newcommand{\cthree}[3]{
	\draw[thick, #3] (#1+1/2, -5) -- (#1+1/2, 5);
	\draw[thick, #3] (#1+#2+1/2, -5) -- (#1+#2+1/2, 5);
	{\ifthenelse{#2>1}{
	\foreach \y in {-5,...,2} {
	\draw[thick, #3] (#1+1/2, \y +1/2) -- (#1+#2+1/2, \y + #2+1/2);
	}
	}{
	\foreach \y in {-5,-4,-3,-2,-1,0,1,2,3} {\draw[thick, #3] (#1+1/2, \y +1/2) -- (#1+#2+1/2, \y + #2+1/2);}}
	};
}

\newcommand{\lab}[3]{
\draw[#3] (#1+1/2,5.5) node{\tiny{$#2$}}
}

\title{The $RO(C_3)$-graded Bredon cohomology of $C_3$-surfaces in $\underline{\Z/3}$-coefficients}
\author{Kelly Pohland}
\date{May 2023}

\begin{document}

\maketitle

\begin{abstract}
    All closed surfaces with a $C_p$-action where $p$ is an odd prime were classified in \cite{Pohl} using equivariant surgery methods. Using this classification in the case $p=3$, we compute the $RO(C_3)$-graded Bredon cohomology of all $C_3$-surfaces in constant $\Z/3$-coefficients as modules over the cohomology of a fixed point. We show that the cohomology of a given $C_3$-surface is determined by a handful of topological invariants and is directly determined by the construction of the surface via equivariant surgery.
\end{abstract}

\section{Introduction}

For a space $X$ with an action of a finite group $G$, the $RO(G)$-graded Bredon cohomology of $X$ is a sequence of abelian groups, graded on the Grothendieck group of real, finite-dimensional, orthogonal $G$-representations. Represented by a genuine equivariant Eilenberg-MacLane spectrum, this ordinary cohomology theory provides a direct analogue for singular cohomology in the equivariant setting. Increased interest in equivariant homotopy theory has led to greater efforts to understand the properties of $RO(G)$-graded Bredon cohomology. In particular, it has inspired many recent Bredon cohomology computations \cite{CHT21,Dug15,LFdS09,Haz19b,Haz19a,Hog18,May18,San21,Wil19}. 

Computations in Bredon cohomology often prove quite complicated despite their fundamental role in equivariant homotopy theory. As a consequence, most results focus on the case where the group action is by the cyclic group of order $2$. Our goal for this paper is to present a complete family of computations in $RO(C_3)$-graded Bredon cohomology, where $C_3$ is the cyclic group of order $3$. In particular, we will be computing the cohomology of all closed, connected $2$-manifolds with a nontrivial action of $C_3$ in $\underline{\Z/3}$-coefficients. The work in this paper uses similar computational methods as those in \cite{Haz19b} and serves as an analogue to her result at the prime $3$. 

A key ingredient in our computation is a recent equivariant surgery classification of $C_3$-surfaces
\cite{Pohl}. This classification provides blueprints for building $C_3$-surfaces using a handful of surgery methods and informs the construction of equivariant cofiber sequences. These tools allow us to present the cohomology of all $C_3$-surfaces in two ways. We first provide the answer in terms of their construction as presented in \cite{Pohl}. We are then able to provide explicit formulas for the cohomology which depend on a handful of numerical invariants for the surface. 

In order to state the main result, let us begin with some background on $RO(C_3)$-graded Bredon cohomology. Up to isomorphism, there are two irreducible real representations of $C_3$, namely the trivial representation ($\R_{\text{triv}}$) and the two-dimensional representation given by rotation of $120^\circ$ about the origin ($\R^2_{\text{rot}}$). So any element of $RO(C_3)$ can be represented as $\R_{\text{triv}}^{\oplus p-2q}\oplus \left(R^2_{\text{rot}}\right)^{\oplus q}$ and is completely determined by the values $p$ and $q$. As a result, $RO(C_3)$-graded Bredon cohomology can be viewed as a bigraded theory, with the cohomology of a $C_3$-space $X$ with coefficients in the Mackey functor $M$ denoted $H^{p,q}(X;M)$. Note that under this convention, our first grading $p$ represents the total topological dimension of our representation, and $q$ represents the number of copies of $\R^2_{\text{rot}}$.

Define $\M_3$ to be the $RO(C_3)$-graded Bredon cohomology of a fixed point in $\underline{\Z/3}$ coefficients. In this paper, we compute the cohomology of all closed, connected, non-trivial $C_3$-surfaces as $\M_3$-modules. It turns out there are only a few $\M_3$-modules which show up in the cohomology of $C_3$-surfaces. These modules are $\M_3$, the cohomology of the freely rotating circle ($S^1_{\text{free}}$), the cohomology of $C_3$, and a module called $\ebc$ which denotes the reduced cohomology of the unreduced suspension of $C_3$. Since our Bredon theory is bigraded, we can depict each of these modules in the $(p,q)$-plane, where the $(p,q)$th cohomology group is depicted above and to the right of the $(p,q)$th spot on the grid. Figures \ref{modules1} and \ref{modules2} give depictions of these $\M_3$-modules in the $(p,q)$-plane. Each dot in these figures represents a copy of $\Z/3$. 

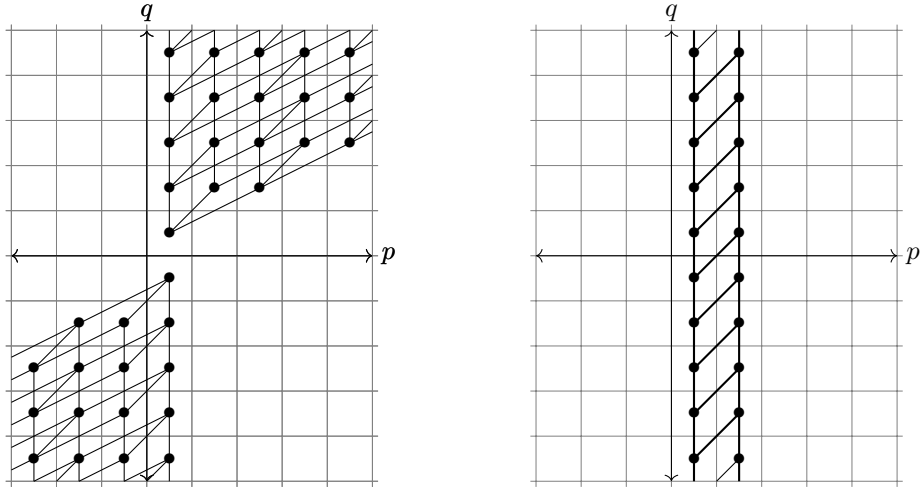
\begin{figure}
\begin{minipage}{0.45\textwidth}
\begin{center}
\begin{tikzpicture}[scale=0.6]
	\draw[help lines] (-3.125,-5.125) grid (5.125, 5.125);
	\draw[<->] (-3,0)--(5,0)node[right]{$p$};
	\draw[<->] (0,-5)--(0,5)node[above]{$q$};;
	\draw[help lines] (-3.125,-5.125) grid (5.125, 5.125);
		\draw[<->] (-3,0)--(5,0)node[right]{$p$};
		\draw[<->] (0,-5)--(0,5)node[above]{$q$};;
		\draw (.5,-.5) node{$\bullet$};
		\draw (.5,-1.5) node{$\bullet$};
		\draw (.5,-2.5) node{$\bullet$};
		\draw (.5,-3.5) node{$\bullet$};
		\draw (.5,-4.5) node{$\bullet$};
		\draw (1.5,2.5) node{$\bullet$};
		\draw (1.5,3.5) node{$\bullet$};
		\draw (1.5,4.5) node{$\bullet$};
		\draw (2.5,2.5) node{$\bullet$};
		\draw (2.5,3.5) node{$\bullet$};
		\draw (2.5,4.5) node{$\bullet$};
		\draw (3.5,2.5) node{$\bullet$};
		\draw (3.5,3.5) node{$\bullet$};
		\draw (3.5,4.5) node{$\bullet$};
		\draw (4.5,2.5) node{$\bullet$};
		\draw (4.5,3.5) node{$\bullet$};
		\draw (4.5,4.5) node{$\bullet$};
		\draw (-.5,-1.5) node{$\bullet$};
		\draw (-.5,-2.5) node{$\bullet$};
		\draw (-.5,-3.5) node{$\bullet$};
		\draw (-.5,-4.5) node{$\bullet$};
		\draw (-1.5,-1.5) node{$\bullet$};
		\draw (-1.5,-2.5) node{$\bullet$};
		\draw (-1.5,-3.5) node{$\bullet$};
		\draw (-1.5,-4.5) node{$\bullet$};
		\draw (-2.5,-2.5) node{$\bullet$};
		\draw (-2.5,-3.5) node{$\bullet$};
		\draw (-2.5,-4.5) node{$\bullet$};
		\draw (1.5,1.5) -- (1.5,5);
		\draw (2.5,1.5) -- (2.5,5);
		\draw (3.5,2.5) -- (3.5,5);
		\draw (4.5,2.5) -- (4.5,5);
		\draw (.5,-.5) -- (.5,-5);
		\draw (-.5,-1.5) -- (-.5,-5);
		\draw (-1.5,-1.5) -- (-1.5,-5);
		\draw (-2.5,-2.5) -- (-2.5,-5);
		\draw (2.5,1.5) -- (3.5,2.5);
		\draw (2.5,2.5) -- (3.5,3.5);
		\draw (2.5,3.5) -- (3.5,4.5);
		\draw (2.5,4.5) -- (3,5);
		\draw (4.5,2.5) -- (5,3);
		\draw (4.5,3.5) -- (5,4);
		\draw (4.5,4.5) -- (5,5);
		\draw (-.5,-1.5) -- (.5,-.5);
		\draw (-.5,-2.5) -- (.5,-1.5);
		\draw (-.5,-3.5) -- (.5,-2.5);
		\draw (-.5,-4.5) -- (.5,-3.5);
		\draw (0,-5) -- (.5,-4.5);
		\draw (-2.5,-2.5) -- (-1.5,-1.5);
		\draw (-2.5,-3.5) -- (-1.5,-2.5);
		\draw (-2.5,-4.5) -- (-1.5,-3.5);
		\draw (-2,-5) -- (-1.5,-4.5);
		\draw (2.5,1.5) -- (5,2.75);
		\draw (.5,1.5) -- (5,3.75);
		\draw (.5,2.5) -- (5,4.75);
		\draw (.5,3.5) -- (3.5,5);
		\draw (.5,4.5) -- (1.5,5);
		\draw (1.5,1.5) -- (5,3.25);
		\draw (1.5,2.5) -- (5,4.25);
		\draw (1.5,3.5) -- (4.5,5);
		\draw (1.5,4.5) -- (2.5,5);
		\draw (.5,-.5) -- (-3,-2.25);
		\draw (.5,-1.5) -- (-3,-3.25);
		\draw (.5,-2.5) -- (-3,-4.25);
		\draw (.5,-3.5) -- (-2.5,-5);
		\draw (.5,-4.5) -- (-.5,-5);
		\draw (-.5,-1.5) -- (-3,-2.75);
		\draw (-.5,-2.5) -- (-3,-3.75);
		\draw (-.5,-3.5) -- (-3,-4.75);
		\draw (-.5,-4.5) -- (-1.5,-5);
		\draw (2.5,1.5) node{$\bullet$};
		\draw (1.5,1.5) node{$\bullet$};
		\draw (.5,.5) -- (.5,5);
		\draw (.5,.5) -- (2.5,1.5);
		\draw (.5,.5) -- (1.5,1.5);
		\draw (.5,1.5) -- (1.5,2.5);
		\draw (.5,2.5) -- (1.5,3.5);
		\draw (.5,3.5) -- (1.5,4.5);
		\draw (.5,4.5) -- (1,5);
		\draw (.5,.5) node{$\bullet$};
		\draw (.5,1.5) node{$\bullet$};
		\draw (.5,2.5) node{$\bullet$};
		\draw (.5,3.5) node{$\bullet$};
		\draw (.5,4.5) node{$\bullet$};
\end{tikzpicture}
	\end{center}
\end{minipage}\ \begin{minipage}{0.45\textwidth}
\begin{center}
	\begin{tikzpicture}[scale=0.6]
		\draw[help lines] (-3.125,-5.125) grid (5.125, 5.125);
		\draw[<->] (-3,0)--(5,0)node[right]{$p$};
		\draw[<->] (0,-5)--(0,5)node[above]{$q$};;
		\cthree{0}{1}{black};
		\draw (.5,.5) node{$\bullet$};
		\draw (.5,1.5) node{$\bullet$};
		\draw (.5,2.5) node{$\bullet$};
		\draw (.5,3.5) node{$\bullet$};
		\draw (.5,4.5) node{$\bullet$};
		\draw (.5,-.5) node{$\bullet$};
		\draw (.5,-1.5) node{$\bullet$};
		\draw (.5,-2.5) node{$\bullet$};
		\draw (.5,-3.5) node{$\bullet$};
		\draw (.5,-4.5) node{$\bullet$};
		\draw (1.5,.5) node{$\bullet$};
		\draw (1.5,1.5) node{$\bullet$};
		\draw (1.5,2.5) node{$\bullet$};
		\draw (1.5,3.5) node{$\bullet$};
		\draw (1.5,4.5) node{$\bullet$};
		\draw (1.5,-.5) node{$\bullet$};
		\draw (1.5,-1.5) node{$\bullet$};
		\draw (1.5,-2.5) node{$\bullet$};
		\draw (1.5,-3.5) node{$\bullet$};
		\draw (1.5,-4.5) node{$\bullet$};
		\draw (.5,4.5) -- (1,5);
		\draw (1,-5) -- (1.5,-4.5);
	\end{tikzpicture}
	\end{center}
\end{minipage}
\caption{\label{modules1} The $\M_3$-modules $\M_3$ (left) and $H^{*,*}(S^1_{\text{free}})$ (right).}
\end{figure}

\begin{figure}
\begin{minipage}{0.45\textwidth}
\begin{center}
\begin{tikzpicture}[scale=0.6]
	\draw[help lines] (-3.125,-5.125) grid (5.125, 5.125);
	\draw[<->] (-3,0)--(5,0)node[right]{$p$};
	\draw[<->] (0,-5)--(0,5)node[above]{$q$};;
	\cthree{0}{0}{black};
	\draw (0.5,0.5) node{$\bullet$};
	\draw (0.5,1.5) node{$\bullet$};
	\draw (0.5,2.5) node{$\bullet$};
	\draw (0.5,3.5) node{$\bullet$};
	\draw (0.5,4.5) node{$\bullet$};
	\draw (0.5,-0.5) node{$\bullet$};
	\draw (0.5,-1.5) node{$\bullet$};
	\draw (0.5,-2.5) node{$\bullet$};
	\draw (0.5,-3.5) node{$\bullet$};
	\draw (0.5,-4.5) node{$\bullet$};
\end{tikzpicture}
	\end{center}
\end{minipage}\ \begin{minipage}{0.45\textwidth}
\begin{center}
	\begin{tikzpicture}[scale=0.6]
		\draw[help lines] (-3.125,-5.125) grid (5.125, 5.125);
		\draw[<->] (-3,0)--(5,0)node[right]{$p$};
		\draw[<->] (0,-5)--(0,5)node[above]{$q$};;
		\draw (.5,-.5) node{$\bullet$};
		\draw (.5,-1.5) node{$\bullet$};
		\draw (.5,-2.5) node{$\bullet$};
		\draw (.5,-3.5) node{$\bullet$};
		\draw (.5,-4.5) node{$\bullet$};
		\draw (1.5,-.5) node{$\bullet$};
		\draw (1.5,-1.5) node{$\bullet$};
		\draw (1.5,-2.5) node{$\bullet$};
		\draw (1.5,-3.5) node{$\bullet$};
		\draw (1.5,-4.5) node{$\bullet$};
		\draw (1.5,2.5) node{$\bullet$};
		\draw (1.5,3.5) node{$\bullet$};
		\draw (1.5,4.5) node{$\bullet$};
		\draw (2.5,2.5) node{$\bullet$};
		\draw (2.5,3.5) node{$\bullet$};
		\draw (2.5,4.5) node{$\bullet$};
		\draw (3.5,2.5) node{$\bullet$};
		\draw (3.5,3.5) node{$\bullet$};
		\draw (3.5,4.5) node{$\bullet$};
		\draw (4.5,2.5) node{$\bullet$};
		\draw (4.5,3.5) node{$\bullet$};
		\draw (4.5,4.5) node{$\bullet$};
		\draw (-.5,-1.5) node{$\bullet$};
		\draw (-.5,-2.5) node{$\bullet$};
		\draw (-.5,-3.5) node{$\bullet$};
		\draw (-.5,-4.5) node{$\bullet$};
		\draw (-1.5,-1.5) node{$\bullet$};
		\draw (-1.5,-2.5) node{$\bullet$};
		\draw (-1.5,-3.5) node{$\bullet$};
		\draw (-1.5,-4.5) node{$\bullet$};
		\draw (-2.5,-2.5) node{$\bullet$};
		\draw (-2.5,-3.5) node{$\bullet$};
		\draw (-2.5,-4.5) node{$\bullet$};
		\draw (1.5,1.5) -- (1.5,5);
		\draw (2.5,1.5) -- (2.5,5);
		\draw (3.5,2.5) -- (3.5,5);
		\draw (4.5,2.5) -- (4.5,5);
		\draw (1.5,-.5) -- (1.5,-5);
		\draw (.5,-.5) -- (.5,-5);
		\draw (-.5,-1.5) -- (-.5,-5);
		\draw (-1.5,-1.5) -- (-1.5,-5);
		\draw (-2.5,-2.5) -- (-2.5,-5);
		\draw (2.5,1.5) -- (3.5,2.5);
		\draw (2.5,2.5) -- (3.5,3.5);
		\draw (2.5,3.5) -- (3.5,4.5);
		\draw (2.5,4.5) -- (3,5);
		\draw (4.5,2.5) -- (5,3);
		\draw (4.5,3.5) -- (5,4);
		\draw (4.5,4.5) -- (5,5);
		\draw (-.5,-1.5) -- (.5,-.5);
		\draw (-.5,-2.5) -- (.5,-1.5);
		\draw (-.5,-3.5) -- (.5,-2.5);
		\draw (-.5,-4.5) -- (.5,-3.5);
		\draw (0,-5) -- (.5,-4.5);
		\draw (-2.5,-2.5) -- (-1.5,-1.5);
		\draw (-2.5,-3.5) -- (-1.5,-2.5);
		\draw (-2.5,-4.5) -- (-1.5,-3.5);
		\draw (-2,-5) -- (-1.5,-4.5);
		\draw (2.5,1.5) -- (5,2.75);
		\draw (2.5,2.5) -- (5,3.75);
		\draw (2.5,3.5) -- (5,4.75);
		\draw (2.5,4.5) -- (3.5,5);
		\draw (1.5,1.5) -- (5,3.25);
		\draw (1.5,2.5) -- (5,4.25);
		\draw (1.5,3.5) -- (4.5,5);
		\draw (1.5,4.5) -- (2.5,5);
		\draw (.5,-.5) -- (-3,-2.25);
		\draw (.5,-1.5) -- (-3,-3.25);
		\draw (.5,-2.5) -- (-3,-4.25);
		\draw (.5,-3.5) -- (-2.5,-5);
		\draw (.5,-4.5) -- (-.5,-5);
		\draw (1.5,-.5) -- (-3,-2.75);
		\draw (1.5,-1.5) -- (-3,-3.75);
		\draw (1.5,-2.5) -- (-3,-4.75);
		\draw (1.5,-3.5) -- (-1.5,-5);
		\draw (1.5,-4.5) -- (.5,-5);
		\draw (2.5,1.5) node{$\bullet$};
		\draw (1.5,1.5) node{$\bullet$};
	\end{tikzpicture}
	\end{center}
\end{minipage}
\caption{\label{modules2} The $\M_3$-modules $H^{*,*}(C_3)$ (left) and $\ebc$ (right).}
\end{figure}
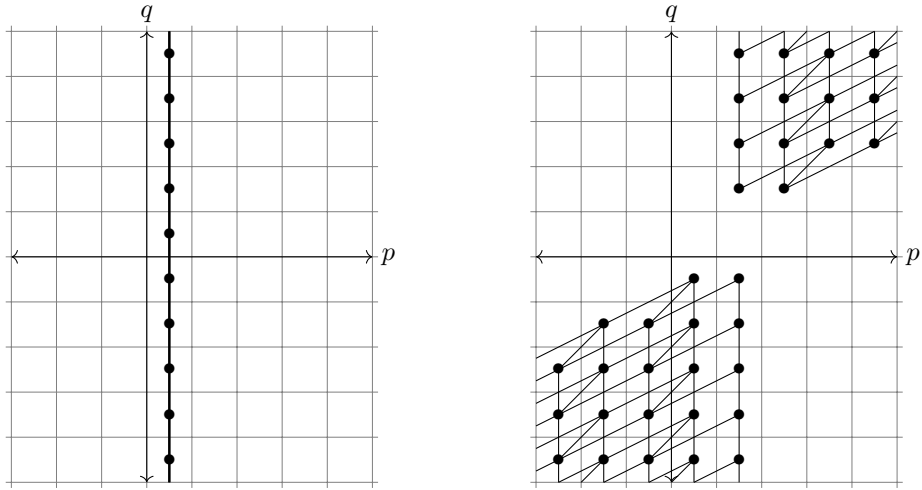

The $\M_3$-module structure of these important pieces are discussed more thoroughly in Section \ref{preliminaries}. For now, we introduce some notation and preview the main result on the cohomology of $C_3$-surfaces. Let $F(X)$ denote the number of fixed points of a given $C_3$-surface $X$. It is useful to note that when the action is non-trivial, $F(X)$ must be finite. We also let $\beta(X)$ denote the \textbf{$\beta$-genus} of $X$, defined to be $\operatorname{dim}_{\Z/2} H^1_{\text{sing}}(X;\Z/2)$. 

\begin{theorem}
Let $X$ be a free $C_3$-surface. 
\begin{enumerate}
\item If $X$ is orientable, then
\[H^{*,*}(X)\cong H^{*,*}(S^1_{\textup{free}})\oplus \Sigma^{1,0}H^{*,*}(S^1_{\textup{free}})\oplus \left(\Sigma^{1,0}H^{*,*}(C_3)\right)^{\oplus \frac{\beta(X)-2}{3}}.\]
\item If $X$ is non-orientable, then
\[H^{*,*}(X)\cong H^{*,*}(S^1_{\textup{free}})\oplus \left(\Sigma^{1,0}H^{*,*}(C_3)\right)^{\oplus \frac{\beta(X)-2}{3}}.\]
\end{enumerate}
\end{theorem}

\begin{theorem}\label{introthm2}
Let $X$ be a $C_3$-surface with $F(X)>0$.
\begin{enumerate}
\item If $X$ is orientable, then 
\[H^{*,*}(X)\cong \M_3\oplus \Sigma^{2,1}\M_3 \oplus \ebc^{\oplus F(X)-2}\oplus \left(\Sigma^{1,0}H^{*,*}(C_3)\right)^{\oplus \frac{\beta(X)-2F(X)+4}{3}}\]
\item If $X$ is non-orientable and $F(X)$ is even, then
\[H^{*,*}(X)\cong \M_3\oplus \ebc^{\oplus F(X)-2}\oplus \left(\Sigma^{1,0}H^{*,*}(C_3)\right)^{\oplus \frac{\beta(X)-2F(X)+1}{3}}\]
\item If $X$ is non-orientable and $F(X)$ is odd, then
\[H^{*,*}(X)\cong \M_3\oplus \ebc^{\oplus F(X)-1}\oplus\left(\Sigma^{1,0}H^{*,*}(C_3)\right)^{\oplus \frac{\beta(X)-2F(X)+1}{3}}\]
\end{enumerate}
\end{theorem}

We can quickly observe from these results that given any $C_3$-surface $X$, the Bredon cohomology of $X$ is completely determined by $\beta(X)$, $F(X)$, and whether or not $X$ is orientable. It is important to note however that Bredon cohomology does not provide a complete invariant for $C_3$-surfaces. For example, there exist spaces in the second and third groups of Theorem \ref{introthm2} whose cohomology are the same. There are nonisomorphic orientable surfaces with the same cohomology as well.

There is a potential concern that $\frac{\beta(X)-2F(X)+4}{3}$ and $\frac{\beta(X)-2F(X)+1}{3}$ may not be integers. However, for any space $X$ with a $C_3$-action, it must be that $F(X)\equiv 2-\beta(X)\pmod{3}$, so this is not an issue. Consequently, when the action on $X$ is free, it must be that $\frac{\beta(X)-2}{3}\in\Z$. A proof of these facts can be found in Chapter 4 of \cite{Pohl}. 

\subsection{Organization of the Paper}
We start with a discussion of important properties and computational tools for Bredon cohomology in Section \ref{preliminaries}. Section \ref{surfacesection} contains a summary of the classification result in \cite{Pohl}. Computations for the Bredon cohomology of all surfaces with free $C_3$-action appear in Section \ref{computationsI}, followed by cohomology computations for spaces with nonfree action in Section \ref{computationsII}.

\subsection{Acknowledgements}
The work in this paper was a portion of the author's thesis project at the University of Oregon. The author would first like to thank her doctoral advisor Dan Dugger for his invaluable guidance and support. The author would also like to thank Chrsity Hazel and Clover May for countless helpful conversations.  %

\section{Premilinaries on $RO(C_3)$-graded Bredon Chomology}\label{preliminaries}

In this section we discuss background knowledge and computational tools for $RO(G)$-graded Bredon Cohomology in the case $G=C_3$. This theory takes coefficients in a Mackey functor, so we begin with a discussion of Mackey functors and define the specific Mackey functor which will be used throughout the paper. We next discuss notation and terminology related to this theory and introduce several computational tools which will be used throughout this paper. This section ends with a few small computations which utilize these tools and introduce some of the key methods used in later computations.

\begin{definition}
A \textbf{Mackey Functor} $M$ for $G=C_3$ is the data of
\begin{center}
\tikzset{every picture/.style={line width=0.75pt}} 
\begin{tikzpicture}[x=0.75pt,y=0.75pt,yscale=-1,xscale=1]

\draw    (205.8,155.2) .. controls (175.26,195.58) and (154.43,104.02) .. (203.51,129.93) ;
\draw [shift={(205.8,131.2)}, rotate = 209.98] [fill={rgb, 255:red, 0; green, 0; blue, 0 }  ][line width=0.08]  [draw opacity=0] (10.72,-5.15) -- (0,0) -- (10.72,5.15) -- (7.12,0) -- cycle    ;
\draw    (257.8,128.2) .. controls (282.63,111.96) and (301.08,116.71) .. (317.49,126.75) ;
\draw [shift={(319.8,128.2)}, rotate = 212.91] [fill={rgb, 255:red, 0; green, 0; blue, 0 }  ][line width=0.08]  [draw opacity=0] (10.72,-5.15) -- (0,0) -- (10.72,5.15) -- (7.12,0) -- cycle    ;
\draw    (261.5,157.3) .. controls (286.06,175.7) and (303.43,168.71) .. (320.8,155.2) ;
\draw [shift={(258.8,155.2)}, rotate = 38.93] [fill={rgb, 255:red, 0; green, 0; blue, 0 }  ][line width=0.08]  [draw opacity=0] (10.72,-5.15) -- (0,0) -- (10.72,5.15) -- (7.12,0) -- cycle    ;

\draw (210,132) node [anchor=north west][inner sep=0.75pt]   [align=left] {$\displaystyle M( C_{3})$};
\draw (321,132) node [anchor=north west][inner sep=0.75pt]   [align=left] {$\displaystyle M( *)$};
\draw (282,95) node [anchor=north west][inner sep=0.75pt]   [align=left] {$\displaystyle p_{*}$};
\draw (284,172) node [anchor=north west][inner sep=0.75pt]   [align=left] {$\displaystyle p^{*}$};
\draw (138,121) node [anchor=north west][inner sep=0.75pt]   [align=left] {$\displaystyle t_{*} ,t^{*}$};
\end{tikzpicture}
\end{center}
where $M(C_3)$ and $M(*)$ are abelian groups, and $p^*$, $p_*$, $t^*$, and $t_*$ are homomorphisms that satisfy
\begin{enumerate}[i.]
\item $(t^*)^3=id$
\item $(t_*)^3=id$
\item $t^*p^*=p^*$
\item $p_*t_*=p_*$
\item $t_*t^*=id$
\item $p^*p_*=1+t^*+(t^*)^2$.
\end{enumerate}
\end{definition}

In this paper we will be primarily focused on \textbf{the constant $\mathbb{Z}/3$ Mackey functor}, which is denoted $\underline{\mathbb{Z}/3}$ and is defined by $M(C_3)=M(*)=\mathbb{Z}/3$, $p^*=t^*=t_*=id$, and $p_*=0$. 

\subsection{Bigraded Theory}

For a group $G$, the $RO(G)$-graded Bredon cohomology of a space with a $G$-action is graded on the Grothendieck group of real, orthogonal, finite-dimensional $G$-representations. In the case $G=C_3$, there are only two such irreducible $G$-representations up to isomorphism. These are the $1$-dimensional trivial representation $\mathbb{R}_{\text{triv}}$, and the $2$-dimensional representation given by rotation of the plane about the origin by $120^\circ$. We denote this representation by $\mathbb{R}^2_{\text{rot}}$.

Given a $C_3$-representation $V$, we can write $V=\left(\mathbb{R}_{\text{triv}}\right)^{\oplus p-2q}\oplus\left(\mathbb{R}^2_{\text{rot}}\right)^{\oplus q}$ where $p$ represents the total dimension of $V$ and $q$ represents the number of copies of $\mathbb{R}^2_{\text{rot}}$ in $V$. Notice that $V$ is completely determined by the values of $p$ and $q$, so $RO(C_3)$ is a rank $2$ free abelian group. In particular, we can write $H^{p,q}_{C_3}(X;M)$ to denote the $V$th cohomology group of $X$ in this theory. Note that the subscript of $C_3$ will be omitted when the context of $G=C_3$ is understood. We also let $\mathbb{R}^{p,q}$ denote the $C_3$-representation $\left(\mathbb{R}_{\text{triv}}\right)^{\oplus p-2q}\oplus\left(\mathbb{R}^2_{\text{rot}}\right)^{\oplus q}$ and the element $(p-2q)[\mathbb{R}_{\text{triv}}]+q[\mathbb{R}^2_{\text{rot}}]$ of $RO(C_3)$. 

Let $V$ be a real $G$-representation, and consider the space $\hat{V}$ obtained by one-point compactifying $V$ by adding a fixed point at infinity. The space $\hat{V}$ is equivalent to a sphere with a $G$-action. We call this a \textbf{representation sphere} and denote it by $S^V$.

We can then form the equivariant suspension 
\[\Sigma^V := S^V\wedge X\]
where $X$ is a $G$-space with a fixed base point. If $X$ is a free $G$-space, we can add a fixed base point to form the space $X_+:= X\sqcup\{*\}$. In general, the notation $X_+$ represents a $G$-space $X$ with a disjoint base point which is fixed under the action of $G$. 

For every finite-dimensional, real, orthogonal $G$-representation $V$, we get natural isomorphisms
\[\Sigma^V\colon \tilde{H}^\alpha_G(-;M)\rightarrow \tilde{H}_G^{\alpha+V}\left(\Sigma^V(-);M\right)\]
where coefficients are taken in the Mackey functor $M$. Given a cofiber sequence of based $G$-spaces
\[X\xrightarrow{f} Y \rightarrow C(f)\]
we get a Puppe sequence
\[X\rightarrow Y \rightarrow C(f) \rightarrow \Sigma^{\textbf{1}}X\rightarrow \Sigma^{\textbf{1}}Y \rightarrow \Sigma^{\textbf{1}}C(f) \rightarrow\cdots\]
where \textbf{1} represents the $1$-dimensional trivial representation of $G$. We can then use the suspension isomorphism to get a long exact sequence
\[\tilde{H}_G^V(X;M)\leftarrow \tilde{H}_G^V(Y;M)\leftarrow \tilde{H}_G^V(C(f);M)\leftarrow\tilde{H}_G^{V-\textbf{1}}(X;M) \leftarrow \tilde{H}_G^{V-\textbf{1}}(Y;M)\leftarrow\cdots\]
for each $V\in RO(G)$. 

In the case $G=C_3$, we know $V\cong \mathbb{R}^{p,q}$ for some $p$ and $q$. For brevity, we use $S^{p,q}$ to denote the representation sphere $S^{\mathbb{R}^{p,q}}$. Examples of representation spheres in this case can be found in Figure \ref{repspheres}. We use blue to denote points which are fixed under the action. We additionally use $\Sigma^{p,q}X$ to denote the $V$th suspension of a $C_3$-space $X$. This means there are isomorphisms
\[\Sigma^{p,q}\colon \tilde{H}^{a,b}(X;M) \rightarrow\tilde{H}^{a+p,b+q}(\Sigma^{p,q}X;M)\]
for all $p,q\geq 0$. Moreover, given a cofiber sequence of based $C_3$-spaces
\[X\xrightarrow{f}Y\rightarrow C(f)\]
we get a long exact sequence
\[\cdots \rightarrow \tilde{H}^{p,q}(Y;M)\rightarrow \tilde{H}^{p,q}(X;M)\xrightarrow{d^{p,q}} \tilde{H}^{p+1,q}(C(f);M) \rightarrow \tilde{H}^{p+1,q}(Y;M)\rightarrow \cdots\]
for each $q\in\mathbb{Z}$.

\begin{figure}
\begin{center}
\includegraphics[scale=.5]{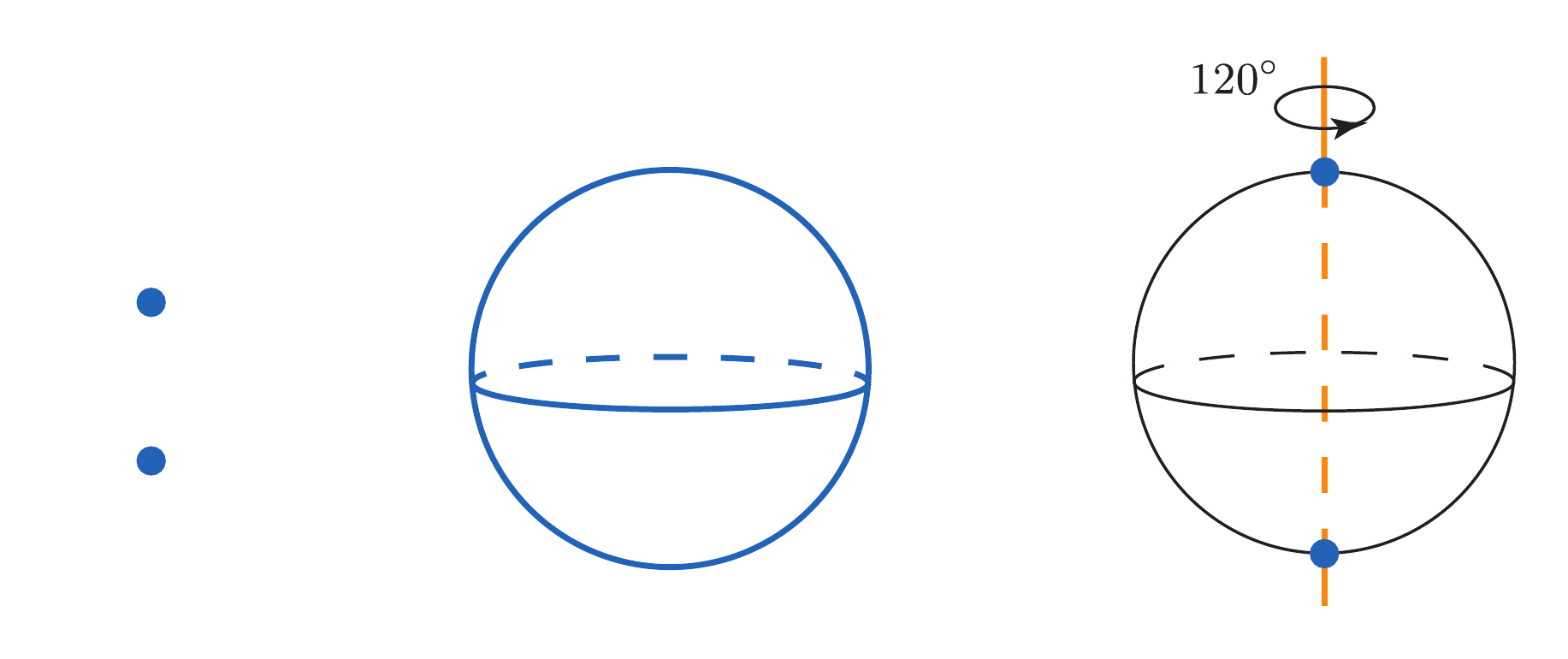}
\end{center}
\caption{\label{repspheres} The representation spheres $S^{0,0}$, $S^{2,0}$, and $S^{2,1}$, respectively.}
\end{figure}

\subsection{Cohomology of Orbits}

Here we give the cohomology of $C_3/C_3=\text{pt}$ and the free orbit $C_3$ in constant $\underline{\mathbb{Z}/3}$ coefficients. These computations have been done in \cite{Lew88}, so we just give the ring structure below.

Let $\mathbb{M}_3$ denote the ring $H^{*,*}(\text{pt};\underline{\mathbb{Z}/3})$ which is depicted in Figure \ref{M3}. The $(p,q)$ spot on the grid denotes the cohomology group $H^{p,q}(\text{pt};\underline{\mathbb{Z}/3})$, and each dot represents a copy of $\mathbb{Z}/3$. Solid lines indicate ring structure as we explain below. We use the convention that the $(p,q)$th entry is plotted above and to the right of the $(p,q)$th coordinate.

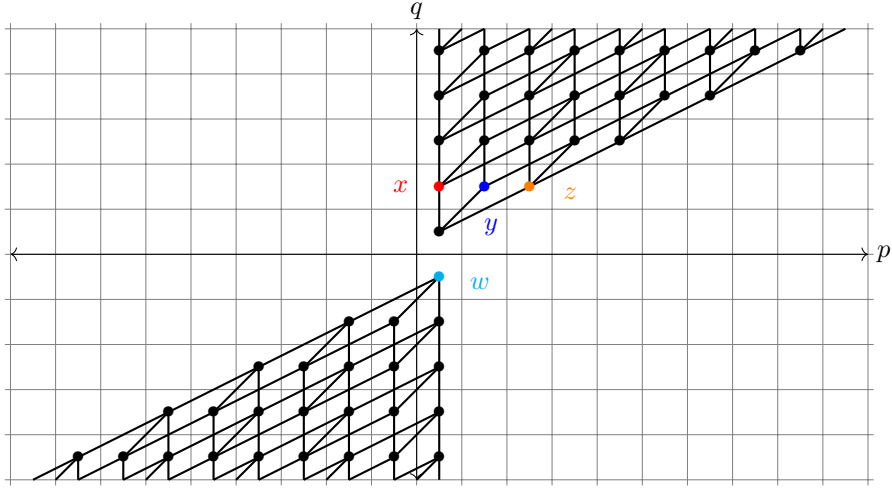
\begin{figure}
\begin{center}
	\begin{tikzpicture}[scale=0.6]
		\draw[help lines] (-9.125,-5.125) grid (10.125, 5.125);
		\draw[<->] (-9,0)--(10,0)node[right]{$p$};
		\draw[<->] (0,-5)--(0,5)node[above]{$q$};
		\cone{0}{0}{black};
		\draw (1/2,1/2) node{$\bullet$};
		\draw (1/2,5/2) node{$\bullet$};
		\draw (1/2,7/2) node{$\bullet$};
		\draw (1/2,9/2) node{$\bullet$};
		\draw (1/2,-3/2) node{$\bullet$};
		\draw (1/2,-5/2) node{$\bullet$};
		\draw (1/2,-7/2) node{$\bullet$};
		\draw (1/2,-9/2) node{$\bullet$};
		\draw (3/2,5/2) node{$\bullet$};
		\draw (3/2,7/2) node{$\bullet$};
		\draw (3/2,9/2) node{$\bullet$};
		\draw (5/2,5/2) node{$\bullet$};
		\draw (5/2,7/2) node{$\bullet$};
		\draw (5/2,9/2) node{$\bullet$};
		\draw (7/2,5/2) node{$\bullet$};
		\draw (7/2,7/2) node{$\bullet$};
		\draw (7/2,9/2) node{$\bullet$};
		\draw (9/2,5/2) node{$\bullet$};
		\draw (9/2,7/2) node{$\bullet$};
		\draw (9/2,9/2) node{$\bullet$};
		\draw (11/2,7/2) node{$\bullet$};
		\draw (11/2,9/2) node{$\bullet$};
		\draw (13/2,7/2) node{$\bullet$};
		\draw (13/2,9/2) node{$\bullet$};
		\draw (15/2,9/2) node{$\bullet$};
		\draw (17/2,9/2) node{$\bullet$};
		\draw (-1/2,-3/2) node{$\bullet$};
		\draw (-1/2,-5/2) node{$\bullet$};
		\draw (-1/2,-7/2) node{$\bullet$};
		\draw (-1/2,-9/2) node{$\bullet$};
		\draw (-3/2,-3/2) node{$\bullet$};
		\draw (-3/2,-5/2) node{$\bullet$};
		\draw (-3/2,-7/2) node{$\bullet$};
		\draw (-3/2,-9/2) node{$\bullet$};
		\draw (-5/2,-5/2) node{$\bullet$};
		\draw (-5/2,-7/2) node{$\bullet$};
		\draw (-5/2,-9/2) node{$\bullet$};
		\draw (-7/2,-5/2) node{$\bullet$};
		\draw (-7/2,-7/2) node{$\bullet$};
		\draw (-7/2,-9/2) node{$\bullet$};
		\draw (-9/2,-7/2) node{$\bullet$};
		\draw (-9/2,-9/2) node{$\bullet$};
		\draw (-11/2,-7/2) node{$\bullet$};
		\draw (-11/2,-9/2) node{$\bullet$};
		\draw (-13/2,-9/2) node{$\bullet$};
		\draw (-15/2,-9/2) node{$\bullet$};
		\draw[thick] (3/2,3/2) -- (3/2,5);
		\draw[thick] (5/2,3/2) -- (5/2,5);
		\draw[thick] (7/2,5/2) -- (7/2,5);
		\draw[thick] (9/2,5/2) -- (9/2,5);
		\draw[thick] (11/2,7/2) -- (11/2,5);
		\draw[thick] (13/2,7/2) -- (13/2,5);
		\draw[thick] (-1/2,-3/2) -- (-1/2,-5);
		\draw[thick] (-3/2,-3/2) -- (-3/2,-5);
		\draw[thick] (-5/2,-5/2) -- (-5/2,-5);
		\draw[thick] (-7/2,-5/2) -- (-7/2,-5);
		\draw[thick] (-9/2,-7/2) -- (-9/2,-5);
		\draw[thick] (-11/2,-7/2) -- (-11/2,-5);
		\draw[thick] (15/2,9/2) -- (15/2,5);
		\draw[thick] (17/2,9/2) -- (17/2,5);
		\draw[thick] (-13/2,-9/2) -- (-13/2,-5);
		\draw[thick] (-15/2,-9/2) -- (-15/2,-5);
		\draw[thick] (1/2,3/2) -- (15/2,5);
		\draw[thick] (1/2,5/2) -- (11/2,5);
		\draw[thick] (1/2,7/2) -- (7/2,5);
		\draw[thick] (1/2,9/2) -- (3/2,5);
		\draw[thick] (3/2,3/2) -- (17/2,5);
		\draw[thick] (3/2,5/2) -- (13/2,5);
		\draw[thick] (3/2,7/2) -- (9/2,5);
		\draw[thick] (3/2,9/2) -- (5/2,5);
		\draw[thick] (1/2,1/2) -- (3/2,3/2);
		\draw[thick] (1/2,3/2) -- (3/2,5/2);
		\draw[thick] (1/2,5/2) -- (3/2,7/2);
		\draw[thick] (1/2,7/2) -- (3/2,9/2);
		\draw[thick] (1/2,9/2) -- (1,5);
		\draw[thick] (5/2,3/2) -- (7/2,5/2);
		\draw[thick] (5/2,5/2) -- (7/2,7/2);
		\draw[thick] (5/2,7/2) -- (7/2,9/2);
		\draw[thick] (5/2,9/2) -- (3,5);
		\draw[thick] (9/2,5/2) -- (11/2,7/2);
		\draw[thick] (9/2,7/2) -- (11/2,9/2);
		\draw[thick] (9/2,9/2) -- (5,5);
		\draw[thick] (13/2,7/2) -- (15/2,9/2);
		\draw[thick] (13/2,9/2) -- (7,5);
		\draw[thick] (17/2,9/2) -- (9,5);
		\draw[thick] (1/2,-3/2) -- (-13/2,-5);
		\draw[thick] (1/2,-5/2) -- (-9/2,-5);
		\draw[thick] (1/2,-7/2) -- (-5/2,-5);
		\draw[thick] (1/2,-9/2) -- (-1/2,-5);
		\draw[thick] (-1/2,-3/2) -- (-15/2,-5);
		\draw[thick] (-1/2,-5/2) -- (-11/2,-5);
		\draw[thick] (-1/2,-7/2) -- (-7/2,-5);
		\draw[thick] (-1/2,-9/2) -- (-3/2,-5);
		\draw[thick] (1/2,-1/2) -- (-1/2,-3/2);
		\draw[thick] (1/2, -3/2) -- (-1/2,-5/2);
		\draw[thick] (1/2,-5/2) -- (-1/2,-7/2);
		\draw[thick] (1/2,-7/2) -- (-1/2,-9/2);
		\draw[thick] (1/2,-9/2) -- (0,-5);
		\draw[thick] (-3/2,-3/2) -- (-5/2,-5/2);
		\draw[thick] (-3/2,-5/2) -- (-5/2,-7/2);
		\draw[thick] (-3/2,-7/2) -- (-5/2,-9/2);
		\draw[thick] (-3/2,-9/2) -- (-2,-5);
		\draw[thick] (-7/2,-5/2) -- (-9/2,-7/2);
		\draw[thick] (-7/2,-7/2) -- (-9/2,-9/2);
		\draw[thick] (-7/2,-9/2) -- (-4,-5);
		\draw[thick] (-11/2,-7/2) -- (-13/2,-9/2);
		\draw[thick] (-11/2,-9/2) -- (-6,-5);
		\draw[thick] (-15/2,-9/2) -- (-8,-5);
		\draw[orange] (5/2,3/2) node{$\bullet$};
		\draw[blue] (3/2,3/2) node{$\bullet$};
		\draw[red] (1/2,3/2) node{$\bullet$};
		\draw[orange] (3.4,1.35) node{$z$};
		\draw[red] (-.35,1.5) node{$x$};
		\draw[blue] (1.65,.6) node{$y$};
		\draw[cyan] (1/2,-1/2) node{$\bullet$};
		\draw[cyan] (1.4,-.65) node{$w$};
	\end{tikzpicture}
	\end{center}
\caption{\label{M3} The ring $\mathbb{M}_3=H^{*,*}(\text{pt})$.}
\end{figure}

We will refer to the portion above the $p$-axis as the ``top cone'' and the portion below as the ``bottom cone''. The top cone is isomorphic to the polynomial ring $\mathbb{Z}/3[x,y,z]/(y^2)$ where $x$ is a generator of $\mathbb{Z}/3$ in degree $(0,1)$, $y$ is a generator in degree $(1,1)$, and $z$ is a generator in degree $(2,1)$. Multiplication by $x$ is denoted by vertical lines, multiplication by $y$ is denoted by lines of slope $1$, and multiplication by $z$ is denoted by lines of slope $1/2$.

The generator $w$ in degree $(0,-1)$ is infinitely divisible by $x$ and $z$ and is divisible by $y$. For example, there is an element denoted $\frac{w}{x}$ in degree $(0,-2)$ with the property that $x\cdot \frac{w}{x}=w$. More generally, all nonzero elements of the bottom cone are of the form $\pm\frac{w}{x^ky^iz^\ell}$ for some $k,\ell\in\mathbb{N}$ and $i\in\{0,1\}$.

Going forward we will use an abbreviated picture for $\mathbb{M}_3$ which we can see in Figure \ref{simpleM3}. Although this simpler version allows us to keep our diagrams from getting too busy, we are leaving out a lot of information about the ring structure.

\begin{figure}
\begin{minipage}{0.45\textwidth}
\begin{center}
\begin{tikzpicture}[scale=0.6]
	\draw[help lines] (-3.125,-5.125) grid (5.125, 5.125);
	\draw[<->] (-3,0)--(5,0)node[right]{$p$};
	\draw[<->] (0,-5)--(0,5)node[above]{$q$};;
	\scone{0}{0}{black};
\end{tikzpicture}
\end{center}
\end{minipage}\ \begin{minipage}{0.45\textwidth}
\begin{center}
\begin{tikzpicture}[scale=0.6]
	\draw[help lines] (-3.125,-5.125) grid (5.125, 5.125);
	\draw[<->] (-3,0)--(5,0)node[right]{$p$};
	\draw[<->] (0,-5)--(0,5)node[above]{$q$};;
	\scone{2}{1}{black};
\end{tikzpicture}
\end{center}
\end{minipage}
\caption{\label{simpleM3} Abbreviated representations of $\mathbb{M}_3$ and $\Sigma^{2,1}\mathbb{M}_3$, respectively.}
\end{figure}
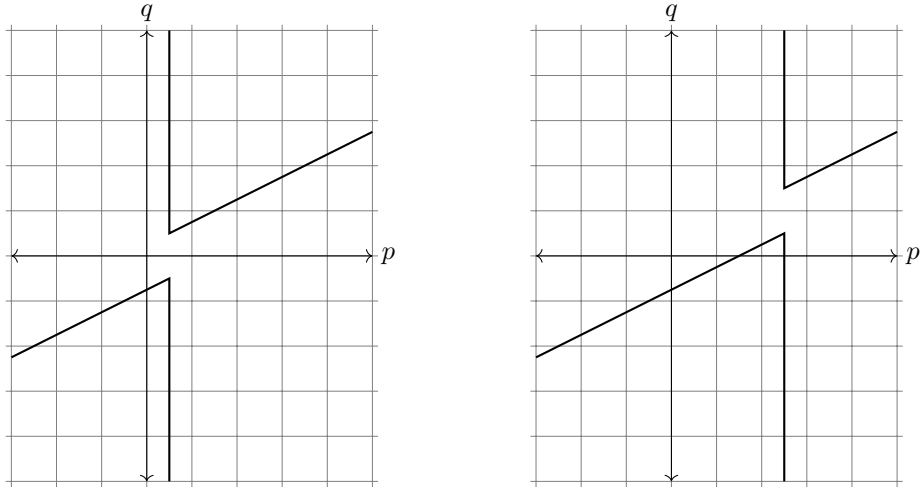

Given any $C_3$-space $X$, there is an equivariant map $X\rightarrow \text{pt}$ sending everything to a fixed point. We then get an induced map $\mathbb{M}_3\rightarrow H^{*,*}(X;\underline{\mathbb{Z}/3})$ so that $H^{*,*}(X;\underline{\mathbb{Z}/3})$ can be made into an $\mathbb{M}_3$-module for any $C_3$-space $X$. In this paper, we will utilize this structure and compute the cohomology of all non-trivial, closed $C_3$-surfaces as modules over $\mathbb{M}_3$. 

We next consider the free orbit $C_3$. As an $\mathbb{M}_3$-module, the cohomology of $C_3$ is isomorphic to $x^{-1}\mathbb{M}_3/(y,z)$. The module $H^{*,*}(C_3;\underline{\mathbb{Z}/3})$ is given on the left of Figure \ref{cthree} with an abbreviated picture on the right which we will use in future computations.

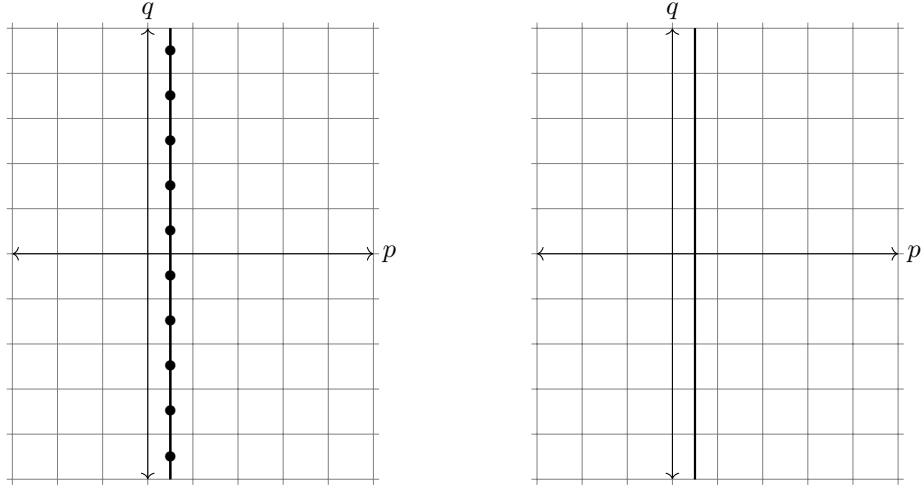
\begin{figure}
\begin{minipage}{0.45\textwidth}
\begin{center}
\begin{tikzpicture}[scale=0.6]
	\draw[help lines] (-3.125,-5.125) grid (5.125, 5.125);
	\draw[<->] (-3,0)--(5,0)node[right]{$p$};
	\draw[<->] (0,-5)--(0,5)node[above]{$q$};;
	\cthree{0}{0}{black};
	\draw (0.5,0.5) node{$\bullet$};
	\draw (0.5,1.5) node{$\bullet$};
	\draw (0.5,2.5) node{$\bullet$};
	\draw (0.5,3.5) node{$\bullet$};
	\draw (0.5,4.5) node{$\bullet$};
	\draw (0.5,-0.5) node{$\bullet$};
	\draw (0.5,-1.5) node{$\bullet$};
	\draw (0.5,-2.5) node{$\bullet$};
	\draw (0.5,-3.5) node{$\bullet$};
	\draw (0.5,-4.5) node{$\bullet$};
\end{tikzpicture}
\end{center}
\end{minipage}\ \begin{minipage}{0.45\textwidth}
\begin{center}
\begin{tikzpicture}[scale=0.6]
	\draw[help lines] (-3.125,-5.125) grid (5.125, 5.125);
	\draw[<->] (-3,0)--(5,0)node[right]{$p$};
	\draw[<->] (0,-5)--(0,5)node[above]{$q$};;
	\cthree{0}{0}{black};
\end{tikzpicture}
\end{center}
\end{minipage}
\caption{\label{cthree} The module $H^{*,*}(C_3)$ (left) and an abbreviated representation (right).}
\end{figure}

\subsection{Computational Tools}

We now introduce several properties relating $RO(C_3)$-graded Bredon cohomology to singular cohomolgy which will become extremely useful in later computations.

\begin{lemma}
[The quotient lemma] Let $X$ be a finite $C_3$-CW complex. We have the following isomorphism for all $p$:
\[H^{p,0}(X;\underline{\mathbb{Z}/3})\cong H^{p,0}(X/C_3;\underline{\mathbb{Z}/3})\cong H^p_{\textup{sing}}(X/C_3;\mathbb{Z}/3).\]
\end{lemma}

A proof for the analogous statement in the $G=C_2$ case is nearly identical to that of the $C_3$ case and can be found in \cite{Haz19a}, but we will briefly summarize the main idea. The map $X\to X/C_3$ induces $H^{p,0}(X;\underline{\Z/3})\leftarrow H^{p,0}(X/C_3\underline{\Z/3})$, and it is quick to check that this induced map is an isomorphism of integer-graded equivariant cohomology theories.

\begin{lemma}\label{times C_3}
Let $Y$ be a non-equivariant space. The cohomology of the free $C_3$-space $C_3\times Y$ is given by
\[H^{*,*}(C_3\times Y;\underline{\mathbb{Z}/3})\cong \mathbb{Z}/3[x,x\inv]\otimes_{\mathbb{Z}/3} H^*_{\textup{sing}}(Y;\mathbb{Z}/3)\]
as $\mathbb{M}_3$-modules.
\end{lemma}

\begin{proof}
For this proof, all coefficients are understood to be $\underline{\mathbb{Z}/3}$, so we will suppress the notation.

 The equivariant map $C_3\times Y\rightarrow C_3$ sending each copy of $Y$ to a single point induces a map $H^{*,*}(C_3)\rightarrow H^{*,*}(C_3\times Y)$. Since $H^{*,*}(C_3)\cong\mathbb{Z}/3[x,x^{-1}]$, we can make $H^{*,*}(C_3\times Y)$ a graded algebra over $\mathbb{Z}/3[x,x^{-1}]$. This means there exist natural maps
\[\mathbb{Z}/3[x,x^{-1}]\otimes_{\mathbb{Z}/3} H^{p,0}(C_3\times Y)\rightarrow H^{p,q}(C_3\times Y).\]
Restricting to the $q$th graded piece gives us a map
\[\left[\mathbb{Z}/3[x,x^{-1}]\otimes_{\mathbb{Z}/3} H^{*,0}(C_3\times Y)\right]^q\rightarrow H^{*,q}(C_3\times Y)\]
which is natural in $Y$. In particular, this is a map of cohomology theories. We can quickly see that this map is an isomorphism on both $Y=\text{pt}$ and $Y=C_3$, which means it defines an isomorphism of equivariant cohomology theories for each $q$.

We know from the quotient lemma that $H^{*,0}(C_3\times Y)\cong H^*_{\text{sing}}(Y)$ for all $Y$, so we have isomorphisms $\mathbb{Z}/3[x,x^{-1}]\otimes_{\mathbb{Z}/3} H^*_{\text{sing}}(Y)\rightarrow H^{*,*}(C_3\times Y)$ for each $q$th graded piece. Together, these give us an isomorphism of $\mathbb{M}_3$-modules, and the result follows. 
\end{proof}

Another useful tool to aid us in computations is the \textbf{forgetful map}. Let $X$ be a pointed $C_3$-space. For every integer $q$, we have map
\[\tilde{H}^{p,q}(X;\underline{\mathbb{Z}/3})\xrightarrow{\psi} \tilde{H}^p_{\text{sing}}(X;\mathbb{Z}/3).\]
To understand this map, for each $V\cong \mathbb{R}^{p,q}$, we define $H^{p,q}(X;\underline{\mathbb{Z}/3})$ as maps from $X$ to the equivariant Eilenberg-MacLane space $K(\underline{\mathbb{Z}/3},p,q)$. Forgetting this equivariant structure leaves us with a map from the underlying topological space $X$ to the Eilenberg-MacLane space $K(\mathbb{Z}/3,p)$.



\begin{example} We will now use these tools to compute the cohomology of the freely rotating cirlce, $S^1_{\text{free}}$. All coefficients are understood to be $\underline{\mathbb{Z}/3}$, so the coefficient notation will be suppressed.We begin by constructing a cofiber sequence
\[{C_3}_+ \hookrightarrow {S^1_{\text{free}}}_+ \rightarrow S^{1,0}\wedge {C_3}_+\]
which is illustrated in Figure \ref{freespherecofibersequence}. This cofiber sequence gives rise to a long exact sequence on cohomology:
\[\cdots \rightarrow \tilde{H}^{p,q}(S^{1,0}\wedge {C_3}_+)\rightarrow H^{p,q}(S^1_{\text{free}})\rightarrow H^{p,q}(C_3)\xrightarrow{d^{p,q}} \tilde{H}^{p+1,q}(S^{1,0}\wedge {C_3}_+)\rightarrow\cdots\]
for each value of $q$. The total differential of these long exact sequences $d=\bigoplus_{p,q}d^{p,q}$ is an $\mathbb{M}_3$-module map, so we can understand $H^{*,*}(S^1_{\text{free}})$ by computing the total differential and solving the extension problem
\[0\rightarrow \coker(d)\rightarrow H^{*,*}(S^1_{\text{free}})\rightarrow \ker(d)\rightarrow 0.\] 
Figure \ref{S^1_f differential} shows all possible nonzero differential maps
\[d^{p,q}\colon H^{p,q}(C_3)\rightarrow \tilde{H}^{p+1,q}(S^{1,0}\wedge {C_3}_+).\]

\begin{figure}
\begin{center}
\includegraphics[scale=.75]{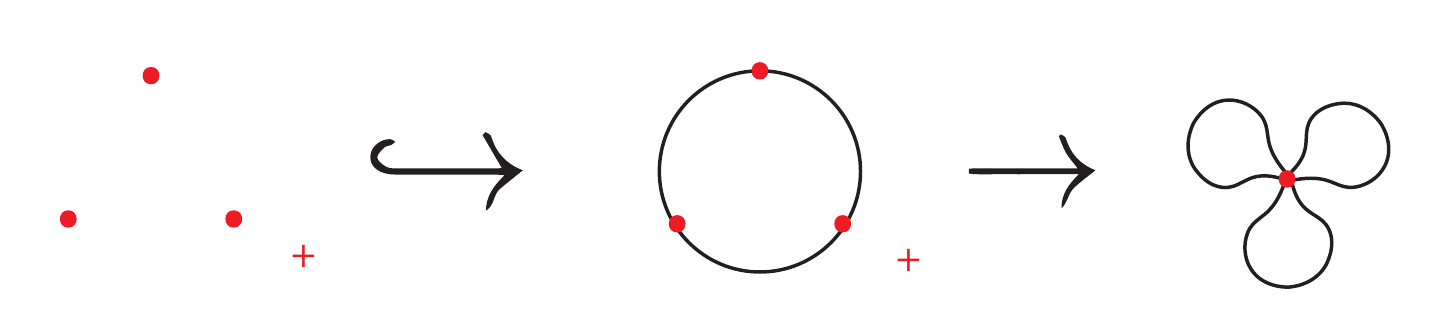}
\end{center}
\caption{\label{freespherecofibersequence} The cofiber sequence ${C_3}_+ \hookrightarrow {S^1_{\text{free}}}_+ \rightarrow S^{1,0}\wedge {C_3}_+$.}
\end{figure}

\begin{figure}
\begin{center}
	\begin{tikzpicture}[scale=0.6]
		\draw[help lines] (-2.125,-5.125) grid (5.125, 5.125);
		\draw[<->] (-2,0)--(5,0)node[right]{$p$};
		\draw[<->] (0,-5)--(0,5)node[above]{$q$};;
		\cthree{0}{0}{red};
		\cthree{1}{0}{blue};
		\draw[very thick,->] (.5,.5) -- (1.5,.5);
		\draw[very thick,->] (.5,1.5) -- (1.5,1.5);
		\draw[very thick,->] (.5,2.5) -- (1.5,2.5);
		\draw[very thick,->] (.5,3.5) -- (1.5,3.5);
		\draw[very thick,->] (.5,4.5) -- (1.5,4.5);
		\draw[very thick,->] (.5,-.5) -- (1.5,-.5);
		\draw[very thick,->] (.5,-1.5) -- (1.5,-1.5);
		\draw[very thick,->] (.5,-2.5) -- (1.5,-2.5);
		\draw[very thick,->] (.5,-3.5) -- (1.5,-3.5);
		\draw[very thick,->] (.5,-4.5) -- (1.5,-4.5);
		\draw[red] (.5,.5) node{$\bullet$};
		\draw[red] (.5,1.5) node{$\bullet$};
		\draw[red] (.5,2.5) node{$\bullet$};
		\draw[red] (.5,3.5) node{$\bullet$};
		\draw[red] (.5,4.5) node{$\bullet$};
		\draw[red] (.5,-.5) node{$\bullet$};
		\draw[red] (.5,-1.5) node{$\bullet$};
		\draw[red] (.5,-2.5) node{$\bullet$};
		\draw[red] (.5,-3.5) node{$\bullet$};
		\draw[red] (.5,-4.5) node{$\bullet$};
		\draw[blue] (1.5,.5) node{$\bullet$};
		\draw[blue] (1.5,1.5) node{$\bullet$};
		\draw[blue] (1.5,2.5) node{$\bullet$};
		\draw[blue] (1.5,3.5) node{$\bullet$};
		\draw[blue] (1.5,4.5) node{$\bullet$};
		\draw[blue] (1.5,-.5) node{$\bullet$};
		\draw[blue] (1.5,-1.5) node{$\bullet$};
		\draw[blue] (1.5,-2.5) node{$\bullet$};
		\draw[blue] (1.5,-3.5) node{$\bullet$};
		\draw[blue] (1.5,-4.5) node{$\bullet$};
	\end{tikzpicture}
	\end{center}
\caption{\label{S^1_f differential} The differential $d\colon H^{*,*}(C_3)\rightarrow \tilde{H}^{*+1,*}(S^{1,0}\wedge{C_3}_+)$.}
\end{figure}
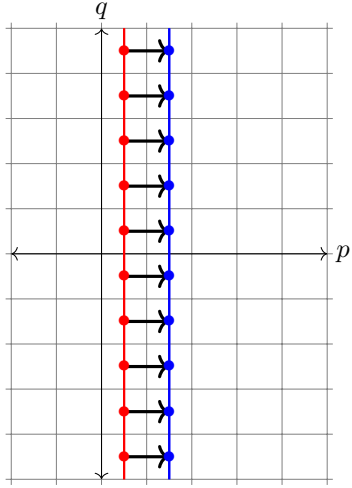

Since $H^{p,q}(C_3)\cong \mathbb{Z}/3[x,x\inv]$, we know that $H^{0,q}(C_3)$ and $\tilde{H}^{1,q}(S^{1,0}\wedge {C_3}_+)$ must be $\mathbb{Z}/3$ for each $q$. By linearity of the differential, $d^{0,q}$ is either $0$ or an isomorphism for all $q$. 

The quotient lemma tells us that $H^{1,q}(S^1_{\text{free}})\cong \mathbb{Z}/3$, so the map $\tilde{H}^{1,q}(S^{1,0}\wedge {C_3}_+)\rightarrow H^{1,q}(S^1_{\text{free}})$ in the long exact sequence must be an isomorphism when $q=0$. This implies that $d^{0,0}=0$, and thus the total differential is 0 by linearity. We can then conclude that $H^{p,q}(S^1_{\text{free}})=\mathbb{Z}/3$ when $p=0$ or $1$. 

There is still a question of whether or not the extension $\operatorname{coker}(d) \rightarrow H^{*,*}(S^1_{\text{free}})\rightarrow \ker (d)$ is trivial. In particular, we want to know if $ya$ is nonzero for $a\in H^{0,q}(S^1_{\text{free}})$. To do this, we will instead compute the $\mathbb{M}_3$-module structure on $\tilde{H}^{*,*}(S^{1,0}\wedge ({S^1_{\text{free}}}_+))\cong \tilde{H}^{*-1,*}({S^1_{\text{free}}}_+)\cong H^{*-1,*}(S^1_{\text{free}})$ using another cofiber sequence. 

The space $S^{1,0}\wedge ({S^1_{\text{free}}}_+)$ can be constructed as the cofiber of the map $S^{0,0}\rightarrow S^{2,1}$. Suspending along the Puppe sequence yields
\[S^{2,1}\rightarrow S^{1,0}\wedge ({S^1_{\text{free}}}_+) \rightarrow S^{1,0}\wedge S^{0,0}.\]
From here we can follow the same procedure of looking at the long exact sequence on cohomology and computing its differential $d\colon H^{*,*}(S^{2,1})\rightarrow H^{*+1,*}(S^{1,0}\wedge S^{0,0})$ which is shown in the left of Figure \ref{S^1_f differential2}. We know the group structure of $\tilde{H}^{*,*}\left(S^{1,0}\wedge ({S^1_{\text{free}}}_+)\right)$ from our previous computations, so it must be the case that $d$ maps the generator of $\Sigma^{2,1}\mathbb{M}_3$ to $z$ times the generator of $\Sigma^{1,0}\mathbb{M}_3$. The right of Figure \ref{S^1_f differential2} shows $\ker (d)$ and $\operatorname{coker}(d)$ in this case. Comparing the information from Figures \ref{S^1_f differential} and \ref{S^1_f differential2} (noting that the latter represents a shifted copy of $H^{*,*}(S^1_{\text{free}})$), we can see that 
\[H^{*,*}(S^1_{\text{free}})\cong x\inv \mathbb{M}_3/(z).\]

\begin{figure}
\begin{minipage}{0.45\textwidth}
\begin{center}
	\begin{tikzpicture}[scale=0.6]
		\draw[help lines] (-3.125,-5.125) grid (5.125, 5.125);
		\draw[<->] (-3,0)--(5,0)node[right]{$p$};
		\draw[<->] (0,-5)--(0,5)node[above]{$q$};;
		\scone{1}{0}{blue};
		\scone{2}{1}{red};
		\draw[very thick,->] (2.5,1.5) -- (3.5,1.5);
	\end{tikzpicture}
	\end{center}
\end{minipage} \ \begin{minipage}{0.45\textwidth}
\begin{center}
	\begin{tikzpicture}[scale=0.6]
		\draw[help lines] (-3.125,-5.125) grid (5.125, 5.125);
		\draw[<->] (-3,0)--(5,0)node[right]{$p$};
		\draw[<->] (0,-5)--(0,5)node[above]{$q$};;
		\draw[blue] (1.5,.5) -- (1.5,5);
		\draw[blue] (2.5,1.5) -- (2.5,5);
		\draw[blue] (1.5,.5) -- (2.5,1.5);
		\draw[blue] (1.5,1.5) -- (2.5,2.5);
		\draw[blue] (1.5,2.5) -- (2.5,3.5);
		\draw[blue] (1.5,3.5) -- (2.5,4.5);
		\draw[blue] (1.5,4.5) -- (2,5);
		\draw[red] (1.5,-.5) -- (1.5,-5);
		\draw[red] (2.5,.5) -- (2.5,-5);
		\draw[red] (1.5,-.5) -- (2.5,.5);
		\draw[red] (1.5,-1.5) -- (2.5,-.5);
		\draw[red] (1.5,-2.5) -- (2.5,-1.5);
		\draw[red] (1.5,-3.5) -- (2.5,-2.5);
		\draw[red] (1.5,-4.5) -- (2.5,-3.5);
		\draw[red] (2,-5) -- (2.5,-4.5);
	\end{tikzpicture}
\end{center}
\end{minipage} 
\caption{\label{S^1_f differential2} The map $d\colon \Sigma^{2,1}\mathbb{M}_3\rightarrow\Sigma^{1,0}\mathbb{M}_3$ (left) and its kernel and cokernel (right).}
\end{figure}
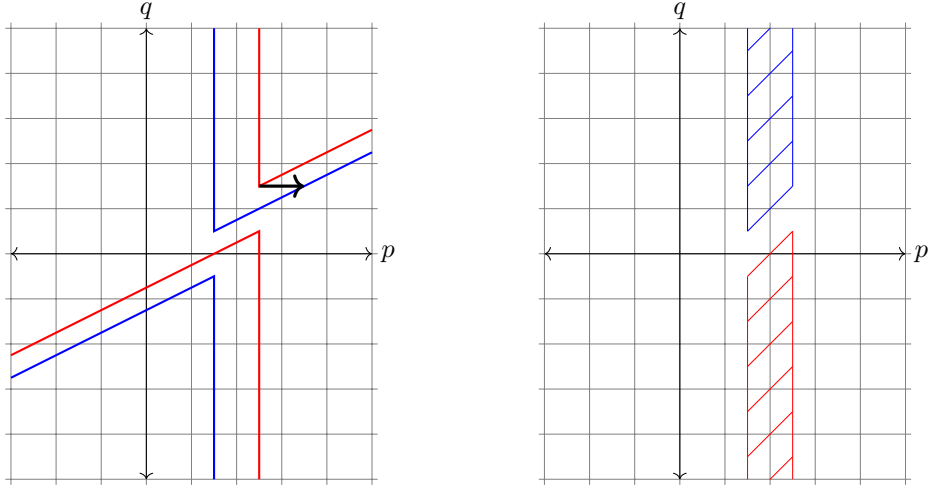
\end{example}

\begin{example}\label{eb}
We next compute the reduced cohomology of the ``eggbeater'' space. The eggbeater, denoted by $EB_3$ (or $EB$ when the action of $C_3$ is understood), can be defined as the cofiber of the map ${C_3}_+\rightarrow S^{0,0}$ which sends all of ${C_3}$ to a fixed point. An illustration of this cofiber sequence can be found in Figure \ref{EB cofib seq}.

\begin{figure}
\begin{center}
\includegraphics[scale=.75]{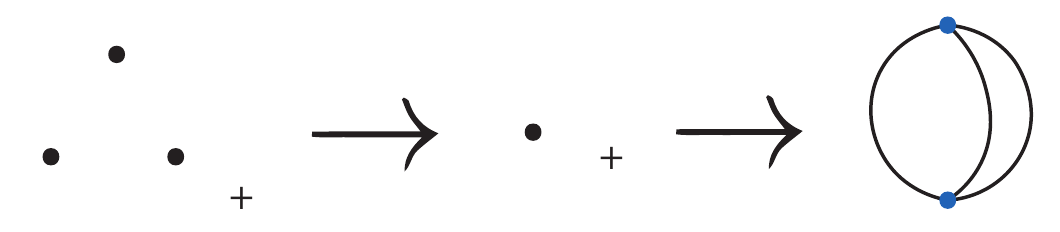}
\end{center}
\caption{\label{EB cofib seq} The cofiber sequence ${C_3}_+ \rightarrow S^{0,0} \rightarrow EB$.}
\end{figure}

To determine the cohomology of $EB$, we will instead consider the cofiber sequence 
\[EB \hookrightarrow S^{2,1} \rightarrow S^{2,0}\wedge{C_3}_+\]
which is depicted in Figure \ref{EBcofib2}. We can extend this via the Puppe sequence to get another cofiber sequence
\[S^{1,0}\wedge {C_3}_+ \rightarrow EB \rightarrow S^{2,1}.\]
Thus we get a long exact sequence on cohomology
\[\cdots \rightarrow \tilde{H}^{p,q}(S^{1,0}\wedge {C_3}_+)\rightarrow \tilde{H}^{p,q}(EB) \rightarrow \tilde{H}^{p,q}(S^{2,1})\xrightarrow{d} H^{p+1,q}(S^{1,0}\wedge {C_3}_+)\cdots\]
which can be understood by analyzing its total differential
\[d^{p,q}\colon \tilde{H}^{p,q}(S^{1,0}\wedge {C_3}_+)\rightarrow \tilde{H}^{p+1,q}(S^{2,1})\]
for all $(p,q)$. The differential $d^{1,0}$ is depicted in Figure \ref{EB diff 1}. Since the total differential $\bigoplus_{p,q}d^{p,q}$ is an $\mathbb{M}_3$-module map and $\tilde{H}^{p,q}(S^{1,0}\wedge {C_3}_+)=0$ when $p\neq 1$, we only need to compute $d^{1,0}$.

\begin{figure}
\begin{center}
\includegraphics[scale=.4]{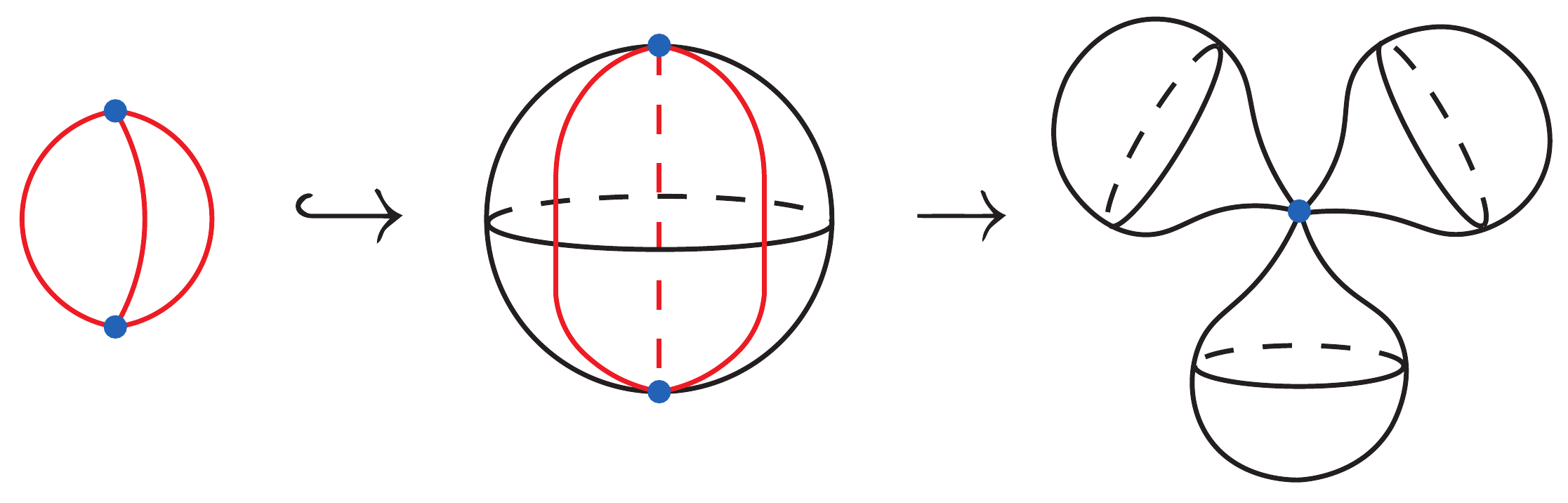}
\end{center}
\caption{\label{EBcofib2} The cofiber sequence $EB\hookrightarrow S^{2,1}\rightarrow S^{2,0}\wedge{C_3}_+$.}
\end{figure}

\begin{figure}
\begin{center}
\begin{minipage}{0.45\textwidth}
\begin{center}
	\begin{tikzpicture}[scale=0.6]
		\draw[help lines] (-3.125,-5.125) grid (5.125, 5.125);
		\draw[<->] (-3,0)--(5,0)node[right]{$p$};
		\draw[<->] (0,-5)--(0,5)node[above]{$q$};;
		\scone{2}{1}{blue};
		\cthree{1}{0}{red};
		\draw[very thick,->] (1.5,.5) -- (2.5,.5);
	\end{tikzpicture}
\end{center}
\end{minipage} \ \begin{minipage}{0.45\textwidth}
\begin{center}
	\begin{tikzpicture}[scale=0.6]
		\draw[help lines] (-3.125,-5.125) grid (5.125, 5.125);
		\draw[<->] (-3,0)--(5,0)node[right]{$p$};
		\draw[<->] (0,-5)--(0,5)node[above]{$q$};;
		\draw[thick, blue] (1.5,-.5) -- (1.5,-5);
		\draw[thick, red] (1.5,1.5) -- (1.5,5);
		\draw[thick, blue] (2+1/2,5) -- (2+1/2,1+1/2) -- (5,{(5-2)/2+1+1/4});
		\draw[thick, blue] (1/2,-5) -- (1/2,-0.5) -- (-3,{(-3)/2-3/4});
	\end{tikzpicture}
\end{center}
\end{minipage} 
\end{center}
\caption{\label{EB diff 1} The differential $d^{0,0}$ (left), and $\text{ker}(d)$ and $\text{coker}(d)$ (right).}
\end{figure}
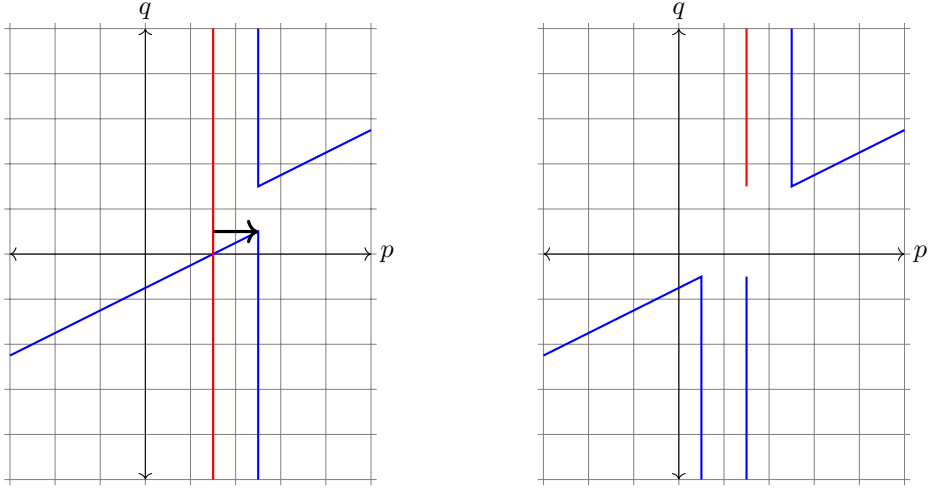

Observe that $EB/C_3\simeq pt$, so using the quotient lemma we know that 
\[\tilde{H}^{p,0}(EB)\cong \tilde{H}^p_{\text{sing}}(EB/C_3)\cong \tilde{H}^p_{\text{sing}}(\text{pt}).\]
So $\tilde{H}^{p,0}(EB)=0$ for all $p$. It then must be the case that $d^{1,0}$ is an isomorphism. By linearity, we have that $d^{1,q}$ is an isomorphism for all $q\leq 0$. We next want to understand the extension problem
\[0 \rightarrow \text{coker}(d)\rightarrow \tilde{H}^{p,q}(EB)\rightarrow \text{ker}(d)\rightarrow 0.\]
Since $\text{coker}(d)\subseteq \tilde{H}^{p,q}(EB)$, the module structure of $\text{coker}(d)$ is preserved. The extension here is nontrivial which we can see by going through a similar computation with the cofiber sequence $S^{0,0} \rightarrow EB \rightarrow S^{1,0}\wedge ({C_3}_+)$. In the end, we can think of the $\mathbb{M}_3$-module structure on $\tilde{H}^{*,*}(EB)$ as generated by elements $\alpha$ in degree $(2,1)$ and $\beta$ in degree $(1,1)$ with the relations $y\beta=0$ and $z\beta=y\alpha$. A more complete picture of this module structure is depicted on the left of Figure \ref{EBcohomology}. For brevity, we will use the representation of $\tilde{H}^{*,*}(EB;\underline{\mathbb{Z}/3})$ shown to the right of Figure \ref{EBcohomology}. 

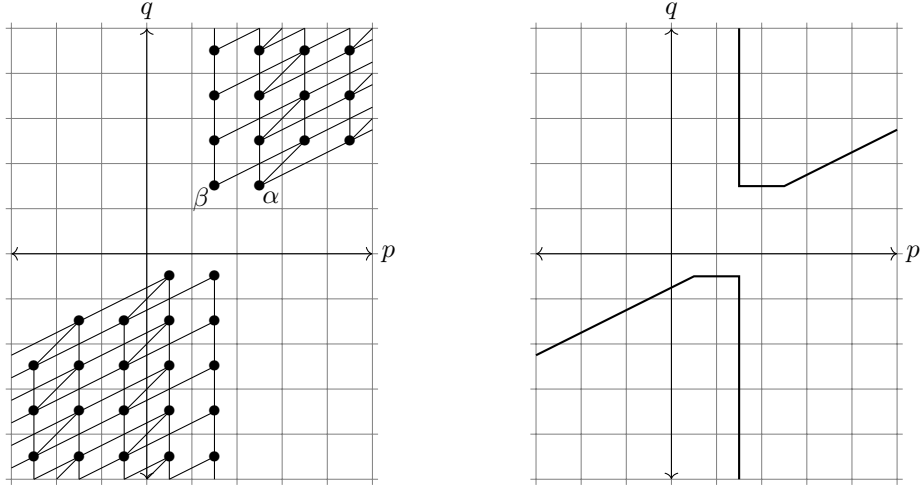
\begin{figure}
\begin{minipage}{0.45\textwidth}
\begin{center}
	\begin{tikzpicture}[scale=0.6]
		\draw[help lines] (-3.125,-5.125) grid (5.125, 5.125);
		\draw[<->] (-3,0)--(5,0)node[right]{$p$};
		\draw[<->] (0,-5)--(0,5)node[above]{$q$};;
		\draw (.5,-.5) node{$\bullet$};
		\draw (.5,-1.5) node{$\bullet$};
		\draw (.5,-2.5) node{$\bullet$};
		\draw (.5,-3.5) node{$\bullet$};
		\draw (.5,-4.5) node{$\bullet$};
		\draw (1.5,-.5) node{$\bullet$};
		\draw (1.5,-1.5) node{$\bullet$};
		\draw (1.5,-2.5) node{$\bullet$};
		\draw (1.5,-3.5) node{$\bullet$};
		\draw (1.5,-4.5) node{$\bullet$};
		\draw (1.5,2.5) node{$\bullet$};
		\draw (1.5,3.5) node{$\bullet$};
		\draw (1.5,4.5) node{$\bullet$};
		\draw (2.5,2.5) node{$\bullet$};
		\draw (2.5,3.5) node{$\bullet$};
		\draw (2.5,4.5) node{$\bullet$};
		\draw (3.5,2.5) node{$\bullet$};
		\draw (3.5,3.5) node{$\bullet$};
		\draw (3.5,4.5) node{$\bullet$};
		\draw (4.5,2.5) node{$\bullet$};
		\draw (4.5,3.5) node{$\bullet$};
		\draw (4.5,4.5) node{$\bullet$};
		\draw (-.5,-1.5) node{$\bullet$};
		\draw (-.5,-2.5) node{$\bullet$};
		\draw (-.5,-3.5) node{$\bullet$};
		\draw (-.5,-4.5) node{$\bullet$};
		\draw (-1.5,-1.5) node{$\bullet$};
		\draw (-1.5,-2.5) node{$\bullet$};
		\draw (-1.5,-3.5) node{$\bullet$};
		\draw (-1.5,-4.5) node{$\bullet$};
		\draw (-2.5,-2.5) node{$\bullet$};
		\draw (-2.5,-3.5) node{$\bullet$};
		\draw (-2.5,-4.5) node{$\bullet$};
		\draw (1.5,1.5) -- (1.5,5);
		\draw (2.5,1.5) -- (2.5,5);
		\draw (3.5,2.5) -- (3.5,5);
		\draw (4.5,2.5) -- (4.5,5);
		\draw (1.5,-.5) -- (1.5,-5);
		\draw (.5,-.5) -- (.5,-5);
		\draw (-.5,-1.5) -- (-.5,-5);
		\draw (-1.5,-1.5) -- (-1.5,-5);
		\draw (-2.5,-2.5) -- (-2.5,-5);
		\draw (2.5,1.5) -- (3.5,2.5);
		\draw (2.5,2.5) -- (3.5,3.5);
		\draw (2.5,3.5) -- (3.5,4.5);
		\draw (2.5,4.5) -- (3,5);
		\draw (4.5,2.5) -- (5,3);
		\draw (4.5,3.5) -- (5,4);
		\draw (4.5,4.5) -- (5,5);
		\draw (-.5,-1.5) -- (.5,-.5);
		\draw (-.5,-2.5) -- (.5,-1.5);
		\draw (-.5,-3.5) -- (.5,-2.5);
		\draw (-.5,-4.5) -- (.5,-3.5);
		\draw (0,-5) -- (.5,-4.5);
		\draw (-2.5,-2.5) -- (-1.5,-1.5);
		\draw (-2.5,-3.5) -- (-1.5,-2.5);
		\draw (-2.5,-4.5) -- (-1.5,-3.5);
		\draw (-2,-5) -- (-1.5,-4.5);
		\draw (2.5,1.5) -- (5,2.75);
		\draw (2.5,2.5) -- (5,3.75);
		\draw (2.5,3.5) -- (5,4.75);
		\draw (2.5,4.5) -- (3.5,5);
		\draw (1.5,1.5) -- (5,3.25);
		\draw (1.5,2.5) -- (5,4.25);
		\draw (1.5,3.5) -- (4.5,5);
		\draw (1.5,4.5) -- (2.5,5);
		\draw (.5,-.5) -- (-3,-2.25);
		\draw (.5,-1.5) -- (-3,-3.25);
		\draw (.5,-2.5) -- (-3,-4.25);
		\draw (.5,-3.5) -- (-2.5,-5);
		\draw (.5,-4.5) -- (-.5,-5);
		\draw (1.5,-.5) -- (-3,-2.75);
		\draw (1.5,-1.5) -- (-3,-3.75);
		\draw (1.5,-2.5) -- (-3,-4.75);
		\draw (1.5,-3.5) -- (-1.5,-5);
		\draw (1.5,-4.5) -- (.5,-5);
		\draw (2.5,1.5) node{$\bullet$};
		\draw (1.5,1.5) node{$\bullet$};
		\draw (2.75,1.25) node{$\alpha$};
		\draw (1.2,1.2) node{$\beta$};
	\end{tikzpicture}
	\end{center}
\end{minipage}\ \begin{minipage}{0.45\textwidth}
\begin{center}
	\begin{tikzpicture}[scale=0.6]
		\draw[help lines] (-3.125,-5.125) grid (5.125, 5.125);
		\draw[<->] (-3,0)--(5,0)node[right]{$p$};
		\draw[<->] (0,-5)--(0,5)node[above]{$q$};;
		\eb{1}{1}{black};
	\end{tikzpicture}
	\end{center}
\end{minipage}
\caption{\label{EBcohomology} The $\mathbb{M}_3$-module $\ebc$ (left) with abbreviated representation (right).}
\end{figure}
\end{example}

Going forward, we will let $\mathbb{EB}$ represent the $\mathbb{M}_3$-module $\tilde{H}^{*,*}(EB;\underline{\mathbb{Z}/3})$.

\begin{example}\label{N_1[1] cohomology}
Let $N_1[1]$ denote the $C_3$-surface depicted in Figure \ref{N1[1]hexagon}. Note that the underlying topological space is $\mathbb{R} P^2$, sometimes denoted by $N_1$. To compute the cohomology of this surface, let us consider the cofiber sequence
\[{S^1_{\text{free}}}_+ \hookrightarrow N_1[1] \rightarrow S^{2,1}\]
which is illustrated in Figure \ref{N1[1] cofib seq}. Thus we have the following long exact sequence on cohomology
\[\cdots \rightarrow \tilde{H}^{p,q}(S^{2,1})\rightarrow H^{p,q}(N_1[1])\rightarrow H^{p,q}(S^1_{\text{free}}) \xrightarrow{d^{p,q}} \tilde{H}^{p+1,q}(S^{2,1})\rightarrow \cdots\]
As in the previous examples, in order to compute $H^{p,q}(N_1[1])$ we need to understand the differential maps
\[d^{p,q} \colon H^{p,q}(S^1_{\text{free}})\rightarrow \tilde{H}^{p+1,q}(S^{2,1})\]
for each $(p,q)$. We can use the quotient lemma to compute $d^{1,0}$ which will then determine the value of all other possible nonzero differential maps. The differential $d^{1,0}$ is depicted in Figure \ref{N1[1] diff}. 

\begin{figure}
\begin{center}
\includegraphics[scale=.45]{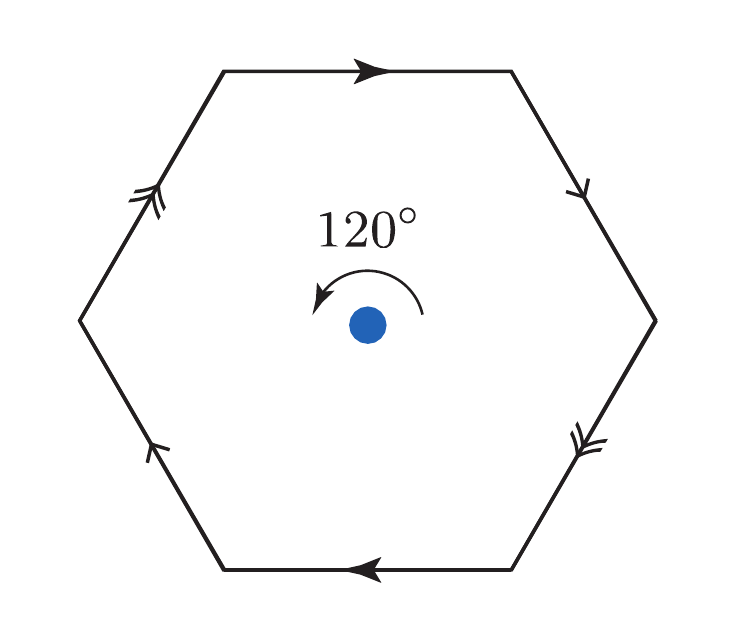}
\end{center}
\caption{\label{N1[1]hexagon} The space $N_1[1]$ whose underlying non-equivariant surface is $\mathbb{R} P^2$.}
\end{figure}

\begin{figure}
\begin{center}
\includegraphics[scale=.4]{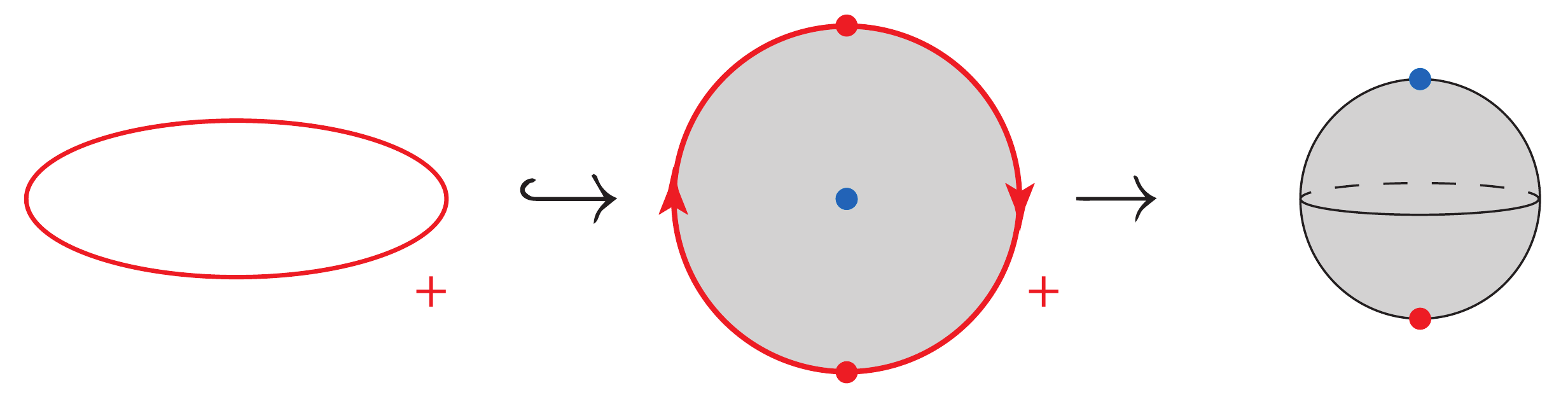}
\end{center}
\caption{\label{N1[1] cofib seq} The cofiber sequence ${S^1_{\text{free}}}_+ \hookrightarrow N_1[1]_+ \rightarrow S^{2,1}$.}
\end{figure}

\begin{figure}
\begin{center}
\begin{minipage}{0.45\textwidth}
\begin{center}
	\begin{tikzpicture}[scale=0.6]
		\draw[help lines] (-3.125,-5.125) grid (5.125, 5.125);
		\draw[<->] (-3,0)--(5,0)node[right]{$p$};
		\draw[<->] (0,-5)--(0,5)node[above]{$q$};;
		\cthree{0}{1}{red};
		\scone{2}{1}{blue};
		\draw[->,very thick] (1.5,.5) -- (2.5,.5);
	\end{tikzpicture}
\end{center}
\end{minipage} \ \begin{minipage}{0.45\textwidth}
\begin{center}
	\begin{tikzpicture}[scale=0.6]
		\draw[help lines] (-3.125,-5.125) grid (5.125, 5.125);
		\draw[<->] (-3,0)--(5,0)node[right]{$p$};
		\draw[<->] (0,-5)--(0,5)node[above]{$q$};;
		\draw[thick, blue] (2+1/2,5) -- (2+1/2,1+1/2) -- (5,{(5-2)/2+1+1/4});
		\draw[thick, blue] (1/2,-5) -- (1/2,-0.5) -- (-3,{(-3)/2-3/4});
		\draw[thick, red] (0.5,0.5) -- (0.5,5);
		\draw[thick, red] (1.5,1.5) -- (1.5,5);
		\draw[thick, red] (0.5,.5) -- (1.5,1.5);
		\draw[thick, red] (0.5,1.5) -- (1.5,2.5);
		\draw[thick, red] (0.5,2.5) -- (1.5,3.5);
		\draw[thick, red] (0.5,3.5) -- (1.5,4.5);
	\end{tikzpicture}
\end{center}
\end{minipage}
\end{center}
\caption{\label{N1[1] diff} The differential $d^{1,0}$ (left), and $\text{ker}(d)$ and $\text{coker}(d)$ (right).}
\end{figure}
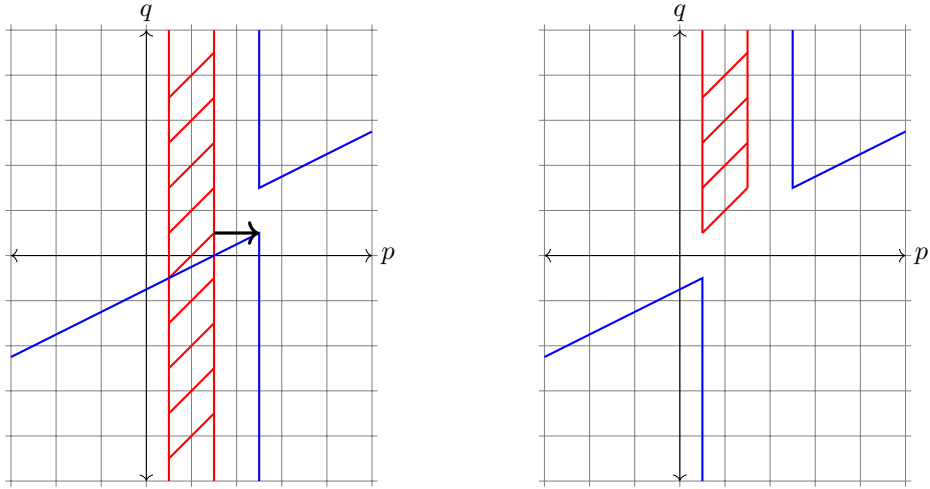

We can observe that $N_1[1]/C_3\cong \mathbb{R} P^2$. Recall $H^p_{\text{sing}}(\mathbb{R} P^2;\mathbb{Z}/3)\cong\mathbb{Z}/3$ when $p=0$, and it is $0$ for all other values of $p$. This implies the differential $d^{1,0}$ must be an isomorphism. The $\mathbb{M}_3$-module structure of $H^{*,*}(S^1_{\text{free}})$ then guarantees that $d^{1,q}$ is an isomorphism for all $q\leq 0$. We similarly find that $d^{0,q}$ must be an isomorphism for $q<0$. 

Now we are left to solve the extension problem of $\mathbb{M}_3$-modules
\[0\rightarrow\text{coker} (d)\rightarrow H^{*,*}(N_1[1])\rightarrow \text{ker}(d)\rightarrow 0.\]
However since $\mathbb{M}_3$ must be a submodule of $H^{*,*}(N_1[1])$, we conclude that $H^{*,*}(N_1[1])\cong \mathbb{M}_3$ and the extension is nontrivial.
\end{example}

\begin{remark}\label{M3summand}
In general, given a $C_3$-space $X$ with at least one fixed point, $\mathbb{M}_3$ is a summand of $H^{*,*}(X;\underline{\mathbb{Z}/3})$. This is because the inclusion $\textup{pt}\hookrightarrow X$ induces a surjective $\mathbb{M}_3$-module map $H^{*,*}(X)\twoheadrightarrow H^{*,*}(\textup{pt})=\mathbb{M}_3$.
\end{remark}




\section{Classification of $C_3$-surfaces}\label{surfacesection}


Up to isomorphism, all non-trivial, closed, connected surfaces with a $C_3$-action were classified in \cite{Pohl}. A method of constructing $C_3$-surfaces through a series of operations was developed, and it was proved that all $C_3$-surfaces can be constructed using the prescribed operations. The main result of \cite{Pohl} shows that there are six distinct families of isomorphism classes of $C_3$-surfaces which can be constructed using these equivariant surgery methods.

We use this section to introduce the language of equivariant surgery and state the main classification theorem. All proofs will be omitted from this paper but can be found in \cite{Pohl}.

\begin{notation}
    The following convention will always be used to discuss non-equivariant surfaces: $M_g$ denotes the genus $g$ orientable surface, and $N_r$ represents the genus $r$ non-orientable surface. 
\end{notation}

\subsection{Building blocks}

So far we have discussed a few important $C_3$-surfaces such as the representation sphere $S^{2,1}$ and the space $N_1[1]$ whose underlying surface is $\R P^2$. Our next goal is to introduce several other $C_3$-surfaces (informally referred to as ``building blocks'') from which we will construct all other closed surfaces with $C_3$-action. This section contains brief descriptions of these building block surfaces; the curious reader is directed to \cite{Pohl} for a more precise treatment of their definitions. 

There is a free $C_3$-action on the torus $M_1$ given by rotation of $120^\circ$ about its center. Denote this equivariant space by $M_1^{\text{free}}$. The Klein bottle also has a free action. Start with two copies of $N_1[1]$ and with each remove a copy of $D^{2,1}$ about the lone fixed point. The result is two M\"{o}bius bands with a free action. Then use an equivariant map to identify the boundary of these M\"{o}bius bands. The resulting space is non-equivariantly equivalent to a Klein bottle and inherits a free $C_3$-action. 

The final family of building block surfaces must be defined inductively. Much like the construction of $N_1[1]$, we start with a hexagon which has a natural rotation action of $C_3$. After identifying opposite edges of the hexagon as shown on the left in Figure \ref{cleanerhex}, the resulting space (denoted $\Hex_1$) is non-equivariantly equivalent to a torus, and its action has 3 fixed points.

\begin{figure}
    \centering
    \includegraphics[scale=.4]{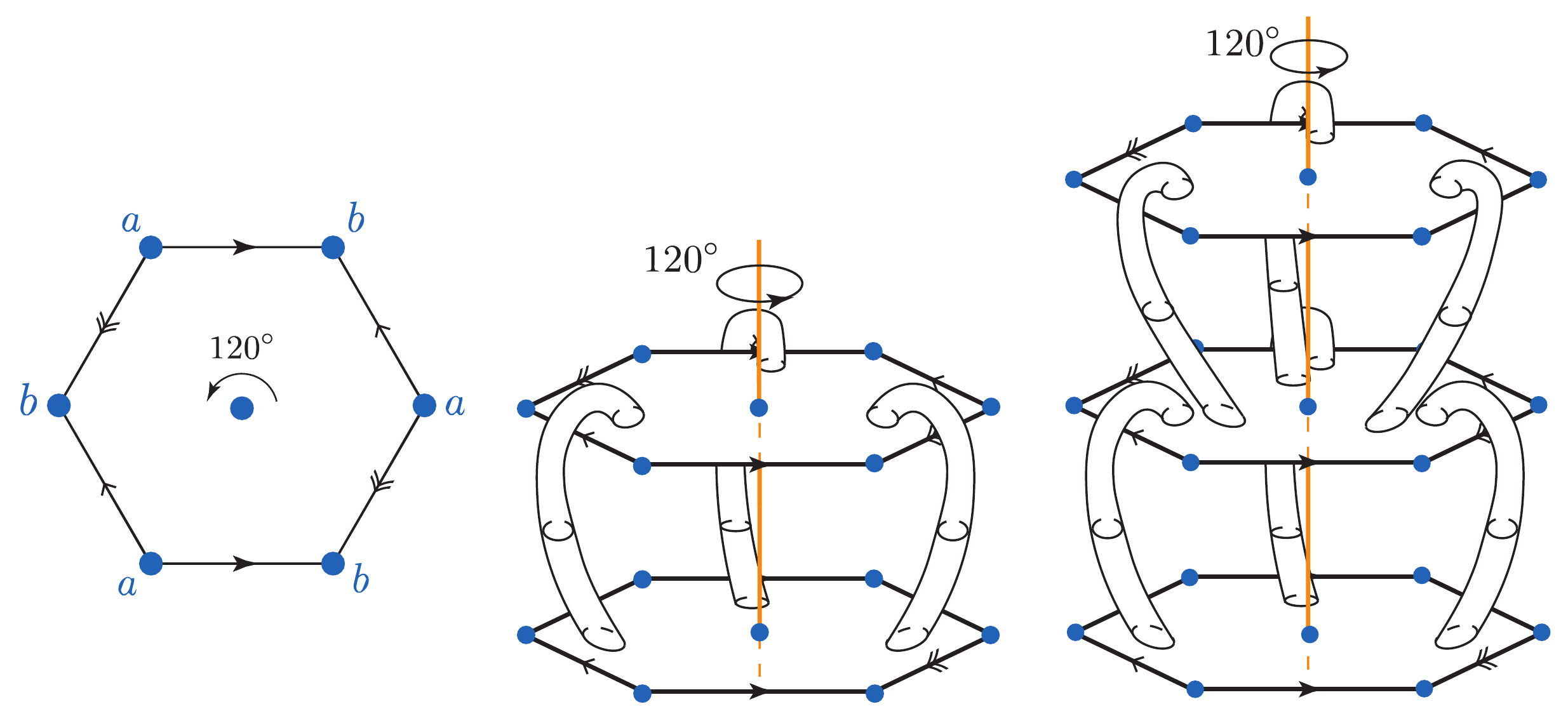}
    \caption{From left to right: the spaces $\Hex_1$, $\Hex_2$, and $\Hex_3$}
    \label{cleanerhex}
\end{figure}

To define $\Hex_2$, start with two copies of $\Hex_1$, and from each remove a set of conjugate disks. Glue the boundary of these spaces together along an equivariant map of degree $-1$. The result is the space $\Hex_2$, depicted in the center of Figure \ref{cleanerhex}. 

In general, $\Hex_n$ can be constructed from $\Hex_{n-1}$ by attaching a copy of $\Hex_1$ in the same way. From each of $\Hex_1$ and $\Hex_{n-1}$, remove three conjugate disks. Then identify their boundaries using an equivariant map of degree $-1$ to obtain $\Hex_n$. 

\subsection{Equivariant surgery constructions}

Given any known $C_3$-surface $X$, we next define two ways of constructing a new equivariant surface from $X$. Though these are the only operations needed for the statement of the classification theorem, other equivariant surgery constructions are required for its proof. The reader is again directed to \cite{Pohl} for a complete treatment of this story. 


\begin{definition}
Let $Y$ be a non-equivariant surface and $X$ a surface with a nontrivial order $3$ homeomorphism $\sigma\colon X\rightarrow X$. Define $\tilde{Y}:=Y\setminus D^2$, and let $D$ be a disk in $X$ so that $D$ is disjoint from each of its conjugates $\sigma^i D$. Similarly let $\tilde{X}$ denote $X$ with each of the $\sigma^i D$ removed. Choose an isomorphism $f\colon \partial \tilde{Y}\rightarrow \partial D$. We define an \textbf{equivariant connected sum} $X\#_3 Y$, by 
\[\left[\tilde{X}\sqcup\coprod_{i=0}^{2}\left(\tilde{Y}\times \{i\}\right)\right]/\sim\]
where $(y,i)\sim \sigma^i(f(y))$ for $y\in \partial \tilde{Y}$ and $0\leq i \leq 2$. We can see an example of this surgery in Figure \ref{equivcnctsum}.
\end{definition}

\begin{figure}
\begin{center}
\includegraphics[scale=.5]{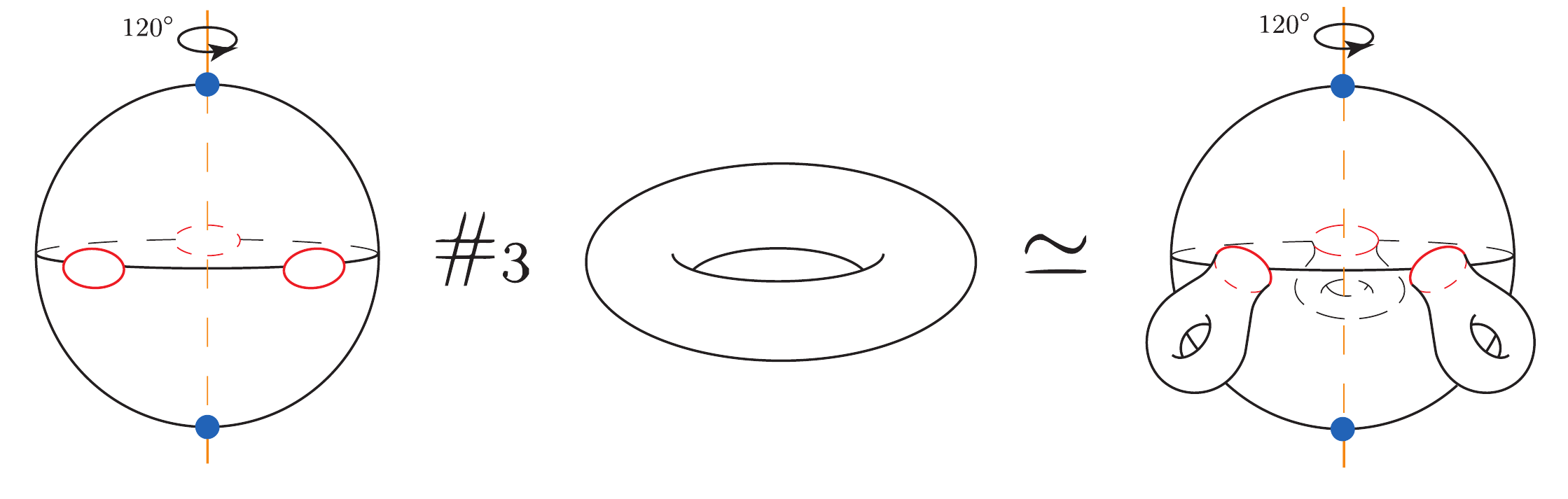}
\end{center}
\caption{\label{equivcnctsum} We can see above the result of the surgery $S^{2,1}\#_3 M_1$.}
\end{figure}

Before defining the next surgery operation, let us introduce notation for a particularly important equivariant surface. Let $D$ be a disk in $S^{2,1}$ that is disjoint from each of its conjugate disks. We define a \textbf{$C_3$-equivariant ribbon} as 
\[S^{2,1}\setminus \left(\coprod_{j=0}^2 \sigma^j D\right),\]
and we denote this space $R_{3}$. We can see $R_{3}$ depicted in the center of Figure \ref{ribbonsurgery}. Its action can be described as rotation about the orange axis. There are two fixed points in $R_3$, given by the points $a$ and $b$ in blue where the axis of rotation intersects the surface.

\begin{figure}
\begin{center}
\includegraphics[scale=.5]{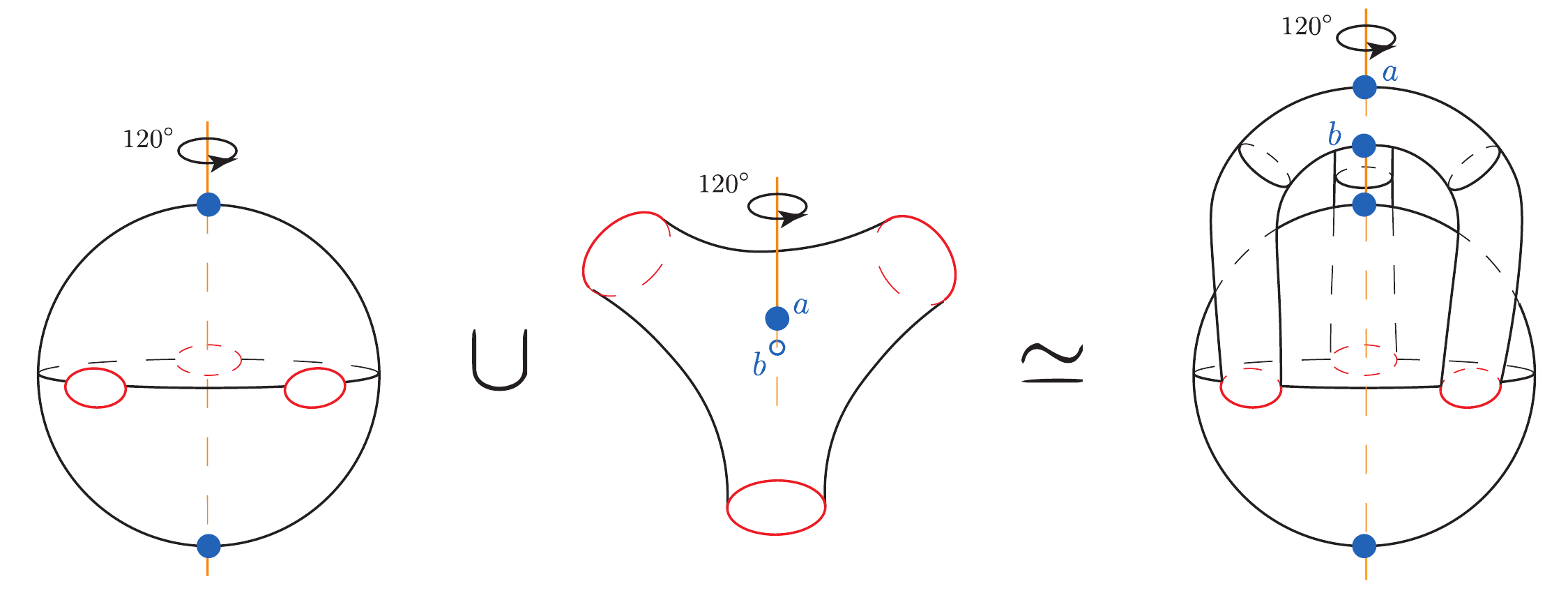}
\end{center}
\caption{\label{ribbonsurgery} We can see above the result of the surgery $S^{2,1}+[R_3]$.}
\end{figure}

\begin{definition}
Let $X$ be a surface with a nontrivial order $3$ homeomorphism $\sigma\colon X\rightarrow X$. Choose a disk $D_1$ in $X$ that is disjoint from $\sigma^jD_1$ for each $j$. Then remove each of the $\sigma^jD_1$ to form the space $\tilde{X}$. Let $D$ be the disk in $S^{2,1}$ which was removed (along with its conjugates) to form $R_3$. Choose an isomorphism $f\colon \partial D_1\rightarrow \partial D$ and extend this equivariantly to an isomorphism $\tilde{f}\colon \partial \tilde{X}\rightarrow\partial R_3$. We then define \textbf{$C_3$-ribbon surgery} on $X$ to be the space
\[\left(\tilde{X}\sqcup R_3\right)/\sim\]
where $x\sim \tilde{f}(x)$ for $x\in\partial\tilde{X}$. This is a new $C_3$-surface which we will denote $X+[R_3]$.
\end{definition}

\subsection{Results}

\begin{theorem}\label{orientableclassification}
Let $X$ be a connected, closed surface with a free action of $C_3$. Then $X$ can be constructed via one of the following surgery procedures.
\begin{enumerate}
	\item $M_{1+3g}^{\textup{free}}:= M_1^{\textup{free}}\#_3M_g$, \qquad $g\geq 0$
         \item $N_{2+3r}^{\textup{free}}\cong N_2^{\textup{free}}\#_3N_r$, \qquad $r\geq 0$
\end{enumerate}
\end{theorem}

\begin{theorem}\label{nonorientableclassification}
Let $X$ be a connected, closed, surface with a nonfree action of $C_3$. Then $X$ can be constructed via one of the following surgery procedures.
\begin{enumerate}
        \item $\Sph_{2k+3g}[2k+2]:=\left(S^{2,1}+k[R_3]\right)\#_3M_g$, \qquad $k,g\geq 0$
	\item $\Hex_{n,(3n-2)+2k+3g}[3n+2k]:=\left(\Hex_n+k[R_3]\right) \#_3M_g$, \qquad $k,g\geq 0$, $n\geq 1$
	\item $N_{4k+3r}[2k+2]:=\left(S^{2,1}+k[R_3]\right)\#_3N_r$, \qquad $r\geq 1$
	\item $N_{1+4k+3r}[1+2k]:=\left(N_1[1]+k[R_3]\right)\#_3N_r$, \qquad $k,r\geq 0$
\end{enumerate}
\end{theorem}

The notation established in the theorem statement allows for quick identification of $F(X)$ and genus. The value in brackets gives $F(X)$ for each $X$. Additionally, a lone subscript indicates the genus of the underlying surface. For example, $\Sph_4[6]$ describes an orientable surface of genus $4$ constructed from a sphere whose action has $6$ fixed points. The only exception is the class constructed from $\Hex_n$ whose notation contains two subscripts. The first describes the space $\Hex_n$ used in construction of the surface, and the second subscript indicates genus. Thus, $\Hex_{2,4}[6]$ was built from $\Hex_2$ in the equivariant surgery construction. Its underlying non-equivariant surface has genus $4$, and its $C_3$-action has $6$ fixed points. 

\begin{remark}\label{badnews}
It is important to note that for orientable surfaces, genus and the size of the fixed set do not provide enough information to distinguish between these classes. For example, $\Hex_{2,4}[6]$ and $\Sph_{4}[6]$ are non-isomorphic orientable surfaces with the same genus and number of fixed points. Nonetheless, we will see in Section \ref{computationsII} that the cohomology depends only on the genus and number of fixed points, so these spaces also have the same cohomology. 

In the case of non-orientable surfaces, fixed set size and genus do distinguish between isomorphism classes. In other words, given a non-orientable surface $X$ with specific values for $F(X)$ and $\beta(X)$, one can explicitly determine how $X$ was constructed via equivariant surgeries.
\end{remark}

\section{Cohomology Computations of Free $C_3$-surfaces}\label{computationsI}

In this section, we prove the main cohomology result for free actions by directly computing the cohomology of all free $C_3$-surfaces in $\underline{\Z/3}$-coefficients. We assume going forward that coefficients are always the constant Mackey functor $\underline{\Z/3}$, so this will be left out of the notation in favor of brevity.

\begin{theorem}\label{freethm}
  The following are true for all $g,r\geq 0$.
\begin{enumerate}\setlength{\itemsep}{.5\baselineskip}
\item $\displaystyle{H^{*,*}(M_{1+3g}^{\text{free}})\cong H^{*,*}(S^1_{\text{free}})\oplus \Sigma^{1,0}H^{*,*}(S^1_{\text{free}})\oplus \left(\Sigma^{1,0}H^{*,*}(C_3)\right)^{\oplus 2g}}$
\item $\displaystyle{H^{*,*}(N_{2+3r}^{\text{free}})\cong H^{*,*}(S^1_{\text{free}})\oplus \left(\Sigma^{1,0}H^{*,*}(C_3)\right)^{\oplus r-1}}$
\end{enumerate}  
\end{theorem}

The theorem as written depends on the equivariant surgery construction presented in Section \ref{surfacesection}, but a quick translation allows you to state the result completely in terms of its genus and whether or not the surface is orientable.

\begin{corollary}
Let $X$ be a closed and connected surface with a free action of $C_3$. 
\begin{enumerate}
\item If $X$ is orientable, then
\[H^{*,*}(X)\cong H^{*,*}(S^1_{\textup{free}})\oplus \Sigma^{1,0}H^{*,*}(S^1_{\textup{free}})\oplus \left(\Sigma^{1,0}H^{*,*}(C_3)\right)^{\oplus \frac{\beta(X)-2}{3}}.\]
\item If $X$ is non-orientable, then
\[H^{*,*}(X)\cong H^{*,*}(S^1_{\textup{free}})\oplus \left(\Sigma^{1,0}H^{*,*}(C_3)\right)^{\oplus \frac{\beta(X)-2}{3}}.\]
\end{enumerate}
\end{corollary}

Our proof will proceed as a direct computation of $H^{*,*}(X)$ in each of the two cases presented in the theorem statement. For each case, we start by computing the cohomology of the corresponding building block surface (labeled as a base case in the proof). We then perform a separate computation of the more general case when we have a nontrivial equivariant connected sum.


\begin{notation}
Going forward, we will find ourselves making frequent use of spaces of the form $Y\setminus D^2$ where $Y$ is some (closed) non-equivariant surface. For convenience we establish the notation
\[\tilde{Y}:=Y\setminus D^2\]
which will be used throughout the remainder of the paper.
\end{notation}

\begin{proof}[Proof of Theorem \ref{freethm}, Case (1) base, $M_1^{\textup{free}}$]


The rotating torus $M_1^{\text{free}}$ is isomorphic to $S^1_{\text{free}}\times S^{1,0}$. Thus there exists a cofiber sequence 
\[{S^1_{\text{free}}}_+\hookrightarrow \left(M_1^{\text{free}}\right)_+\rightarrow S^{1,0}\wedge \left({S^1_{\text{free}}}_+\right)\]
(see Figure \ref{toruscofib}).
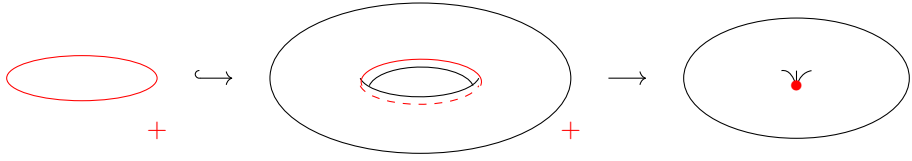
\begin{figure}
\begin{center}
	\begin{tikzpicture}
		\draw[red] (-4.5,0) ellipse (1cm and .3cm);
		\draw[red] (-3.5,-.7) node{$+$};
		\draw (0,0) ellipse (2cm and 1cm);
		\draw (-.8,0) arc (190:350:.8cm and 0.3cm);
		\draw[red,dashed] (-.8,-.1) arc (190:350:.8cm and 0.3cm);
		\draw (.7,-.1) arc (10:170:.7cm and 0.3cm);
		\draw[red] (.8,-.1) arc (-10:190:.8cm and 0.3cm);
		\draw[red] (2,-.7) node{$+$};
		\draw[right hook->] (-3,0)--(-2.5,0);
		\draw[->] (2.5,0)--(3,0);
		\draw (5,0) ellipse (1.5 cm and .8 cm);
		\draw (5,-.10) to[out=110,in=10](4.8,.1);
		\draw (5,-.10) to[out=70,in=190](5.2,.1);
		\draw (5,-.10)--(5,.1);
		\draw[red] (5,-.1) node{$\bullet$};
	\end{tikzpicture}
	\end{center}
	\caption{\label{toruscofib} The cofiber sequence ${S^1_{\text{free}}}_+ \hookrightarrow {M_1^{\text{free}}}_+\rightarrow S^{1,0}\wedge \left({S^1_{\text{free}}}_+\right)$.}
\end{figure}
For each $q$ this gives rise to a long exact sequence on cohomology 
\[\rightarrow \tilde{H}^{p,q}(S^{1,0}\wedge({S^1_{\text{free}}}_+))\rightarrow H^{p,q}(M_1^{\text{free}}) \rightarrow H^{p,q}(S^1_{\text{free}}) \xrightarrow{d^{p,q}} \tilde{H}^{p+1,q}(S^{1,0}\wedge({S^1_{\text{free}}}_+))\rightarrow.\]
Together these long exact sequences have total differential
\[d:= \bigoplus_{p,q}d^{p,q}\colon H^{*,*}\left(S^1_{\text{free}}\right)\rightarrow \tilde{H}^{*+1,*}\left(S^{1,0}\wedge \left({S^1_{\text{free}}}_+\right)\right)\]
which is shown in Figure \ref{freetorusdiff}. To compute $H^{*,*}(M_1^{\text{free}})$, we will analyze the total differential and solve the corresponding extension problem
\[0\rightarrow \coker(d)\rightarrow H^{*,*}(M_1^{\text{free}})\rightarrow \ker(d)\rightarrow 0.\]

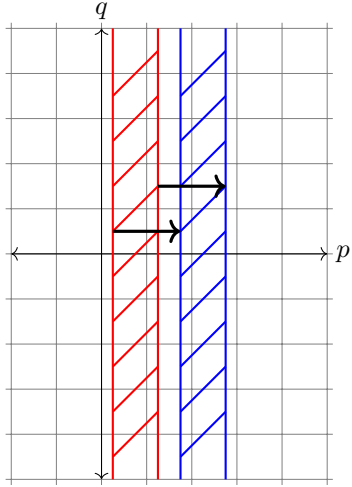
\begin{figure}
\begin{center}
	\begin{tikzpicture}[scale=0.6]
		\draw[help lines] (-2.125,-5.125) grid (5.125, 5.125);
		\draw[<->] (-2,0)--(5,0)node[right]{$p$};
		\draw[<->] (0,-5)--(0,5)node[above]{$q$};;
		\cthree{1.25}{1}{blue};
		\cthree{-.25}{1}{red};
		\draw[very thick,->] (.25,.5) -- (1.75,.5);
		\draw[very thick,->] (1.25,1.5) -- (2.75,1.5);
	\end{tikzpicture}
	\end{center}
\caption{\label{freetorusdiff} The differential $d\colon H^{*,*}\left(S^1_{\text{free}}\right)\rightarrow \tilde{H}^{*+1,*}\left(S^{1,0}\wedge \left({S^1_{\text{free}}}_+\right)\right)$.}
\end{figure} 

We can see from Figure \ref{freetorusdiff} that the only possible nonzero differentials are $d^{0,q}$ and $d^{1,q}$. Since $d$ is an $\M_3$-module map, it suffices to compute $d^{0,0}$ and $d^{1,0}$. The quotient Lemma tells us that 
\begin{align*}
H^{p,0}\left(M_1^{\text{free}}\right)&\cong H^p_{\text{sing}}\left(M_1^{\text{free}}/C_3\right) \\
&\cong H^p_{\text{sing}}(M_1)
\end{align*}
which is $\Z/3$ when $p=0,2$ and $\Z/3\oplus \Z/3$ when $p=1$. So $d^{0,0}$ and $d^{1,0}$ must be the zero map, and thus all differentials are zero by linearity. This leaves us to determine if the following extension is trivial:
\[0\rightarrow \Sigma^{1,0} H^{*,*}(S^1_{\text{free}})\rightarrow H^{*,*}(M_1^{\text{free}})\rightarrow H^{*,*}(S^1_{\text{free}})\rightarrow 0.\]
The only other possibility is a non-trivial $z$-extension from $\ker (d)$ to $\coker(d)$. This begs the question: does there exist $\alpha\in H^{0,q}(M_1^{\text{free}})$ so that $z\alpha\neq 0$?

The following composition is the identity map, implying $\pi_2^*$ is injective on cohomology:
\[S^1_{\text{free}}\xrightarrow{\cong} \text{pt}\times S^1_{\text{free}}\hookrightarrow M_1^{\text{free}} \xrightarrow{\pi_2} S^1_{\text{free}}.\]
Since $H^{0,q}(M_1^{\text{free}})$ and $H^{0,q}(S^1_{\text{free}})$ are both $\Z/3$, it must be that $\pi_2^*$ is an isomorphism in degrees $(0,q)$. Now let $\alpha\in H^{0,q}(M_1^{\text{free}})$. Then there exists $\beta\in H^{0,q}(S^1_{\text{free}})$ such that $\alpha=\pi_2^*(\beta)$. Then $z\alpha=\pi_2^*(z\beta)=0$ since $\beta\in H^{*,*}(S^1_{\text{free}})=x^{-1}\M_3/(z)$. Thus the extension is trivial, and
\[H^{*,*}(M_{1}^{\text{free}})\cong H^{*,*}(S^1_{\text{free}})\oplus \Sigma^{1,0}H^{*,*}(S^1_{\text{free}}).\]
\end{proof}

We now turn our attention to the general case. 

\begin{proof}[Proof of Theorem \ref{freethm}, Case (1) general, $M_{3g+1}^{\textup{free}}$]
%
Recall that $M_{3g+1}^{\text{free}}$ can be constructed via the equivariant connected sum: $M_1^{\text{free}}\#_3 M_g$. This construction suggests a map
\[\left(\tilde{M}_g\times C_3\right)_+ \hookrightarrow {M_{3g+1}^{\text{free}}}_+\]
whose cofiber is the $C_3$-space depicted in Figure \ref{M1hat}. We denote this space by $\hat{M}_1$. The three blue points shown in the figure are all identified, making it a single fixed point under the $C_3$-action. In order to utilize the corresponding long exact sequence on cohomology, we first need to compute $\tilde{H}^{*,*}(\hat{M}_1)$.

\begin{figure}
\begin{center}
\includegraphics[scale=.5]{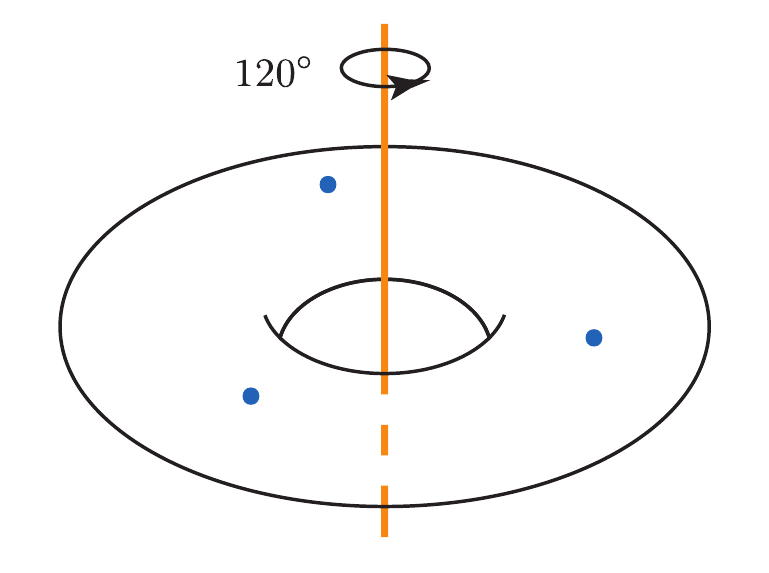}
\end{center}
\caption{\label{M1hat} The space $\hat{M}_1$ where the blue points are identified to a single point.}
\end{figure}

To do this, we use another cofiber sequence
\[{C_3}_+ \hookrightarrow {M_1^{\text{free}}}_+ \rightarrow \hat{M}_1\]
which we can extend to the cofiber sequence
\[{M_1^{\text{free}}}_+ \hookrightarrow \hat{M}_1 \rightarrow S^{1,0}\wedge \left({C_3}_+\right).\]
We next consider the long exact sequence on cohomology which has total differential $d:=\bigoplus_{p,q}d^{p,q}$, where
\[d^{p,q}\colon H^{p,q}(M_1^{\text{free}}) \rightarrow \tilde{H}^{p+1,q}(S^{1,0}\wedge {C_3}_+).\]

\begin{figure}
\begin{center}
	\begin{tikzpicture}[scale=0.6]
		\draw[help lines] (-2.125,-5.125) grid (5.125, 5.125);
		\draw[<->] (-2,0)--(5,0)node[right]{$p$};
		\draw[<->] (0,-5)--(0,5)node[above]{$q$};;
		\cthree{1.25}{1}{red};
		\cthree{-.25}{1}{red};
		\cthree{1}{0}{blue};
		\draw[very thick,->] (.25,.5) -- (1.5,.5);
	\end{tikzpicture}
	\end{center}
\caption{\label{M1hat} The differential $d\colon H^{*,*}(M_1^{\text{free}}) \rightarrow \tilde{H}^{*+1,*}(S^{1,0}\wedge {C_3}_+)$.}
\end{figure}
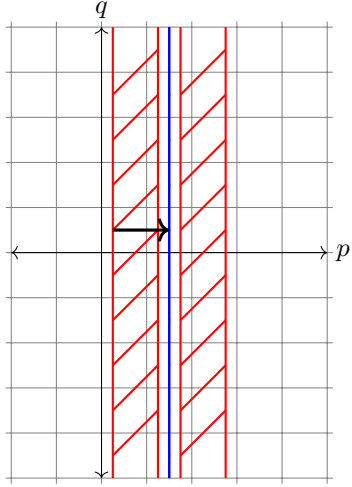 

We can see from Figure \ref{M1hat} that it suffices to compute $d^{0,0}$. The quotient lemma tells us $\tilde{H}^{0,0}(\hat{M}_1)\cong \tilde{H}^0_{\text{sing}}(\hat{M}_1/C_3)\cong \tilde{H}^0_{\text{sing}}(M_1)=0$. This means $d^{0,0}$ must be an isomorphism. Thus we conclude $d^{0,q}$ is an isomorphism for all $q$. So $\text{coker}(d)=0$ and we have 
\[\tilde{H}^{*,*}(\hat{M}_1)\cong \Sigma^{1,0}H^{*,*}(C_3)\oplus \Sigma^{1,0}H^{*,*}(S^1_{\text{free}}).\] 

Now that we know the cohomology of $\hat{M}_1$, we can return to the cofiber sequence 
\[\left(\tilde{M}_g\times C_3\right)_+ \hookrightarrow {M_{3g+1}^{\text{free}}}_+\rightarrow \hat{M}_1\]
and its corresponding long exact sequence on cohomology. For each $q$, we get an exact sequence with differential
\[d^{p,q}\colon H^{p,q}(\tilde{M}_g\times C_3)\rightarrow \tilde{H}^{p+1,q}(\hat{M}_1).\]
By Lemma \ref{times C_3}, we know $H^{*,*}(\tilde{M}_g\times C_3)\cong \Z/3[x,x^{-1}]\otimes_{\Z/3}H^*_{\text{sing}}(\tilde{M}_g)$. We can see in Figure \ref{genfreetorusdiff} that we only need to compute the differential when $p$ is $0$ or $1$.

\begin{figure}
\begin{center}
	\begin{tikzpicture}[scale=0.6]
		\draw[help lines] (-2.125,-5.125) grid (5.125, 5.125);
		\draw[<->] (-2,0)--(5,0)node[right]{$p$};
		\draw[<->] (0,-5)--(0,5)node[above]{$q$};;
		\cthree{.75}{0}{red};
		\cthree{0}{0}{red};
		\lab{.75}{$2g$}{red};
		\cthree{1}{0}{blue};
		\cthree{1.25}{1}{blue};
		\draw[very thick,->] (.5,.5) -- (1.5,.5);
		\draw[very thick,->] (1.25,1.5) -- (2.75,1.5);
	\end{tikzpicture}
	\end{center}
	\caption{\label{genfreetorusdiff} The differential $d\colon H^{*,*}(\tilde{M}_g\times C_3)\rightarrow \tilde{H}^{*+1,*}(\hat{M}_1)$.}
\end{figure}
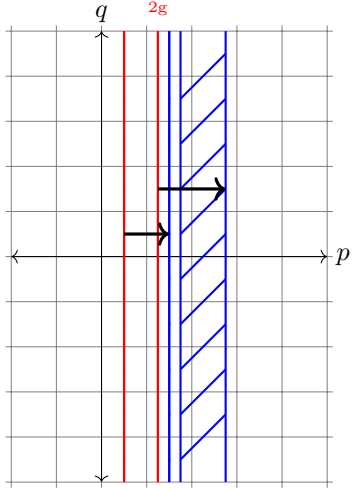 

Again, we know from the quotient lemma that
\begin{align*}
H^{p,0}(M_{3g+1}^{\text{free}})&\cong H^p_{\text{sing}}(M_{3g+1}^{\text{free}}/C_3) \\
&\cong H^p_{\text{sing}}(M_{g+1}).
\end{align*}
In particular, 
\begin{align*}
H^{0,0}(M_{3g+1}^{\text{free}})&=\Z/3 \\
H^{1,0}(M_{3g+1}^{\text{free}})&=\Z/3^{\oplus 2g+2} \\
H^{2,0}(M_{3g+1}^{\text{free}})&=\Z/3.
\end{align*}
So all differentials must be zero. Thus we are left to solve the extension problem 
\[0\rightarrow \tilde{H}^{*,*}(\hat{M}_1) \rightarrow H^{*,*}(M_{3g+1}^{\text{free}}) \rightarrow H^{*+1,*}(\tilde{M}_g\times C_3)\rightarrow 0.\]
All elements of the lower cone of $\M_3$ must act trivially on elements which are infinitely divisible by $x$. So we only need to determine if $y\alpha$ or $z\alpha$ are nonzero for $\alpha\in H^{0,q}(M_{3g+1}^{\text{free}})$. Consider the following map of cofiber sequences: 

\[\begin{tikzcd}
	{(\tilde{M}_g\times C_3)_+} & {{M_{3g+1}^{\text{free}}}_+} & {\hat{M}_1} \\
	{{C_3}_+} & {{M_1^{\text{free}}}_+} & {\hat{M}_1}
	\arrow[from=1-1, to=2-1]
	\arrow["\varphi"', from=1-2, to=2-2]
	\arrow[from=1-3, to=2-3]
	\arrow[hook, from=1-1, to=1-2]
	\arrow[from=1-2, to=1-3]
	\arrow[from=2-2, to=2-3]
	\arrow[hook, from=2-1, to=2-2]
\end{tikzcd}\]

Recall that the differential for the long exact sequence corresponding to the top cofiber sequence was shown to be zero. Moreover, in a previous computation we showed that the differential in the long exact sequence corresponding to ${M_1^{\text{free}}}_+\rightarrow \hat{M}_1\rightarrow {S^{1,0}\wedge C_3}_+$ was always surjective. This implies the differential in the long exact sequence for the bottom cofiber sequence must be 0. So we have the following commutative diagram where the rows are exact:

\[\begin{tikzcd}
	0 & {\tilde{H}^{p,q}(\hat{M}_1)} & {H^{p,q}(M_{3g+1}^{\text{free}})} & {H^{p,q}(\tilde{M}_g\times C_3)} & 0 \\
	0 & {\tilde{H}^{p,q}(\hat{M}_1)} & {H^{p,q}(M_1^{\text{free}})} & {H^{p,q}(C_3)} & 0
	\arrow[from=1-1, to=1-2]
	\arrow[from=2-1, to=2-2]
	\arrow[from=1-2, to=1-3]
	\arrow[from=1-3, to=1-4]
	\arrow[from=1-4, to=1-5]
	\arrow[from=2-4, to=2-5]
	\arrow[from=2-3, to=2-4]
	\arrow[from=2-2, to=2-3]
	\arrow["{\text{id}}"', from=2-2, to=1-2]
	\arrow["{\varphi^*}"', from=2-3, to=1-3]
	\arrow[hook, from=2-4, to=1-4]
\end{tikzcd}\]

Row exactness implies $\varphi^*$ is injective. In fact, $\varphi^*$ must be an isomorphism in dimension $(0,q)$ for all $q$ since both the domain and the codomain are $\Z/3$. Let $\alpha\in H^{0,q}(M_1^{\text{free}})$. Then $\varphi^*(y\alpha)=y\varphi^*(\alpha)$. We know $y\alpha\neq 0$ in $H^{1,q+1}(M_1^{\text{free}})$, so injectivity implies $y\varphi^*(\alpha)\neq 0$. Surjectivity in degrees $(0,q)$ implies $y\beta\neq 0$ for all nonzero $\beta\in H^{0,q}(M_{3g+1}^{\text{free}})$. Also note that $\varphi^*$ must be an isomorphism in degrees $(2,q)$ for all $q$. We know $z\varphi^*(\alpha)=\varphi^*(z\alpha)=0$ since $z\alpha=0$ in $H^{2,q+1}(M_1^{\text{free}})$. So the action of $z$ on $H^{0,q}(M_{3g+1}^{\text{free}})$ must be 0. Putting this together, we conclude
\[H^{*,*}(M_{3g+1}^{\text{free}})\cong H^{*,*}(S^1_{\text{free}})\oplus \Sigma^{1,0}H^{*,*}(S^1_{\text{free}})\oplus\left(\Sigma^{1,0}H^{*,*}(C_3)\right)^{\oplus 2g}.\]
\end{proof}

\begin{proof}[Proof of Theorem \ref{freethm}, Case (2) base, $N_2^{\textup{free}}$]
We compute the cohomology of all free non-orientable $C_3$-surfaces, starting with the free Klein bottle defined in Section \ref{surfacesection}. 

To compute the cohomology of this space, we start with the cofiber sequence
\[{S^1_{\text{free}}}_+ \hookrightarrow {N_2^{\text{free}}}_+\rightarrow N_1[1]\]
which we can see illustrated in Figure \ref{N_2 cofib seq}. Note that the cofiber of the map $f$ is isomorphic to $N_1[1]$, whose cohomology we have already seen in Example \ref{N_1[1] cohomology}. In particular, $\tilde{H}^{*,*}(N_1[1])=0$, so we can immediately conclude
\[H^{*,*}(N_2^{\text{free}})\cong H^{*,*}(S^1_{\text{free}}).\] 

\begin{figure}
\begin{center}
\includegraphics[scale=.5]{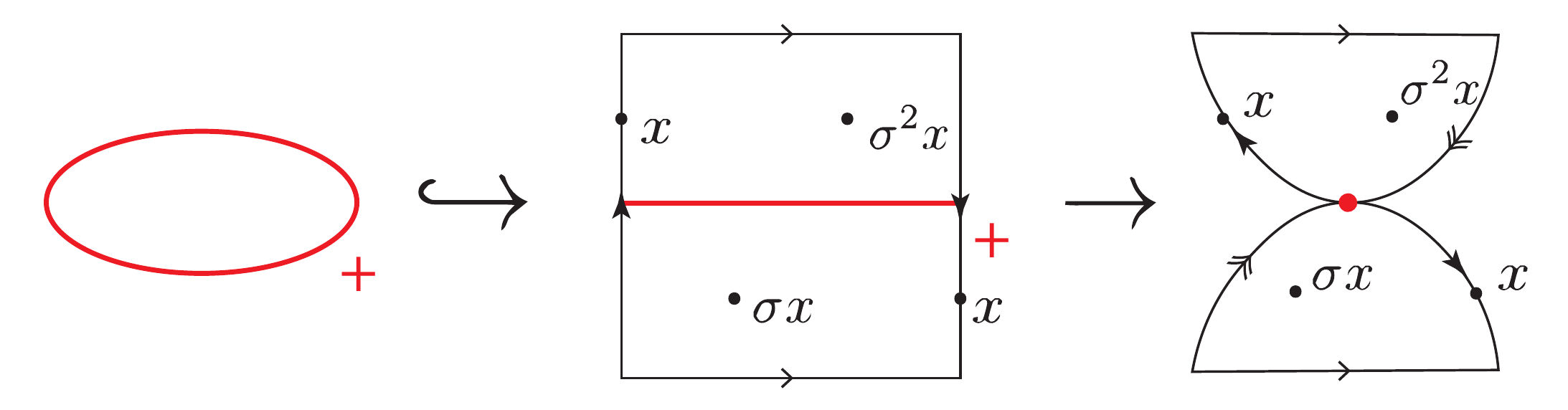}
\end{center}
\caption{\label{N_2 cofib seq} The cofiber sequence ${S^1_{\text{free}}}_+\xrightarrow{f} {N_2^{\text{free}}}_+\rightarrow\operatorname{cofib}(f)\simeq N_1[1]$.}
\end{figure}
\end{proof}


\begin{proof}[Proof of Theorem \ref{freethm}, Case (2) general, $N_{2+3r}^{\textup{free}}$]
We turn to the general case of $N_{2+3r}^{\text{free}}=N_2^{\text{free}}\#_3 N_r$ for $r\geq 1$. For this we consider the cofiber sequence 
\begin{equation}\label{freenonorientablecofiber}
\left(\tilde{N}_r\times C_3\right)_+ \hookrightarrow {N_{2+3r}^{\text{free}}}_+ \rightarrow \hat{N}_2
\end{equation}
where $\hat{N}_2$ is the mapping cone of this inclusion. To make use of this cofiber sequence, we first must compute the reduced cohomology of the space $\hat{N}_2$.

The space $\hat{N}_2$ can be realized as the cofiber of the map ${C_3}_+\hookrightarrow {N_2^{\text{free}}}_+$. Using the Puppe sequence, we can instead consider the cofiber sequence
\[{N_2^{\text{free}}}_+\rightarrow \hat{N}_2\rightarrow \Sigma^{1,0}{C_3}_+\]
and its corresponding long exact sequence on cohomology
\[\cdots \rightarrow  \tilde{H}^{p,q}(\Sigma^{1,0} {C_3}_+)\rightarrow \tilde{H}^{p,q}(\hat{N}_2)\rightarrow H^{p,q}(N_2^{\text{free}})\xrightarrow{d} \tilde{H}^{p+1,q}(\Sigma^{1,0} {C_3}_+)\rightarrow \cdots.\]

\begin{figure}
\begin{center}
	\begin{tikzpicture}[scale=0.6]
		\draw[help lines] (-2.125,-5.125) grid (5.125, 5.125);
		\draw[<->] (-2,0)--(5,0)node[right]{$p$};
		\draw[<->] (0,-5)--(0,5)node[above]{$q$};;
		\cthree{0}{1}{red};
		\cthree{1.25}{0}{blue};
		\draw[very thick,->] (.5,.5) -- (1.75,.5);
	\end{tikzpicture}
	\end{center}
\caption{\label{N2hat} The differential $d\colon H^{*,*}(N_2^{\text{free}})\rightarrow \tilde{H}^{*+1,*}(\Sigma^{1,0}{C_3}_+)$.}
\end{figure}
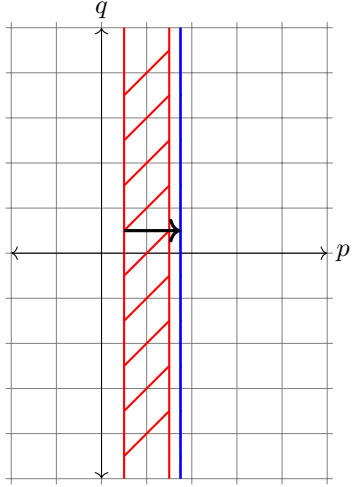

Our goal is to compute the differential of this sequence, which can be seen in Figure \ref{N2hat}. First notice that $\hat{N}_2/C_3\simeq N_2$, so by the quotient lemma we have $\tilde{H}^{p,0}(\hat{N}_2)\cong \tilde{H}^p_{\text{sing}}(N_2)$ which is $\Z/3$ for $p=1$ and 0 otherwise. In particular, $\tilde{H}^{0,0}(\hat{N}_2)=0$ which implies the differential 
\[d^{0,q}\colon H^{0,q}(N_2^{\text{free}})\rightarrow \tilde{H}^{1,q}(\Sigma^{1,0}{C_3}_+)\]
is an isomorphism for $q=0$. By linearity, we can conclude that this differential is in fact an isomorphism for all $q$. So $\coker(d)=0$ and $\tilde{H}^{*,*}(\hat{N}_2)\cong\ker(d)$. In particular,
\[\tilde{H}^{*,*}(\hat{N}_2) \cong \Sigma^{1,0}H^{*,*}(C_3).\]

We can now turn back to our original cofiber sequence (\ref{freenonorientablecofiber}) and examine its corresponding long exact sequence on cohomology
\[\cdots \rightarrow \tilde{H}^{p,q}(\hat{N}_2)\rightarrow H^{p,q}(N_{2+3r}^{\text{free}})\rightarrow H^{p,q}(\tilde{N}_r\times C_3)\xrightarrow{d} \tilde{H}^{p+1,q}(\hat{N}_2)\rightarrow\cdots. \]
As in previous examples, our strategy is to compute the total differential
\[d\colon H^{*,*}(\tilde{N}_r\times C_3)\rightarrow\tilde{H}^{*+1,*}(\hat{N}_2)\]
as seen in Figure \ref{generalfreenonorientablediff}.

Since $N_{2+3r}^{\text{free}}/C_3\simeq N_{2+r}$, we know by the quotient lemma that $H^{p,0}(N_{2+3r}^{\text{free}})\cong H^p_{\text{sing}}(N_{2+r})$ which is $\Z/3$ for $p=0$, $(\Z/3)^{r+1}$ when $p=1$, and 0 otherwise. Linearity of the differential guarantees that this map is zero in all degrees.

\begin{figure}
\begin{center}
	\begin{tikzpicture}[scale=0.6]
		\draw[help lines] (-2.125,-5.125) grid (5.125, 5.125);
		\draw[<->] (-2,0)--(5,0)node[right]{$p$};
		\draw[<->] (0,-5)--(0,5)node[above]{$q$};;
		\cthree{0}{0}{red};
		\cthree{1}{0}{red};
		\cthree{1.25}{0}{blue};
		\lab{1}{$r$}{red};
		\draw[very thick,->] (.5,.5) -- (1.75,.5);
	\end{tikzpicture}
	\end{center}
\caption{\label{generalfreenonorientablediff} The differential to (\ref{freenonorientablecofiber}), $d\colon H^{*,*}(\tilde{N}_r\times C_3)\rightarrow\tilde{H}^{*+1,*}(\hat{N}_2)$.}
\end{figure}
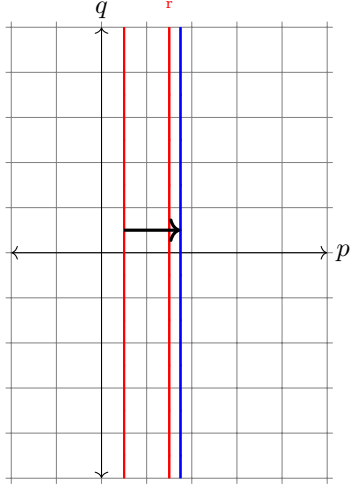

All that remains is to solve the extension problem
\[0\rightarrow \tilde{H}^{*,*}(\hat{N}_2)\rightarrow H^{*,*}(N_{2+3r}^{\text{free}})\rightarrow H^{*,*}(\tilde{N}_r\times C_3)\rightarrow 0.\]
In particular, we need to determine if $y\alpha$ is nonzero for $\alpha\in H^{0,q}(N_{2+3r}^{\text{free}})$. Consider the following map of cofiber sequences: 
\[\begin{tikzcd}
	{(\tilde{N}\times C_3)_+} & {N_{2+3r}^{\text{free}}} & {\hat{N}_2} \\
	{{C_3}_+} & {N_2^{\text{free}}} & {\hat{N}_2}
	\arrow[from=1-1, to=2-1]
	\arrow["q", from=1-2, to=2-2]
	\arrow[from=1-3, to=2-3]
	\arrow[hook, from=1-1, to=1-2]
	\arrow[from=1-2, to=1-3]
	\arrow[hook, from=2-1, to=2-2]
	\arrow[from=2-2, to=2-3]
\end{tikzcd}\]
The differential for each of the corresponding long exact sequences was found in the above computations to be zero. Thus we have the following commutative diagram where the rows are exact:
\[\begin{tikzcd}
	0 & {\tilde{H}^{*,*}(\hat{N}_2)} & {H^{*,*}(N_{2+3r}^{\text{free}})} & {H^{*,*}(\tilde{N_r}\times C_3)} & 0 \\
	0 & {\tilde{H}^{*,*}(\hat{N}_2)} & {H^{*,*}(N_{2}^{\text{free}})} & {H^{*,*}(C_3)} & 0
	\arrow[from=1-1, to=1-2]
	\arrow[from=1-2, to=1-3]
	\arrow[from=1-3, to=1-4]
	\arrow[from=1-4, to=1-5]
	\arrow[from=2-1, to=2-2]
	\arrow[from=2-2, to=2-3]
	\arrow[from=2-3, to=2-4]
	\arrow[from=2-4, to=2-5]
	\arrow["id"', from=2-2, to=1-2]
	\arrow["{q^*}"', from=2-3, to=1-3]
	\arrow[hook, from=2-4, to=1-4]
\end{tikzcd}\]
Row exactness implies that $q^*$ is injective. Moreover, for nonzero $\beta\in H^{0,q}(N_2^{\text{free}})$ we know that $y\beta\neq 0$ in $H^{1,q+1}(N_2^{\text{free}})$. Thus for any nonzero $\alpha\in H^{0,q}(N_{2+3r}^{\text{free}})$, we know $\alpha=q^*(\beta)$ for some nonzero $\beta\in H^{0,q}(N_2^{\text{free}})$. By the above remarks, it follows that $y\alpha =yq^*(\beta)=q^*(y\beta)\neq 0$. So we can conclude
\[H^{*,*}(N_{2+3r}^{\text{free}})\cong H^{*,*}(S^1_{\text{free}})\oplus \left(\Sigma^{1,0}H^{*,*}(C_3)\right)^{\oplus r}.\]
\end{proof}

\section{Cohomology Computations of Non-free $C_3$-surfaces}\label{computationsII}

\resettheoremcounters

We next prove Theorem \ref{introthm2} from the introduction. We obtain this as a corollary of the following theorem which is stated in the language of equivariant surgery. 

\begin{theorem}\label{classificationversion}
    The following are true for all $g,r,k\geq 0$.
    \begin{enumerate}\setlength{\itemsep}{.4\baselineskip}
        \item $\displaystyle{H^{*,*}(\Sph_{2k+3g}[2k+2])\cong\M_3\oplus\Sigma^{2,1}\M_3 \oplus \ebc^{\oplus 2k}\oplus \left(\Sigma^{1,0}H^{*,*}(C_3)\right)^{\oplus 2g}}$.
        \item$\displaystyle{H^{*,*}(\Hex_{n,(3n-2)+2k+3g}[3n+2k]) \cong\M_3\oplus\Sigma^{2,1}\M_3\\ \oplus \ebc^{\oplus (3n-2+2k)}\oplus\left(\Sigma^{1,0}H^{*,*}(C_3)\right)^{\oplus 2g}}$.
        \item $\displaystyle{H^{*,*}(N_{4k+3r}[2k+2])\cong \M_3\oplus\ebc^{\oplus 2k}\oplus\left(\Sigma^{1,0}H^{*,*}(C_3)\right)^{\oplus r-1}}$, \qquad $(r\geq 1)$.
        \item $\displaystyle{H^{*,*}(N_{1+4k+3r}[2k+1])\cong\M_3\oplus\ebc^{\oplus 2k}\oplus\left(\Sigma^{1,0}H^{*,*}(C_3)\right)^{\oplus r}}$.
    \end{enumerate}
\end{theorem}

Presented in this way, it is immediate that the cohomology of a $C_3$-space is determined by its construction via equivariant surgeries as stated in Theorem \ref{nonorientableclassification}. In reality, the cohomology of a given $X$ only depends on $F(X)$, $\beta(X)$, and whether or not $X$ is orientable. It can be quickly verified that the following is a consequence of Theorem \ref{classificationversion}: 

\begin{corollary}
Let $X$ be a $C_3$-surface.
\begin{enumerate}
\item If $X$ is orientable, then 
\[H^{*,*}(X)\cong \M_3\oplus \Sigma^{2,1}\M_3 \oplus \ebc^{\oplus F(X)-2}\oplus \left(\Sigma^{1,0}H^{*,*}(C_3)\right)^{\oplus \frac{\beta(X)-2F(X)+4}{3}}.\]
\item If $X$ is non-orientable and $F(X)$ is even, then
\[H^{*,*}(X)\cong \M_3\oplus \ebc^{\oplus F(X)-2}\oplus \left(\Sigma^{1,0}H^{*,*}(C_3)\right)^{\oplus \frac{\beta(X)-2F(X)+1}{3}}.\]
\item If $X$ is non-orientable and $F(X)$ is odd, then
\[H^{*,*}(X)\cong \M_3\oplus \ebc^{\oplus F(X)-1}\oplus\left(\Sigma^{1,0}H^{*,*}(C_3)\right)^{\oplus \frac{\beta(X)-2F(X)+1}{3}}.\]
\end{enumerate}
\end{corollary}

\begin{remark}
Since $H^{*,*}(X)$ is determined by $\beta(X)$, $F(X)$, and whether or not $X$ is orientable, it follows from the observations in Remark \ref{badnews} that $RO(C_3)$-graded Bredon cohomology in $\underline{\Z/3}$ coefficients is not a complete invariant. 

This is true in the case of both orientable and non-orientable surfaces. We reference the comments of Remark \ref{badnews} and observe for instance that $H^{*,*}(\Hex_{2,4}[6])\cong H^{*,*}(\Sph_{4}[6])$. An example of this can also be found in the non-orientable surfaces $N_3[2]$ and $N_1[1]$.
\end{remark}

We will prove this result by directly computing the cohomology of all non-free $C_3$-surfaces. These computations will be broken up into four classes of non-free surfaces according to our classification in Theorem \ref{nonorientableclassification}.

The techniques used in this section to determine the additive cohomology structure are similar to those used previously. However the extension problems required to understand the $\M_3$-module structure in these cases require a bit more work. We begin by considering several lemmas which will eventually aid in solving these extension problems.

\begin{lemma}\label{EBextensions}
The group $\operatorname{Ext}^{1,(0,0)}_{\M_3}\left(\ebc,\ebc\right)$ is trivial. In particular, given a short exact sequence of $\M_3$-modules
\[0\rightarrow\ebc\hookrightarrow X \twoheadrightarrow \ebc\rightarrow 0\]
it must be that $X\cong \ebc\oplus \ebc$.
\end{lemma}

\begin{proof}
We begin by constructing the first few terms of a free resolution 
\[\dots \rightarrow F_2\xrightarrow{d_2} F_1\xrightarrow{d_1} F_0\xrightarrow{\eta} \ebc\]
of $\ebc$ over $\M_3$. Recall from Example \ref{eb} that $\ebc$ is generated by $\alpha$ in degree $(2,1)$ and $\beta$ in degree $(1,1)$ with $y\beta=0$ and $z\beta=y\alpha$. 

Define $F_0=\M_3\langle a_0\rangle\oplus\M_3\langle b_0\rangle$ where $a_0$ and $b_0$ are generators of each copy of $\M_3$ in degrees $(2,1)$ and $(1,1)$, respectively. There is a surjection $\eta\colon F_0\rightarrow\ebc$ given by $a_0\mapsto \alpha$, $b_0\mapsto \beta$. Its kernel is generated by $yb_0$ and $zb_0-ya_0$, so we can construct another map $d_1\colon \M_3\langle a_1\rangle \oplus\M_3\langle b_1\rangle\rightarrow F_0$ (where $|a_1|=(2,2)$ and $|b_1|=(3,2)$) such that $d_1(a_1)=yb_0$ and $d_1(b_1)=zb_0-ya_0$. Let $F_1$ denote the module $\M_3\langle a_1\rangle\oplus \M_3\langle b_1\rangle$. 

Notice that $\ker(d_1)$ is generated by $ya_1$ and $za_1-yb_1$. For $F_2:= \M_3\langle a_2\rangle\oplus \M_3\langle b_2\rangle$ (with $|a_2|=(3,3)$ and $|b_2|=(4,3)$), we define the map $d_2\colon F_2\rightarrow F_1$ given by $d_2(a_2)=ya_1$ and $d_2(b_2)=za_1-yb_1$. We can stop here as this is the only part of the free resolution necessary to understand the first Ext group.

Next apply the functor $\Hom_{\M_3}(-,\ebc)$ of degree preserving maps to our free resolution:
\[\Hom(F_0,\ebc)\xrightarrow{d_1^*} \Hom(F_1,\ebc)\xrightarrow{d_2^*}\Hom(F_2,\ebc)\rightarrow\cdots\]
and compute $\ker (d_2^*)/\operatorname{im}(d_1^*)$. 

Let's start by computing $d_2^*$. Let $f$ be an element of $\Hom(F_1,\ebc)=\Hom(\M_3\langle a_1\rangle\oplus\M_3\langle b_1\rangle,\ebc)$. Since $f$ is determined by its values on $a_1$ and $b_1$, let us say $f(a_1)=r$ and $f(b_1)=s$ for some $r,s\in \ebc$ in degrees $(2,2)$ and $(3,2)$, respectively. Then $d_2^*(f)\in\Hom(F_2,\ebc)=\Hom(\M_3\langle a_2\rangle\oplus\M_3\langle b_2\rangle,\ebc)$ is determined by its values on $a_2$ and $b_2$. We have
\begin{align*}
d_2^*(f)(a_2)&=f(d_2(a_2))=f(ya_1)=yf(a_1)=yr, \\
d_2^*(f)(b_2)&=f(d_2(b_2))=f(za_1-yb_1)=zf(a_1)-yf(b_1)=zr-ys.
\end{align*}
So $f\in \ker(d_2^*)$ exactly when $yr=0$ and $zr-ys=0$ in $\ebc$. 

Recall that $r$ must be some element of $\tilde{H}^{2,2}(EB)$, so $yr\neq 0$ unless $r=0$. So $f\in\ker(d_2^*)$ if and only if $f(a_1)=0$. Next observe that $ys=0$ for any $s\in\tilde{H}^{3,2}(EB)$. In particular, there are two nonzero elements of $\ker (d_2^*)$; namely, the maps such that $a_1\mapsto 0$ and $b_1\mapsto \pm y\alpha$. Call these maps $f_+$ and $f_-$. We will see that both of these maps are in $\operatorname{im}(d_1^*)$, proving that $\operatorname{Ext}^{1,(0,0)}_{\M_3}\left(\ebc,\ebc\right)=0$. 

To show this, we compute $d_1^*$. Given a map $g\in \Hom(F_0,\ebc)=\Hom(\M_3\langle a_0\rangle\oplus\M_3\langle b_0\rangle,\ebc)$, we know $g$ is determined by its values on $a_0$ and $b_0$, so let's suppose $g(a_0)=t$ and $g(b_0)=u$ for some $t\in\tilde{H}^{2,1}(EB)$ and $u\in\tilde{H}^{1,1}(EB)$. Then $d_1^*(g)\in \Hom(F_1,\ebc)=\Hom(\M_3\langle a_1\rangle\oplus\M_3\langle b_1\rangle,\ebc)$ and can be determined by its values on $a_1$ and $b_1$. In particular,
\begin{align*}
d_1^*(g)(a_1)&=g(d_1(a_1))=g(yb_0)=yg(b_0)=yu \\
d_1^*(g)(b_1)&=g(d_1(b_1))=g(zb_0-ya_0)=zg(b_0)-yg(a_0) = zu-yt.
\end{align*}
Then we can see that $u=-\beta$, $t=\alpha$ defines an element of $\Hom (F_1,\ebc)$ whose image under $d_1^*$ is equal to $f_+$. Similarly, $u=\beta$, $t=-\alpha$ defines an element of $\Hom (F_1,\ebc)$ whose image under $d_1^*$ is $f_-$.
\end{proof}

\begin{lemma}\label{otherextensions}
The group $\operatorname{Ext}_{\M_3}^{1,(2,1)}(\ebc,\M_3)\cong \operatorname{Ext}^{1,(0,0)}_{\M_3}(\ebc,\Sigma^{2,1}\M_3)$ is trivial.
\end{lemma}

\begin{proof}
We begin by considering the same free resolution for $\ebc$ over $\M_3$ as in the proof of Lemma \ref{EBextensions}:
\[\cdots \rightarrow F_2\xrightarrow{d_2} F_1\xrightarrow{d_1} F_0\xrightarrow{\eta} \ebc.\]
To compute $\operatorname{Ext}^{1,(0,0)}_{\M_3}(\ebc,\Sigma^{2,1}\M_3)$, we next apply the functor $\operatorname{Hom}(-,\Sigma^{2,1}\M_3)$ of degree preserving maps to this free resolution. We claim that $\ker(d_2^*)/\operatorname{im}(d_1^*)$ is trivial. 

Let $f\in\ker(d_2^*)$. So $f$ is some map $f\colon \M_3\langle a_1\rangle\oplus\M_3\langle b_1\rangle\rightarrow \Sigma^{2,1}\M_3$. Suppose $f(a_1)=s$ and $f(b_1)=t$ for some $s,t\in\Sigma^{2,1}\M_3$. Recall from the previous lemma that $|a_1|=(2,2)$ and $|b_1|=(3,2)$. Since $f$ is degree preserving, we have that $|s|=(2,2)$ and $|t|=(3,2)$.  

Now, $d_2^*(f)$ is a map $d_2^*(f)\colon \M_3\langle a_2\rangle\oplus \M_3\langle b_2\rangle\rightarrow \Sigma^{2,1}\M_3$ given by $d_2^*(f)(a_2)=ys$ and $d_2^*(f)(b_2)=zs-yt$. Since $f\in\ker(d_2^*)$, we know $ys=0$ and $zs-yt=0$. We can see from Figure \ref{shiftedM3} that $ys=0$ only when $s=0$. Since $s=0$, the second relation simplifies to the requirement that $-yt=0$. This is true for any element of $\Sigma^{2,1}\M_3$ in degree $(3,2)$. This tells us that any function $f$ in $\ker d_2^*$ must be of the form $a_1\mapsto 0$, $b_1\mapsto t$ for any $t$ in degree $(3,2)$. 

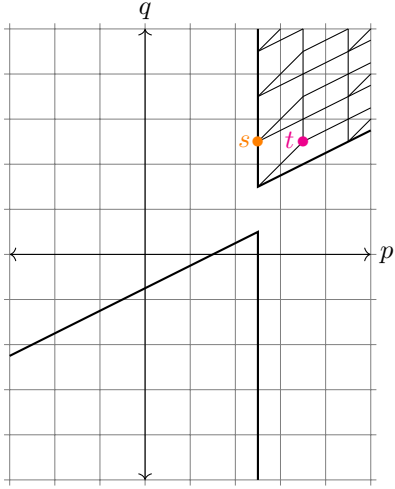
\begin{figure}[htb!]
\begin{center}
	\begin{tikzpicture}[scale=0.6]
		\draw[help lines] (-3.125,-5.125) grid (5.125, 5.125);
		\draw[<->] (-3,0)--(5,0)node[right]{$p$};
		\draw[<->] (0,-5)--(0,5)node[above]{$q$};;
		\scone{2}{1}{black};
		\draw (2.5,1.5) -- (3.5,2.5);
		\draw (2.5,2.5) -- (3.5,3.5);
		\draw (2.5,3.5) -- (3.5,4.5);
		\draw (2.5,4.5) -- (3,5);
		\draw (3.5,2.5) -- (3.5,5);
		\draw (4.5,2.5) -- (4.5,5);
		\draw (4.5,2.5) -- (5,3);
		\draw (4.5,3.5) -- (5,4);
		\draw (4.5,4.5) -- (5,5);
		\draw (2.5,2.5) -- (5,3.75);
		\draw (2.5,3.5) -- (5,4.75);
		\draw (2.5,4.5) -- (3.5,5);
		\draw (3.5,2.5) -- (5,3.25);
		\draw (3.5,3.5) -- (5,4.25);
		\draw (3.5,4.5) -- (4.5,5);
		\draw[orange] (2.5,2.5) node{$\bullet$};
		\draw[magenta] (3.5,2.5) node{$\bullet$};
		\draw[orange] (2.2,2.5) node{$s$};
		\draw[magenta] (3.2,2.55) node{$t$};
	\end{tikzpicture}
	\end{center}
	\caption{\label{shiftedM3} The top cone module structure of $\Sigma^{2,1}\M_3$. Note $ys\neq 0$ when $s\neq 0$.}
\end{figure}

It turns out that any map of this form is also in $\operatorname{im}(d_1^*)$. Let $a$ denote the generator of $\Sigma^{2,1}\M_3$. We want to show the maps $a_1\mapsto 0$ and $b_1\mapsto \pm ya$ are in the image of $d_1^*$. Define the map $g_+\colon F_0\rightarrow\Sigma^{2,1}\M_3$ given by $a_0\mapsto a$ and $b_0\mapsto 0$. Then $d_1^*(g_+)\colon F_1\rightarrow\Sigma^{2,1}\M_3$ is given by:
\begin{align*}
d_1^*(g_+)(a_1)&=g_+(d_1(a_1))=g_+(yb_0)=yg_+(b_0)=0, \\
d_1^*(g_+)(b_1)&=g_+(d_1(b_1))=g_+(zb_0-ya_0)=-ya.
\end{align*}
A similar computation shows the image of the map $g_-\colon F_0\rightarrow\Sigma^{2,1}\M_3$ given by $a_0\mapsto -a$ and $b_0\mapsto 0$ under $d_1^*$ sends $a_1$ to $0$ and $b_1$ to $ya$.

So $\ker (d_2^*)=\operatorname{im}(d_1^*)$ and $\operatorname{Ext}^{1,(2,1)}(\ebc,\M_3)$ is trivial.
\end{proof}

Together, Lemmas \ref{EBextensions} and \ref{otherextensions} tell us that given a short exact sequence of the form
\[0\rightarrow\Sigma^{2,1}\M_3\oplus \ebc^{\oplus k}\rightarrow X \rightarrow \ebc^{\oplus \ell}\rightarrow 0,\]
the $\M_3$-module $X$ must be isomorphic to $\Sigma^{2,1}\M_3\oplus \ebc^{\oplus k+\ell}$. We can even take things one step further to conclude any extension 
\[0\rightarrow\Sigma^{2,1}\M_3\oplus \ebc^{\oplus k}\rightarrow X \rightarrow \M_3\oplus\ebc^{\oplus \ell}\rightarrow 0\]
must be trivial by the projectivity of $\M_3$ as an $\M_3$-module.

\begin{lemma}\label{towerextensions}
There are no nontrivial extensions
\[0\rightarrow \Sigma^{2,1}\M_3\oplus\left(\Sigma^{1,0}H^{*,*}(C_3)\right)^{\oplus2g}\rightarrow X\rightarrow \M_3\oplus\ebc\rightarrow 0.\]
\end{lemma}

\begin{proof}
Using Lemma \ref{EBextensions} and Lemma \ref{otherextensions} as well as the fact that $\M_3$ is free, we only need to show that $\mathrm{Ext}^{1,(0,0)}\left(\ebc,\Sigma^{1,0}H^{*,*}(C_3)\right)=0$. Using the free resolution
\[\cdots \rightarrow F_2\rightarrow F_1\rightarrow F_0\rightarrow \ebc\]
defined in the proof of Lemma \ref{EBextensions}, we can see that $\operatorname{Hom}\left(F_1,\Sigma^{1,0}H^{*,*}(C_3)\right)$ must be 0. Recall $F_1$ is isomorphic to two copies of $\M_3$ generated in degrees $(3,2)$ and $(2,2)$. Since $\Sigma^{1,0}H^{*,*}(C_3)$ is concentrated in degrees $(1,q)$, there are no degree preserving maps $F_1\rightarrow \Sigma^{1,0}H^{*,*}(C_3)$. Thus $\operatorname{Ext}^{1,(0,0)}\left(\ebc,\Sigma^{1,0}H^{*,*}(C_3)\right)$ must be zero.
\end{proof}

With these lemmas, we are now ready to prove Theorem \ref{classificationversion}. For each of the four cases listed in the theorem, we will break up the computations into subcases based on the corresponding equivariant surgery decomposition of the isomorphism class. Each case is split up slightly differently, but most computations will consist of a base case (where we look at the cohomology of a surface with no ribbon surgeries or equivariant connected sum) and a separate inductive step.

We begin with Case (1) as defined in Theorem \ref{classificationversion}. Recall that the space $\Sph_{2k+3g}[2k+2]$ is orientable with $\beta(\Sph_{2k+3g}[2k+2])=2(2k+3g)$ and $F=2k+2$. In particular, $F-2=2k$ and $\frac{\beta-2F+4}{3}=2g$. We will show that 
\[H^{*,*}(\Sph_{2k+3g}[2k+2])\cong \M_3 \oplus \Sigma^{2,1} \M_3\oplus \left(\Sigma^{1,0}H^{*,*}(C_3)\right)^{\oplus 2g}\oplus \ebc^{\oplus 2k}\]
by induction on $k$. 

\begin{proof}[Proof of Theorem \ref{classificationversion}, Case (1) base, {$\Sph_{3g}[2]$}]
Recall that $\Sph_{3g}[2]:= S^{2,1}\#_3 M_g$. 

\begin{figure}
\begin{center}
\includegraphics[scale=.4]{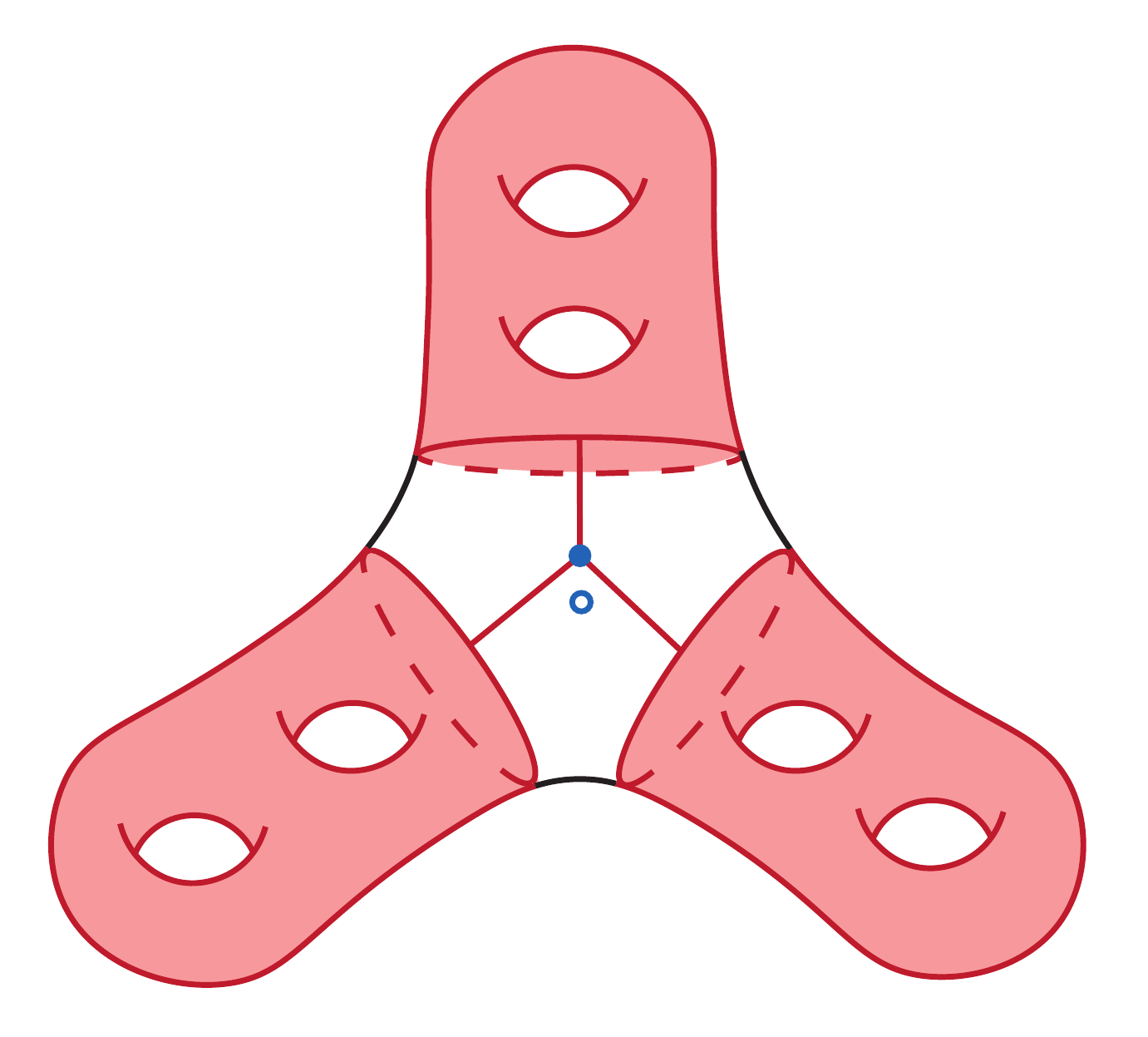}
\end{center}
\caption{\label{newspace} The space $Y\subset \Sph_{3g}[2]$ in red with $g=2$. Note $Y\simeq\displaystyle{\bigvee_{2g}S^{1,0}\wedge {C_3}_+}$.}
\end{figure}

To begin the computation, we will construct a cofiber sequence
\[Y_+\hookrightarrow \Sph_{3g}[2]_+ \rightarrow S^{2,1}\]
where $Y$ is the space in red depicted in Figure \ref{newspace}. The space $Y$ is homotopy equivalent to $\tilde{M}_g\wedge {C_3}_+$ which deformation retracts onto $\left(\bigvee_{2g} S^{1,0}\right)\wedge {C_3}_+$. This gives us a long exact sequence on cohomology:
\[\cdots \rightarrow\tilde{H}^{p,q}(S^{2,1})\rightarrow H^{p,q}(\Sph_{3g}[2])\rightarrow H^{p,q}(Y)\xrightarrow{d}\tilde{H}^{p+1,q}(S^{2,1})\rightarrow\cdots\]
which can be understood by analyzing its total differential
\[\bigoplus_{p,q}d^{p,q}\colon H^{p,q}(Y)\rightarrow \tilde{H}^{p+1,q}(S^{2,1}).\]
We plot the domain and target space of the differential below:

\begin{figure}[htb!]
\begin{center}
	\begin{tikzpicture}[scale=0.6]
		\draw[help lines] (-3.125,-5.125) grid (5.125, 5.125);
		\draw[<->] (-3,0)--(5,0)node[right]{$p$};
		\draw[<->] (0,-5)--(0,5)node[above]{$q$};;
		\cthree{1}{0}{red};
		\scone{0}{0}{red};
		\lab{1}{$2g$}{red};
		\draw (2.25,.95) node{$d^{1,0}$};
		\scone{2}{1.125}{blue};
		\draw[very thick,->] (1.5,.625) -- (2.5,.625);
	\end{tikzpicture}
	\end{center}
\end{figure}

Since the total differential is an $\M_3$-module map, it is completely determined by its values in degrees $(0,0)$ and $(1,0)$ by linearity. It is immediate that $d^{0,0}=0$ since $\tilde{H}^{1,0}(S^{2,1})=0$, and we can use the quotient lemma to determine $d^{1,0}$. In particular, $\Sph_{3g}[2]/C_3\simeq M_g$, and so $H^{2,0}(\Sph_{3g}[2])\cong\Z/3$. Therefore something in degree $(2,0)$ must be in the cokernel of $d$. This can only happen if $d^{1,0}=0$. 

We are able to determine by linearity that the total differential $\bigoplus_{a,b}d^{a,b}$ must be zero everywhere. This leaves us to solve the extension problem
\[\Sigma^{2,1}\M_3 \hookrightarrow H^{*,*}(\Sph_{3g}[2]) \twoheadrightarrow \M_3\oplus \left(\Sigma^{1,0}H^{*,*}(C_3)\right)^{\oplus 2g}.\]
By the observations in Remark \ref{M3summand}, we know $\M_3$ splits off as a summand of $H^{*,*}(\Sph_{3g}[2])$. Moreover, the submodule $\left(\Sigma^{1,0}(\Z/3[x,x^{-1}])\right)^{\oplus 2g}\subseteq \text{ker}(d)$ also splits off. To see this, let $a$ be a nonzero element of $H^{1,0}(\Sph_{3g}[2])$. We know $H^{3,1}(\Sph_{3g}[2])=0$, so $z\cdot a=0$. Since $y^2=0$ and for all nonzero $b$ in degree $(2,1)$, $y\cdot b\neq 0$, it must be the case that $y\cdot a=0$. Finally, any lower cone element must act trivially on $a$ since it is infinitely divisible by $x$. By linearity, we conclude that there cannot be any nonzero $y$, $z$, or lower cone extensions coming from $x^{\ell}a$ for any $\ell\in\Z$. 

Thus we can conclude the extension is trivial, and
\[H^{*,*}(\Sph_{3g}[2])\cong \M_3\oplus \Sigma^{2,1}\M_3\oplus\left(\Sigma^{1,0}H^{*,*}(C_3)\right)^{\oplus 2g}.\]
\end{proof}

We next proceed to the inductive step, assuming that 
\[H^{*,*}(\Sph_{2k+3g}[2k+2])\cong \M_3\oplus \Sigma^{2,1}\M_3\oplus \left(\Sigma^{1,0}H^{*,*}(C_3)\right)^{\oplus 2g}\oplus\ebc^{\oplus 2k}\]
for some $k\geq 0$. Let's use this assumption to compute the cohomology of $\Sph_{2(k+1)+3g}[2(k+1)+2]$. 

\begin{proof}[Proof of Theorem \ref{classificationversion}, Case (1) inductive step on $k$, {$\Sph_{2k+3g}[2k+2]$}]
We proceed by considering the cofiber of a map
\begin{equation}\label{seq1}
EB_+\rightarrow \Sph_{2(k+1)+3g}[2(k+1)+2]_+
\end{equation}
which we define below. The cofiber will be homotopy equivalent to $\Sph_{2k+3g}[2k+2]\vee EB$. To see this, first notice that $\Sph_{2k+3g}[2k+2]$ has at least $2$ fixed points for any $k\geq 0$. Construct $\Sph_{2(k+1)+3g}[2(k+1)+2]$ by performing $C_3$-ribbon surgery on $\Sph_{2k+3g}[2k+2]$ in a neighborhood of one of these fixed points. Then construct the map $EB\rightarrow\Sph_{2(k+1)+3g}[2(k+1)+2]$ by sending $EB$ into this copy of $R_3$ used to construct $\Sph_{2(k+1)+3g}[2(k+1)+2]$ from $\Sph_{2k+3g}[2k+2]$. Figure \ref{genericebcofiberseq} shows the cofiber of such a map.

Next notice that this cofiber is homotopy equivalent to the space shown in Figure \ref{cofiberstuff} which is homotopy equivalent to $\Sph_{2k+3g}[2k+2]\vee EB$.

\begin{figure}
\begin{center}
\includegraphics[scale=.35]{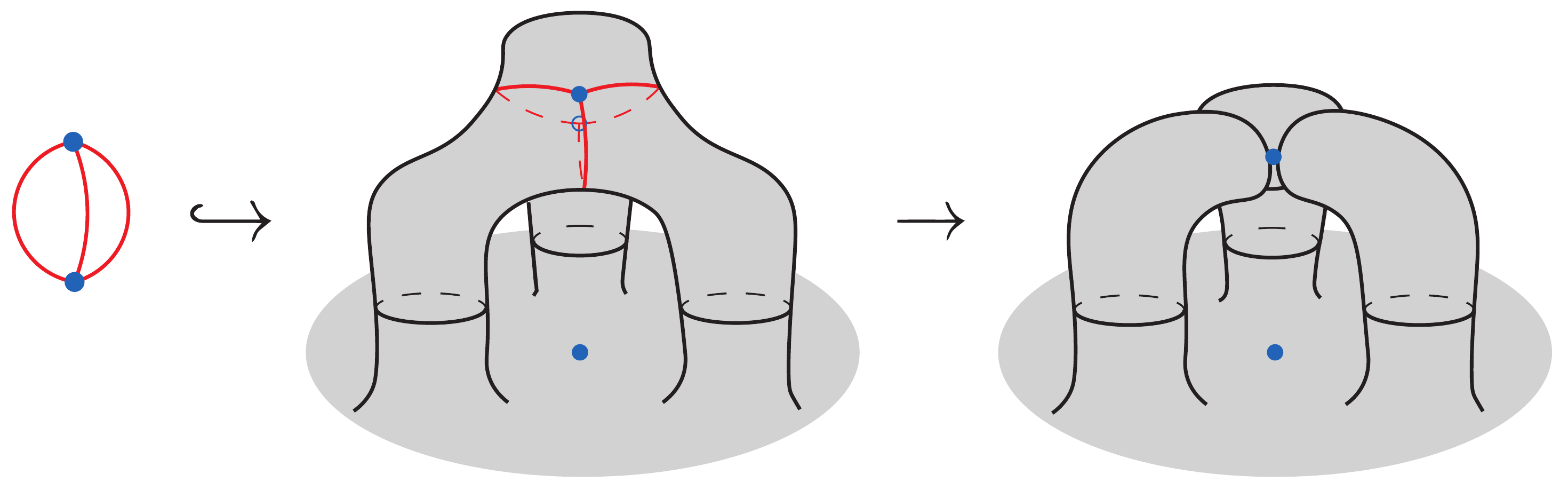}
\end{center}
\caption{\label{genericebcofiberseq} The cofiber sequence corresponding to (\ref{seq1}).}
\end{figure}

\begin{figure}
\begin{center}
\includegraphics[scale=.4]{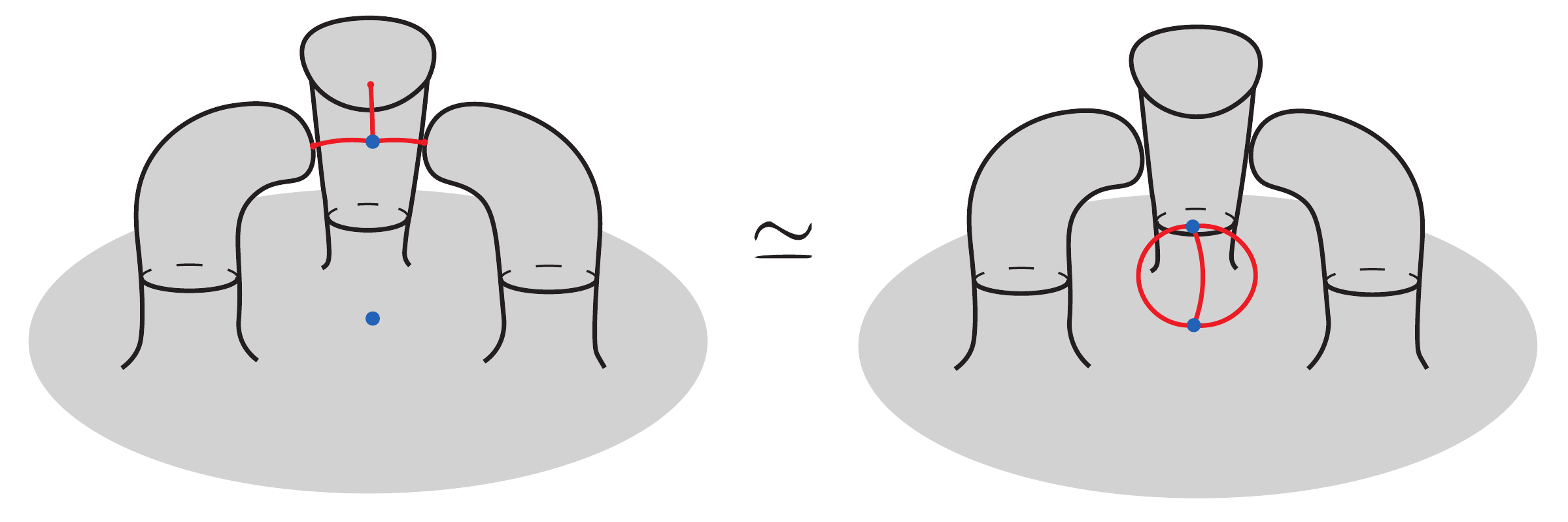}
\end{center}
\caption{\label{cofiberstuff} Up to homotopy, the cofiber of (\ref{seq1}) is $\Sph_{2(k+1)+3g}[2(k+1)+2]\vee EB$.}
\end{figure}

This cofiber sequence gives us a long exact sequence on cohomology given by 
\[\rightarrow H^{p,q}(\Sph_{2(k+1)+3g}[2(k+1)+2])\rightarrow H^{p,q}(EB)\xrightarrow{d}\tilde{H}^{p+1,q}(\Sph_{2k+3g}[2k+2]\vee EB)\rightarrow\]
As in previous examples, we can understand $\tilde{H}^{*,*}(\Sph_{2(k+1)+3g}[2(k+1)+2])$ by computing the total differential 
\[d\colon H^{*,*}(EB)\rightarrow \tilde{H}^{*+1,*}(\Sph_{2k+3g}[2k+2]\vee EB).\]
The domain and target space of this differential is shown on the $(p,q)$-axis in Figure \ref{M2k+3gdiff}. 

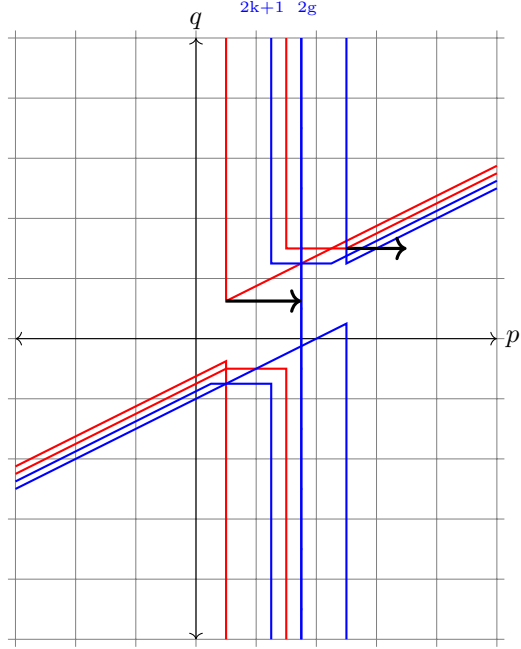
\begin{figure}[htb!]
\begin{center}
	\begin{tikzpicture}[scale=0.8]
		\draw[help lines] (-3.125,-5.125) grid (5.125, 5.125);
		\draw[<->] (-3,0)--(5,0)node[right]{$p$};
		\draw[<->] (0,-5)--(0,5)node[above]{$q$};;
		\eb{1}{1}{red};
		\scone{0}{.125}{red};
		\eb{.75}{.75}{blue};
		\cthree{1.25}{0}{blue};
		\scone{2}{.75}{blue};
		\lab{.6}{$2k+1$}{blue};
		\lab{1.35}{$2g$}{blue};
		\draw[very thick,->] (2.5,1.5) -- (3.5,1.5);
		\draw[very thick,->] (0.5,0.625) -- (1.75,0.625);
	\end{tikzpicture}
\end{center}
\caption{\label{M2k+3gdiff} The differential for the long exact sequence corresponding to (\ref{seq1}).}
\end{figure}

To compute this differential, it suffices to determine its value in degrees $(0,0)$, $(1,1)$, and $(2,2)$. 

First observe that $\Sph_{2(k+1)+3g}[2(k+1)+2]/C_3\simeq M_g$. This follows from the fact that for any $C_3$-space $X$ and non-equivariant space $Y$, $\left(X+k[R_3]\#_3Y\right)/C_3\cong \left(X/C_3\right)\#Y$. In this case, we have $\Sph_{2(k+1)+3g}[2(k+1)+2]\cong S^{2,1}+k[R_3]\#_3M_g$, so $\Sph_{2(k+1)+3g}[2(k+1)+2]/C_3\simeq M_g$. The quotient lemma then tells us that $H^{1,0}(\Sph_{2(k+1)+3g}[2(k+1)+2])\cong \left(\Z/3\right)^{\oplus 2g}$. In particular, $d^{0,0}$ must be 0.

We saw in Example \ref{eb} that $\ebc\cong\M_3\langle \alpha,\beta\rangle/(y\beta, y\alpha-z\beta)$ where $\alpha$ is in degree $(2,1)$ and $\beta$ is in degree $(1,1)$. There is nothing for $d_1^{2,1}$ to hit, so $d(\alpha)=0$. Moreover, $0=yd(\alpha)=d(y\alpha)=d(z\beta)$. If $d(\beta)\neq 0$, then linearity of $d$ would imply $d(z\beta)\neq 0$. This tells us the total differential must be 0.

This leaves us to solve the extension problem 
\[\left(\Sigma^{1,0}H^{*,*}(C_3)\right)^{\oplus 2g}\oplus\ebc^{\oplus 2k+1}\oplus \Sigma^{2,1}\M_3\hookrightarrow H^{*,*}(\Sph_{2(k+1)+3g}[2(k+1)+2])\twoheadrightarrow \M_3\oplus\ebc.\]
We can then use Lemmas \ref{EBextensions}, \ref{otherextensions}, and \ref{towerextensions} to determine that there can be no non-trivial extensions. Thus finally we have that
\[H^{*,*}(\Sph_{2(k+1)+3g}[2(k+1)+2])\cong \M_3\oplus \Sigma^{2,1}\M_3\oplus \left(H^{*,*}(C_3)\right)^{\oplus 2g} \oplus \ebc^{\oplus 2(k+1)},\]
and the result holds by induction.
\end{proof}

This ends the computation for Case (1). Our next goal will be to compute the cohomology of the space $\Hex_{n,3n-2+2k+3g}[3n+2k]$ for all $n\geq 1$ and $k,g\geq 0$. 

First recall that $\Hex_{n,3n-2+2k+3g}[3n+2k]:=\left(\Hex_n+k[R_3]\right)\#_3M_g$ with $\beta$-genus $2(3n-2+2k+3g)$ and $F=3n+2k$. Therefore $F-2=3n+2k-2$ and $(\beta-2F+4)/3=2g$. So we will work towards proving the following:
\[H^{*,*}(\Hex_{n,3n-2+2k+3g}[3n+2k])\cong \M_3\oplus\Sigma^{2,1}\M_3\oplus\ebc^{\oplus(3n-2+2k)}\oplus\left(\Sigma^{1,0}H^{*,*}(C_3)\right)^{\oplus 2g}.\]

This will be done in several steps. First we consider the base case with $g=k=0$ and $n=1$. Then we confirm that the result holds for $n=1$, $g=0$, and $k\geq 0$. The next step will be to induct on $n$ and compute cohomology in the case $n\geq 1$, $g=0$, and $k\geq 0$. The final step will be to consider the $g\geq 0$ case.

\begin{proof}[Proof of Theorem \ref{classificationversion}, Case (2) base, $\Hex_1$]

There is a cofiber sequence
\begin{equation}\label{M1[3] seq 1}
EB \hookrightarrow \Hex_1\rightarrow S^{2,1}
\end{equation}
which we can see depicted in Figure \ref{M1_3_cofibseq}. This gives a long exact sequence on cohomology
\[\cdots \rightarrow \tilde{H}^{p,q}(S^{2,1}) \rightarrow \tilde{H}^{p,q}(\Hex_1)\rightarrow \tilde{H}^{p,q}(EB)\xrightarrow{d_1^{p,q}} \tilde{H}^{p+1,q}(S^{2,1})\rightarrow\cdots\]
which can be understood by computing its total differential 
\[\bigoplus_{p,q}d_1^{p,q}\colon \tilde{H}^{*+1,*}(EB)\rightarrow \tilde{H}^{*,*}(S^{2,1})\]
shown in Figure \ref{M1[3]diff1}.

\begin{figure}
\begin{center}
\includegraphics[scale=.4]{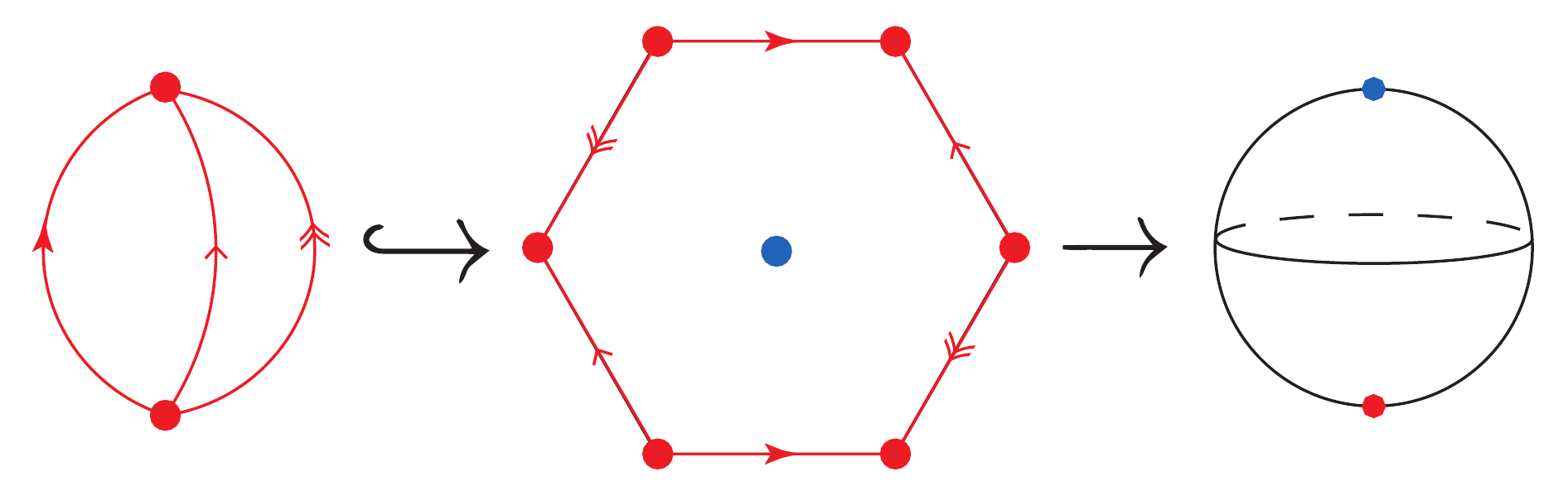}
\end{center}
\caption{\label{M1_3_cofibseq} The cofiber sequence $EB\hookrightarrow \Hex_1 \rightarrow S^{2,1}$.}
\end{figure}

\begin{figure}
\begin{center}
	\begin{tikzpicture}[scale=0.6]
		\draw[help lines] (-3.125,-5.125) grid (5.125, 5.125);
		\draw[<->] (-3,0)--(5,0)node[right]{$p$};
		\draw[<->] (0,-5)--(0,5)node[above]{$q$};;
		\eb{1}{1}{red};
		\scone{2}{1.25}{blue};
		\draw[very thick,->] (1.5,-.5) -- (2.5,-.5);
		\draw[very thick,->] (1.5,1.75) -- (2.5,1.75);
	\end{tikzpicture}
	\end{center}
	\caption{\label{M1[3]diff1} The differential for the long exact sequence corresponding to (\ref{M1[3] seq 1}).}
\end{figure}

Similar reasoning to that of the last example tells us that the total differential to this cofiber sequence must be zero. In particular, we can use the module structure $\ebc\cong\M_3\langle \alpha,\beta\rangle/(y\beta, y\alpha-z\beta)$ and the fact that $d^{2,1}(\alpha)=0$ to determine that $d^{p,q}$ must be zero for all $(p,q)$.

Since the differential is identically zero, we know $\ker(d)=\ebc$ and $\coker(d)=\Sigma^{2,1}\M_3$. We now have to solve the extension problem 
\[0\rightarrow \Sigma^{2,1}\M_3 \rightarrow \tilde{H}^{*,*}(\Hex_1)\rightarrow \ebc\rightarrow 0.\]
The above sequence is split as a consequence of Lemmas \ref{EBextensions} and \ref{otherextensions}, and we have
\[\tilde{H}^{*,*}(\Hex_1)\cong \Sigma^{2,1}\M_3 \oplus \ebc.\]
As per the observations in Remark \ref{M3summand}, it follows that
\[H^{*,*}(\Hex_1)\cong \M_3\oplus \Sigma^{2,1}\M_3\oplus \ebc.\]
\end{proof}

\begin{proof}[Proof of Theorem \ref{classificationversion}, Case (2) inductive step on $k$, {$\Hex_{1,1+2k}[2k+3]$}]
Next assume that for some $k\geq 0$ and $g=0$, the cohomology of $\Hex_{1,1+2k}[2k+3]$ is as stated in Theorem \ref{classificationversion}, and we will show that it holds true for $\Hex_{1,1+2(k+1)}[2(k+1)+3]$. 

Consider the cofiber sequence
\[EB\hookrightarrow \Hex_{1,1+2(k+1)}[2(k+1)+3]\rightarrow \Hex_{1,1+2k}[2k+3]\vee EB\]
whose corresponding long exact sequence on cohomology has differential 
\[d\colon \ebc \rightarrow\tilde{H}^{*+1,*}(\Hex_{1,1+2k}[3+2k]\vee EB).\]
To compute this differential, we consult Figure \ref{orientinducdiff2} and observe that $\alpha\in\ebc$ in degree $(2,1)$ must map to 0 as there is nothing in degree $(3,1)$. However using a similar argument to that in the base case, it must be that $d(\beta)=0$ by linearity. In particular, the total differential is zero.

Moreover, all extensions are trivial as a consequence of Lemmas \ref{EBextensions} and \ref{otherextensions}. From this, we can easily see what $\tilde{H}^{*,*}(\Hex_{1,1+2(k+1)}[2(k+1)+3])$ must be. Thus we have
\[H^{*,*}(\Hex_{1,1+2(k+1)}[2(k+1)+3])\cong \M_3\oplus\Sigma^{2,1}\M_3\oplus \ebc^{\oplus 2k+3}\]
as desired.
\end{proof}

\begin{figure}
\begin{center}
	\begin{tikzpicture}[scale=0.6]
		\draw[help lines] (-3.125,-5.125) grid (5.125, 5.125);
		\draw[<->] (-3,0)--(5,0)node[right]{$p$};
		\draw[<->] (0,-5)--(0,5)node[above]{$q$};;
		\eb{1.25}{1}{blue};
		\lab{1.25}{$2k+2$}{blue};
		\scone{2}{1.25}{blue};
		\eb{1}{1.125}{red};
	\end{tikzpicture}
	\end{center}
	\caption{\label{orientinducdiff2} The differential $d^{p,q}\colon \ebc^{p,q} \rightarrow \tilde{H}^{p+1,q}(\Hex_{1,1+2k}[2k+3]\vee EB)$.}
\end{figure}

\begin{proof}[Proof of Theorem \ref{classificationversion}, Case (2) inductive step on $n$, {$\Hex_{n,3n-2+2k}[3n+2k]$}]
Next, we compute the groups $H^{*,*}(\Hex_{n,3n-2+2k}[3n+2k])$, when $n\geq 1$. The $n=1$ case has been completed with the computation of $H^{*,*}(\Hex_{1,1+2k}[3+2k])$ in the previous step. 

Assume for some $n\geq 1$ that
\[H^{*,*}(\Hex_{n,3n-2+2k}[3n+2k])\cong \M_3\oplus \Sigma^{2,1}\M_3\oplus \ebc^{\oplus (3n-2+2k)}.\]

There is a cofiber sequence
\[Y_+ \hookrightarrow \Hex_{n+1,3(n+1)-2+2k}[3(n+1)+2k]_+ \rightarrow \Hex_{n,3n-2+2k}[3n+2k]\]
where $Y$ is the space depicted in Figure \ref{hexinduction}. This space is homotopy equivalent to $EB\vee EB\vee EB$. As usual, we want to consider the differential in the corresponding long exact sequence on cohomology:
\[d\colon H^{*,*}(Y) \rightarrow \tilde{H}^{*,*}(\Hex_{n,3n-2+2k}[3n+2k]).\] 
The spaces $H^{*,*}(Y)$ and $\tilde{H}^{*,*}(\Hex_{n,3n-2+2k}[3n+2k])$ are shown in Figure \ref{Hexndiff}.

\begin{figure}
\begin{center}
\includegraphics[scale=.5]{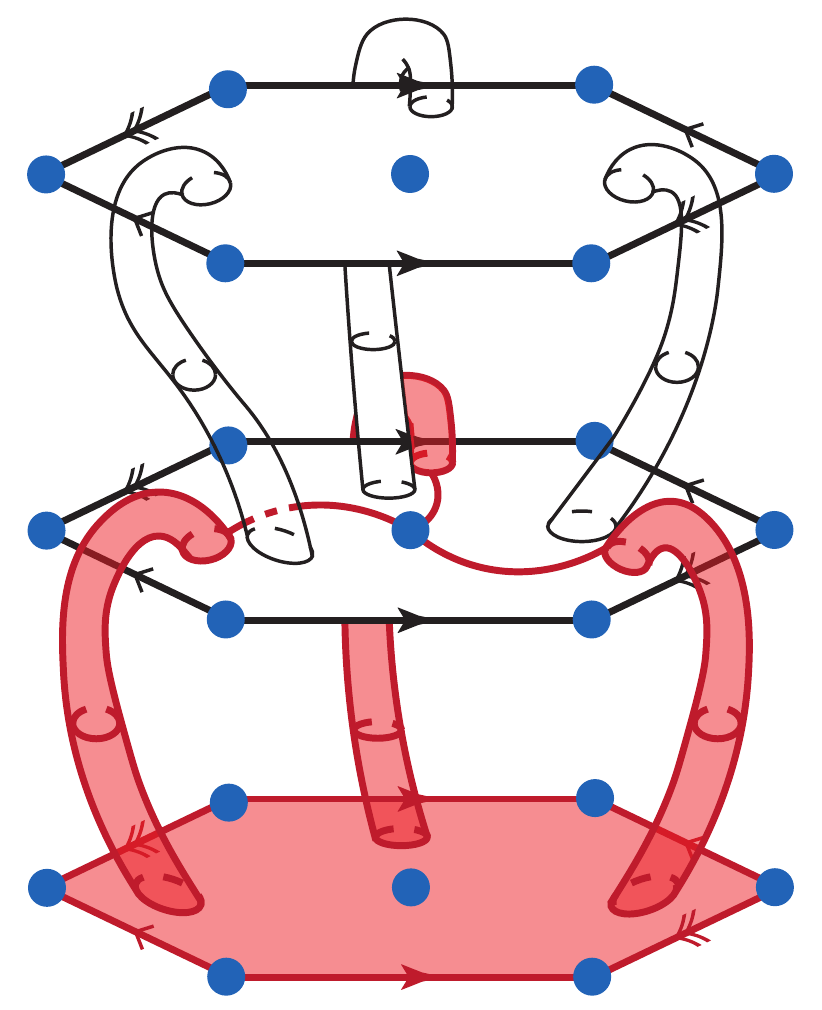}
\end{center}
\caption{\label{hexinduction} The space $Y\simeq EB\vee EB\vee EB$ is shown in red in the case $n=2$, $k=0$.}
\end{figure}

\begin{figure}
\begin{center}
	\begin{tikzpicture}[scale=0.6]
		\draw[help lines] (-3.125,-5.125) grid (5.125, 5.125);
		\draw[<->] (-3,0)--(5,0)node[right]{$p$};
		\draw[<->] (0,-5)--(0,5)node[above]{$q$};;
		\scone{0}{0}{red};
		\scone{2.125}{.865}{blue};
		\eb{.865}{1.125}{red};
		\eb{1.2}{1}{blue};
		\lab{.8}{$3$}{red};
		\draw[blue] (1+1/2,-5.5) node{\tiny{$3n-2+2k$}};
	\end{tikzpicture}
	\end{center}
	\caption{\label{Hexndiff} The spaces $H^{*,*}(Y)$ and $\tilde{H}^{*,*}(\Hex_{n,3n-2+2k}[3n+2k])$.}
\end{figure}
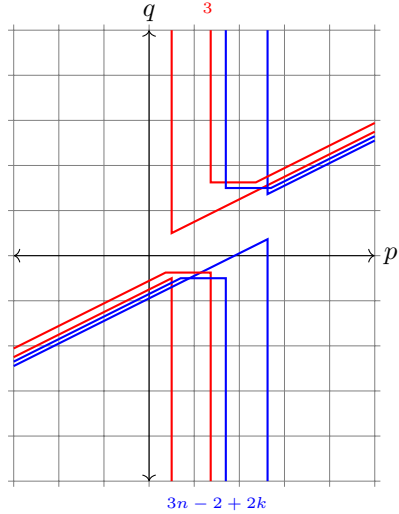

Since $d$ is an $\M_3$-module map, we only need to consider the value of $d$ in degrees $(0,0)$, $(1,1)$, and $(2,1)$. The quotient lemma guarantees that $d^{0,0}=0$. Since there is nothing in degree $(3,1)$, it also must be the case that $d^{2,1}=0$. A similar strategy from previous examples utilizing the module structure of $\ebc$ guarantees that in fact all differentials must be 0. This leaves us with the extension problem 
\[\Sigma^{2,1}\M_3\oplus\ebc^{\oplus (3n-2+2k)}\hookrightarrow H^{*,*}(\Hex_{n+1,3(n+1)-2+2k}[3(n+1)+2k])\twoheadrightarrow \M_3\oplus\ebc^{\oplus 3}.\]
We know from Lemmas \ref{EBextensions} and \ref{otherextensions} that this extension is trivial. Thus we have
\[H^{*,*}(\Hex_{n,3n-2+2k}[3n+2k])\cong \M_3\oplus\Sigma^{2,1}\M_3 \oplus \ebc^{\oplus 3n-2+2k}.\]
\end{proof}

\begin{proof}[Proof of Theorem \ref{classificationversion}, Case (2) inductive step on $g$, {$\Hex_{n,3n-2+2k+3g}[3n+2k]$}]
The last case to be considered is $\Hex_{n,3n-2+2k+3g}[3n+2k]$ when $g>0$. For this we construct the cofiber sequence
\[Y_+\hookrightarrow \Hex_{n,3n-2+2k+3g}[3n+2k]_+\rightarrow \Hex_{n,3n-2+2k}[3n+2k]\]
where $Y$ is the space in red in Figure \ref{hexexcofibseq}. Recall that $Y\simeq \left(\bigvee_{2g}S^{1,0}\right)\wedge {C_3}_+$. So the long exact sequence corresponding to this cofiber sequence has differential
\[d^{p,q}\colon H^{p,q}(Y)\rightarrow \tilde{H}^{*,*}(\Hex_{n,3n-2+2k}[3n+2k]).\]
This differential can be see in Figure \ref{gnonzerocase}.

\begin{figure}
\begin{center}
\includegraphics[scale=.35]{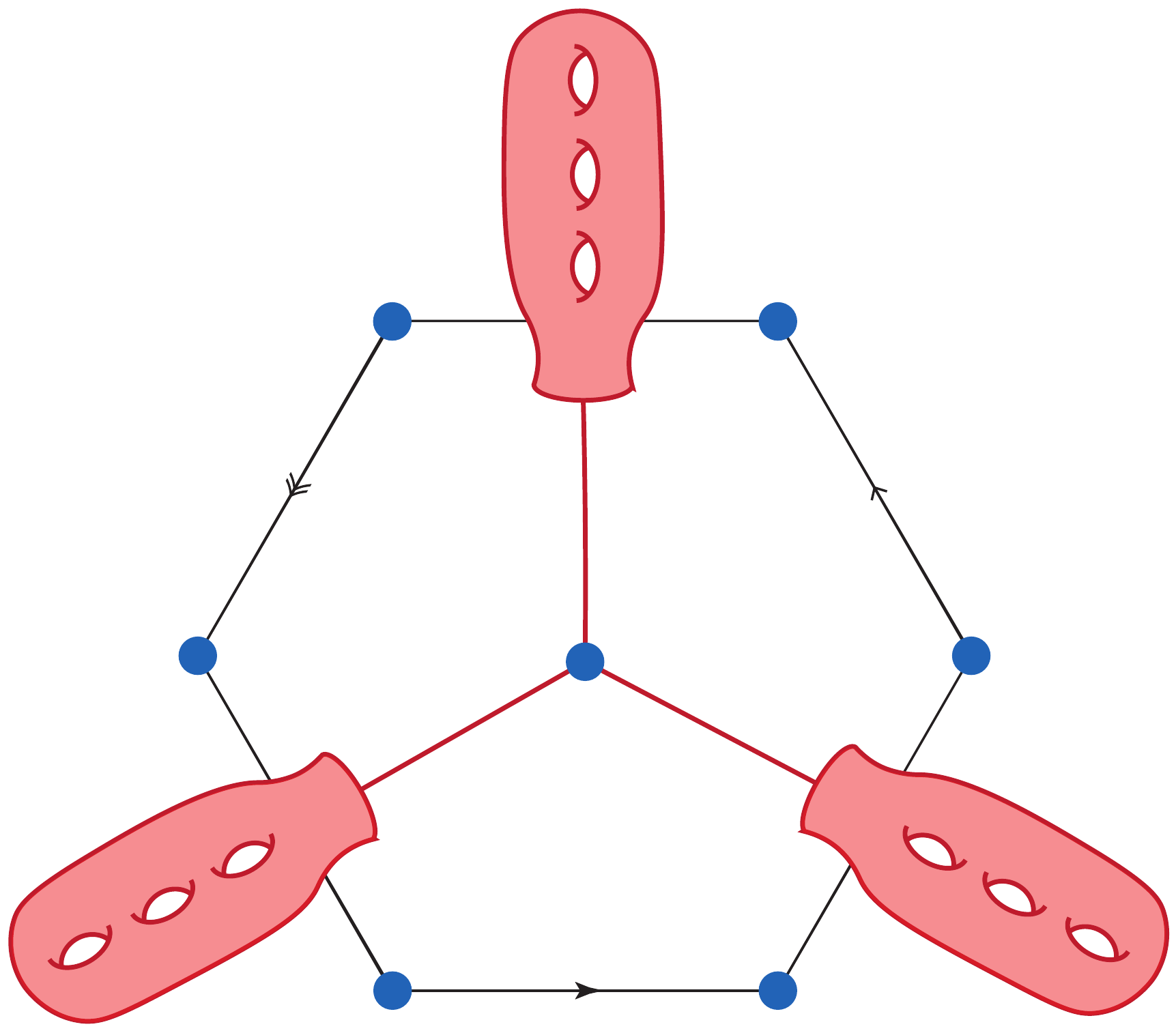}
\end{center}
\caption{\label{hexexcofibseq} The space $Y\simeq \bigvee_{2g}S^{1,0}\wedge {C_3}_+$, shown in red for $n=1$,$k=0$,$g=3$.}
\end{figure}

\begin{figure}
\begin{center}
	\begin{tikzpicture}[scale=0.6]
		\draw[help lines] (-3.125,-5.125) grid (5.125, 5.125);
		\draw[<->] (-3,0)--(5,0)node[right]{$p$};
		\draw[<->] (0,-5)--(0,5)node[above]{$q$};;
		\scone{0}{0}{red};
		\cthree{1}{0}{red};
		\lab{1}{$2g$}{red};
		\eb{1.2}{1}{blue};
		\draw[blue] (1+1/2,-5.5) node{\tiny{$3n-2+2k$}};
		\scone{2}{1.125}{blue};
		\draw[very thick,->] (1.5,.625) -- (2.5,.625);
	\end{tikzpicture}
	\end{center}
	\caption{\label{gnonzerocase} The differential $d\colon H^{*,*}(Y)\rightarrow\tilde{H}^{*,*}(\Hex_{n,3n-2+2k}[3n+2k])$.}
\end{figure}

Since there is nothing for it to hit, we can easily observe that $d^{0,0}=0$. The quotient lemma additionally allows us to conclude $d^{1,0}=0$, and thus $d^{1,q}=0$ by linearity. In particular, the total differential is zero.

We then turn to solve the extension problem 
\[0\rightarrow \coker(d)\rightarrow H^{*,*}(\Hex_{n,3n-2+2k+3g}[3n+2k])\rightarrow \ker(d)\rightarrow 0\]
where $\coker(d)=\ebc^{\oplus 3n-2+2k}\oplus\Sigma^{2,1}\M_3$ and $\ker(d)=\M_3\oplus\left(\Sigma^{1,0}H^{*,*}(C_3)\right)^{\oplus 2g}$. We again recall Remark \ref{M3summand} and observe that $\M_3\subseteq \ker(d)$ must split off as a summand of $H^{*,*}(\Hex_{n,3n-2+2k+3g}[3n+2k])$. A similar argument to that of the base case in isomorphism class (1) of Theorem \ref{classificationversion} 
guarantees that $\left(\Sigma^{1,0}H^{*,*}(C_3)\right)^{\oplus 2g}$ in $\ker (d)$ must split off as well.

Thus we can conclude the extension is trivial, and
\[H^{*,*}(\Hex_{n,3n-2+2k+3g}[3n+2k])\cong \M_3\oplus \Sigma^{2,1}\M_3\oplus \ebc^{\oplus 3n-2+2k}\oplus \left(\Sigma^{1,0}H^{*,*}(C_3)\right)^{\oplus 2g}.\]
\end{proof}

The third isomorphism class of surface presented in Theorem \ref{classificationversion} is $N_{4k+3r}[2k+2]\cong S^{2,1}+k[R_3]\#_3N_r$. To compute the cohomology of this family of isomorphism classes, we need only induct on $r$ since $H^{*,*}(S^{2,1}+k[R_3])$ is determined by Case (1).

\begin{proof}[Proof of Theorem \ref{classificationversion}, Case (3) inductive step on $r$, {$N_{4k+3r}[2k+2]$}]
The space $N_{4k+3r}[2k+2]$ (assuming $r\geq 1$) is non-orientable with $\beta$-genus $4k+3r$ and $F=2k+2$. Then $F-2=2k$ and $(\beta-2F+1)/3=r-1$. So our goal is to show
\[H^{*,*}(N_{4k+3r}[2k+2])\cong \M_3\oplus \ebc^{\oplus 2k}\oplus \left(\Sigma^{1,0}H^{*,*}(C_3)\right)^{\oplus r-1}.\]
We begin with the cofiber sequence
\[\left(\tilde{N}_r\times C_3\right)_+ \hookrightarrow N_{4k+3r}[2k+2]_+ \rightarrow \Sph_{2k}[2k+2]\vee EB.\]
This gives us the following long exact sequence on cohomology
\[\cdots \rightarrow H^{p,q}(N_{4k+3r}[2k+2])\rightarrow H^{p,q}(\tilde{N}_r\times C_3)\xrightarrow{d}\tilde{H}^{p+1,q}(\Sph_{2k}[2k+2]\vee EB)\rightarrow\cdots \]
with differential
\[d^{p,q}\colon H^{p,q}(\tilde{N}_r\times C_3)\rightarrow\tilde{H}^{p+1,q}(\Sph_{2k}[2k+2]\vee EB)\]
as shown below: 
 
\begin{center}
	\begin{tikzpicture}[scale=0.6]
		\draw[help lines] (-3.125,-5.125) grid (5.125, 5.125);
		\draw[<->] (-3,0)--(5,0)node[right]{$p$};
		\draw[<->] (0,-5)--(0,5)node[above]{$q$};;
		\cthree{0}{0}{red};
		\cthree{1}{0}{red};
		\lab{.75}{$r$}{red};
		\eb{1.2}{1}{blue};
		\lab{1.45}{$2k$}{blue};
		\scone{2}{1.25}{blue};
		\draw[very thick,->] (1.5,.75) -- (2.5,.75);
		\draw[very thick,->] (.5,-.75) -- (1.75,-.75);
	\end{tikzpicture}
\end{center}

To determine if this differential is nonzero, we start with the quotient lemma. Observe that $N_{4k+3r}[2k+2]/C_3\simeq N_r$, and we have
\[\tilde{H}^p_{\text{sing}}(N_r;\Z/3)=\begin{cases}
\Z/3 & \text{for }p=0 \\
(\Z/3)^{r-1} & \text{for }p=1 \\
0 & \text{else}
\end{cases}\]
Thus it must be the case that $d^{1,0}$ is nonzero. Otherwise $H^{2,0}(N_{4k+3r}[2k+2])\neq 0$ and we would contradict the results of the quotient lemma. By linearity, we get that $d^{1,q}$ is nonzero for all $q\leq 0$ and is zero when $q>0$.

We next turn to $d^{0,q}$. Using a similar argument from the base case of isomorphism class (1) in Theorem \ref{classificationversion}, we can observe that since $\M_3$ and $\text{coker}(d)$ are submodules of $H^{*,*}(N_{4k+3r}[2k+2])$, $d^{0,q}$ cannot be zero.

We are left to determine if the extension
\[\operatorname{coker}(d) \hookrightarrow H^{*,*}(N_{4k+3r}[2k+2])\twoheadrightarrow \ker d\]
is nontrivial, where $\operatorname{coker}(d)$ and $\ker (d)$ are depicted in Figure \ref{kernelandcokernel}. We already know that this extension is nontrivial since $\M_3\subseteq H^{*,*}(N_{4k+3r}[2k+2])$. As in the computation for Case (1) in Theorem \ref{classificationversion}, we get that $\left(\Sigma^{1,0}H^{*,*}(C_3)\right)^{\oplus r-1}\subseteq \ker d$ must split off as a summand of $H^{*,*}(N_{4k+3r}[2k+2])$ since no possible nontrivial extensions from this module can exist in this case. 

\begin{figure}
\begin{center}
	\begin{tikzpicture}[scale=0.6]
		\draw[help lines] (-3.125,-5.125) grid (5.125, 5.125);
		\draw[<->] (-3,0)--(5,0)node[right]{$p$};
		\draw[<->] (0,-5)--(0,5)node[above]{$q$};;
		\draw[thick,red] (0.5,0.5) -- (0.5,5);
		\draw[thick,red] (1.3,1.5) -- (1.3,5);
		\draw[thick, blue] (2+1/2,5) -- (2+1/2,1.125+1/2) -- (5,{(5-2)/2+1.125+1/4});
		\draw[thick, blue] (0+1/2,-5) -- (0 +1/2, 0 -0.5) -- (-3,{(-3-0)/2-3/4+0});
		\cthree{1}{0}{red};
		\eb{1.2}{1}{blue};
		\lab{1.2}{$2k$}{blue};
		\draw[red] (1+1/2,-5.5) node{\tiny{$r-1$}};
	\end{tikzpicture}
	\end{center}
	\caption{\label{kernelandcokernel} The modules $\ker(d)$ and $\coker(d)$.}
\end{figure}
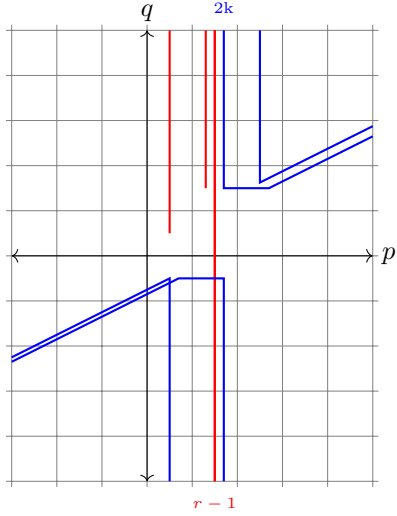

Finally, we conclude that 
\[H^{*,*}(N_{4k+3r}[2k+2])\cong \M_3\oplus \left(\Sigma^{1,0}H^{*,*}(C_3)\right)^{\oplus r-1}\oplus \ebc^{\oplus 2k}.\]
\end{proof}

We next explore Case (4) of Theorem \ref{classificationversion}. Recall that the cohomology of $N_1[1]$ was determined in Example \ref{N_1[1] cohomology}, so we only have left to consider $N_{1+4k+3r}[2k+1]$ when $k\geq0$ and $r\geq 0$. Our final computations will proceed as induction on the values of $k$ and $r$ respectively.  

\begin{proof}[Proof of Theorem \ref{classificationversion}, Case (4) inductive step on $k$, {$N_{1+4k}[2k+1]$}]
The space $N_{1+4k+3r}[2k+1]$ is non-orientable with $\beta=1+4k+3r$ and $F=2k+1$. In particular, $F-1=2k$ and $(\beta-2F+1)/3=r$. So our goal is to show  
\[N_{1+4k+3r}[2k+1]\cong\M_3\oplus\ebc^{\oplus 2k}\oplus \left(\Sigma^{1,0}H^{*,*}(C_3)\right)^{\oplus r}.\]

Recall that if $r=0$, then $N_{1+4k+3r}[2k+1]\cong N_1[1]+k[R_3]$. We start with a cofiber sequence
\[{S^1_{\text{free}}}_+ \hookrightarrow N_1[1]+k[R_3]_+\rightarrow \Sph_{2k}[2k+2]\]
with long exact sequence 
\[\rightarrow \tilde{H}^{p,q}(\Sph_{2k}[2k+2])\rightarrow H^{p,q}(N_{1+4k}[2k+1])\rightarrow H^{p,q}(S^1_{\text{free}})\xrightarrow{d} \tilde{H}^{p+1,q}(\Sph_{2k}[2k+2])\rightarrow \]
on cohomology. We once again try to determine the total differential $\bigoplus_{p,q}d^{p,q}$ which is highlighted on the left of Figure \ref{nonorientablediff2}. 

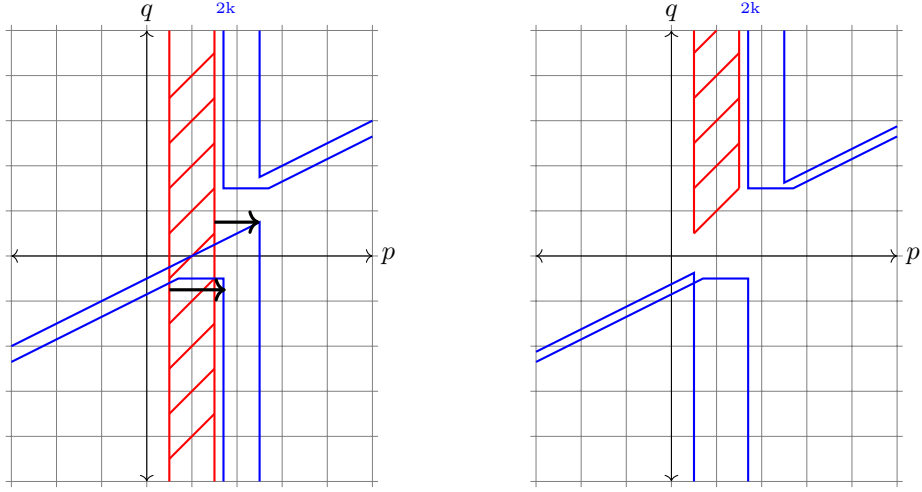
\begin{figure}
\begin{minipage}{0.45\textwidth}
\begin{center}
	\begin{tikzpicture}[scale=.6]
		\draw[help lines] (-3.125,-5.125) grid (5.125, 5.125);
		\draw[<->] (-3,0)--(5,0)node[right]{$p$};
		\draw[<->] (0,-5)--(0,5)node[above]{$q$};;
		\cthree{0}{1}{red};
		\eb{1.2}{1}{blue};
		\lab{1.25}{$2k$}{blue};
		\scone{2}{1.25}{blue};
		\draw[very thick,->] (1.5,.75) -- (2.5,.75);
		\draw[very thick,->] (.5,-.75) -- (1.75,-.75);
	\end{tikzpicture}
	\end{center}
\end{minipage}\ \begin{minipage}{0.45\textwidth}
\begin{center}
	\begin{tikzpicture}[scale=.6]
		\draw[help lines] (-3.125,-5.125) grid (5.125, 5.125);
		\draw[<->] (-3,0)--(5,0)node[right]{$p$};
		\draw[<->] (0,-5)--(0,5)node[above]{$q$};;
		\draw[thick,red] (0.5,0.5) -- (0.5,5);
		\draw[thick,red] (1.5,1.5) -- (1.5,5);
		\draw[thick,red] (0.5,0.5) -- (1.5,1.5);
		\draw[thick,red] (0.5,1.5) -- (1.5,2.5);
		\draw[thick,red] (0.5,2.5) -- (1.5,3.5);
		\draw[thick,red] (0.5,3.5) -- (1.5,4.5);
		\draw[thick,red] (0.5,4.5) -- (1,5);
		\draw[thick, blue] (2+1/2,5) -- (2+1/2,1.125+1/2) -- (5,{(5-2)/2+1.125+1/4});
		\draw[thick, blue] (1/2,-5) -- (1/2, .125 -0.5) -- (-3,{(-3)/2-3/4+.125});
		\eb{1.2}{1}{blue};
		\lab{1.25}{$2k$}{blue};
	\end{tikzpicture}
	\end{center}
\end{minipage}
\caption{\label{nonorientablediff2} The differential $d$ (left) and its kernel and cokernel (right).}
\end{figure}

As usual we start with the quotient lemma. Observe that $\left(N_{1}[1]+k[R_3]\right)/C_3\simeq \R P^2$, so it must be that $H^{p,0}(N_{1+4k}[2k+1])=0$ for $p\neq 0$. In particular, $d^{1,0}$ must be an isomorphism, and by linearity we can determine the behavior of the differential in all other degrees. In particular, $d^{p,q}$ is $0$ when $(p,q)=(0,0)$ or $q\geq 1$. Otherwise $d^{p,q}\neq 0$ with image in $\Sigma^{2,1}\M_3\subseteq \tilde{H}^{*,*}(\Sph_{2k}[2k+2])$. 

Now that we know the value of the differential, we can find $\ker(d)$ and $\coker(d)$. These modules are depicted on the right of Figure \ref{nonorientablediff2}. We are left to solve the extension problem
\[0\rightarrow \coker (d)\rightarrow H^{*,*}(N_{1+4k}[2k+1])\rightarrow\ker(d)\rightarrow 0.\]

Since $\M_3\subseteq H^{*,*}(N_{1+4k}[2k+1])$, we can immediately see that there must be a nontrivial extension. Knowing that $\M_3$ is a summand of $H^{*,*}(N_{1+4k}[2k+1])$, there is only one possible solution:
\[H^{*,*}(N_{1+4k}[2k+1])\cong \M_3\oplus \ebc^{\oplus 2k}.\]
\end{proof}

\begin{proof}[Proof of Theorem \ref{classificationversion}, Case (4) inductive step on $r$, {$N_{1+4k+3r}[2k+1]$}]
We finally turn to the general case with $r\geq 0$ and start by constructing the cofiber sequence
\[Y_+ \hookrightarrow N_{1+4k+3r}[2k+1]_+\rightarrow N_{1+4k}[2k+1]\]
where $Y$ is the space shown in red in Figure \ref{crosscapcofibseq}. Notice that $Y$ is homotopy equivalent to $\left(\bigvee_{r}S^{1,0}\right)\wedge {C_3}_+$.

From here we can examine the differential $d\colon H^{*,*}(Y)\rightarrow\tilde{H}^{*+1,*}(N_{1+4k}[2k+1])$ of the corresponding long exact sequence on cohomology. The left diagram of Figure \ref{nonorientablediff3} shows the $\M_3$-modules $H^{*,*}(Y)$ and $\tilde{H}^{*,*}(N_{1+4k}[2k+1])$. Since $\tilde{H}^{p+1,0}(N_{1+4k}[2k+1])=0$ for all $p$, we immediately see that $d^{p,0}$ must be $0$. By linearity, this guarantees the differential $d^{p,q}$ must be the zero map for all $(p,q)$. 

\begin{figure}
\begin{center}
\includegraphics[scale=.6]{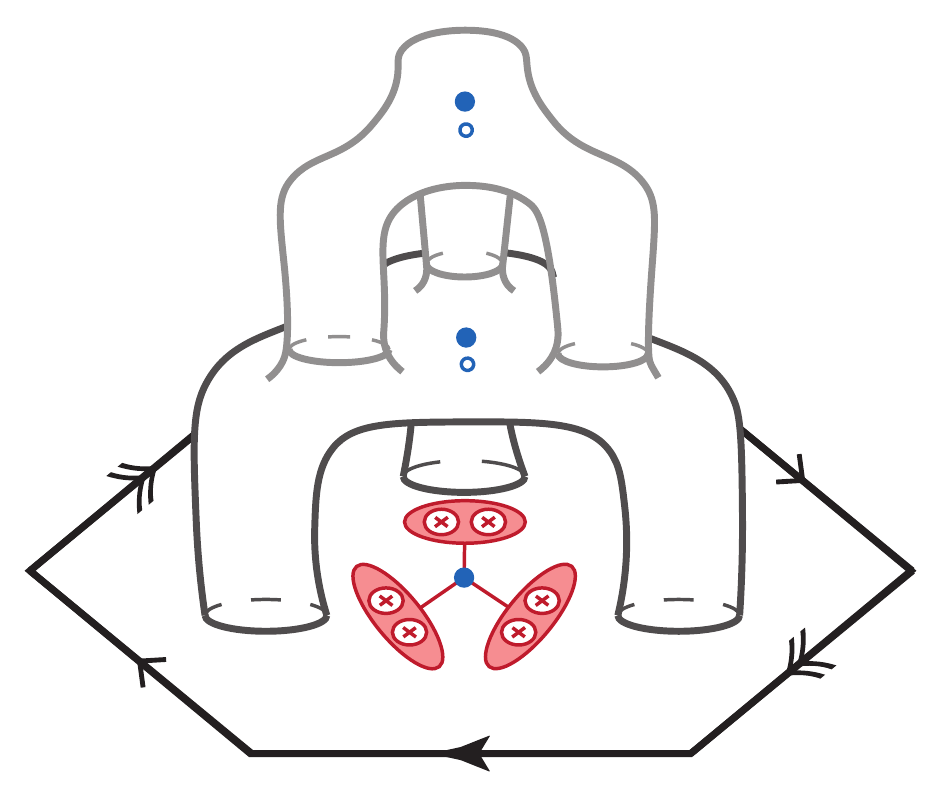}
\end{center}
\caption{\label{crosscapcofibseq} The space $Y$ in $N_{1+2k+3r}[2k+1]$ is shown in red in the case $r=k=2$.}
\end{figure}

\begin{figure}
\begin{center}
	\begin{tikzpicture}[scale=0.6]
		\draw[help lines] (-3.125,-5.125) grid (5.125, 5.125);
		\draw[<->] (-3,0)--(5,0)node[right]{$p$};
		\draw[<->] (0,-5)--(0,5)node[above]{$q$};;
		\scone{0}{0}{red};
		\cthree{1}{0}{red};
		\lab{.85}{$r$}{red};
		\eb{1.2}{1}{blue};
		\lab{1.55}{$2k$}{blue};
	\end{tikzpicture}
	\end{center}
	\caption{\label{nonorientablediff3} The spaces $H^{*,*}(Y)=\ker(d)$ and $\tilde{H}^{*,*}(N_{1+4k}[2k+1])=\coker(d)$.}
\end{figure}
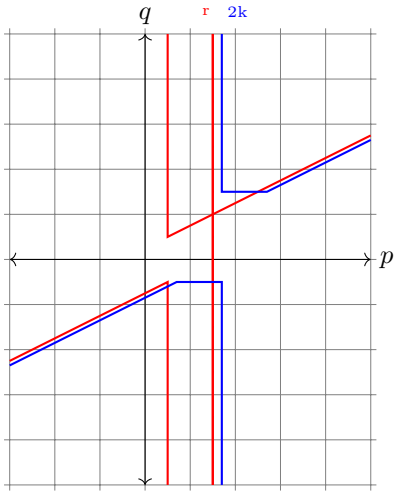

Since the total differential is $0$, $\ker(d)=H^{*,*}(Y)\cong \M_3\oplus \left(\Sigma^{1,0}H^{*,*}(C_3)\right)^{\oplus r}$ and $\coker(d)=\tilde{H}^{*,*}(N_{1+4k}[2k+1])\cong \ebc^{\oplus 2k}$. We now must solve the final extension problem
\[0\rightarrow \ebc^{\oplus 2k}\rightarrow H^{*,*}(N_{1+4k+3r}[2k+1])\rightarrow \M_3\oplus \left(\Sigma^{1,0}H^{*,*}(C_3)\right)^{\oplus r}\rightarrow 0.\]
We know that $\M_3$ must split off as a summand of $H^{*,*}(N_{1+4k+3r}[2k+1])$, and we have seen before in previous arguments that there can be no nontrivial extensions from $\left(\Sigma^{1,0}H^{*,*}(C_3)\right)^{\oplus r}$ to $\ebc$. This gives the desired result of
\[H^{*,*}(N_{1+4k+3r}[2k+1])\cong \M_3\oplus \left(\Sigma^{1,0}(H^{*,*}(C_3)\right)^{\oplus r} \oplus \ebc^{\oplus 2k}.\]
\end{proof}

\bibliographystyle{amsalpha}
\bibliography{main}

\end{document}